\documentclass[11 pt]{amsart}
\usepackage{amssymb,amsmath,amsthm,amsbsy,mathrsfs}
\usepackage[all]{xy}
\usepackage{color}
\usepackage[dvipsnames]{xcolor}
\usepackage{hyperref}

\makeatletter
\def\l@subsection{\@tocline{2}{0pt}{2.5pc}{5pc}{}}
\makeatother
\newcommand{\comment}[1]{}

\numberwithin{equation}{section}

\newtheorem{thm}{Theorem}[section]
\newtheorem{prop}[thm]{Proposition}
\newtheorem{lemma}[thm]{Lemma}
\newtheorem{cor}[thm]{Corollary}
\newtheorem*{thm*}{Theorem}

\theoremstyle{definition}
\newtheorem{defi}[thm]{Definition}%%it was: defi

\theoremstyle{remark}
\newtheorem{remark}[thm]{Remark}%it was:remark
\newtheorem{ex}[thm]{Example}

\newcommand{\la}[4]{
\xymatrix{#1 \ar[r] \ar@<2pt>[d] \ar@<-2pt>[d] & #2 \ar@<2pt>[d] \ar@<-2pt>[d] \\
#3 \ar[r] & #4}}

\newcommand{\pdiff}[2]{\frac{\partial #1}{\partial #2}}

\newcommand{\defequal}{:=}

\newcommand{\reals}{{\mathbb R}}

\newcommand{\vect}{\mathfrak{X}}
\DeclareMathOperator{\sym}{S}
\newcommand{\bigS}{\mbox{{\large{$\sym$}}}}

\renewcommand{\tilde}[1]{\widetilde{#1}}
\renewcommand{\hat}[1]{\widehat{#1}}

\newcommand{\st}[1]{{\scriptscriptstyle{#1}}}

\newcommand{\llbracket}{[\![}
\newcommand{\rrbracket}{]\!]}

 \DeclareMathOperator{\rank}{rank}

 \DeclareMathOperator{\End}{End}
 \DeclareMathOperator{\codim}{codim}
\DeclareMathOperator{\VB2}{VB2}
\DeclareMathOperator{\NMan}{2Man}
\DeclareMathOperator{\Span}{span}

\newcommand{\N}{\ensuremath{\mathbb N}}

\newcommand{\R}{\ensuremath{\mathbb R}}

\newcommand{\g}{\ensuremath{\mathfrak{g}}}
\newcommand{\A}{\ensuremath{\mathfrak{a}}}
\newcommand{\h}{\ensuremath{\mathfrak{h}}}

\newcommand{\cI}{\mathcal{I}}

\newcommand{\cL}{\mathcal{L}}

\newcommand{\cF}{\mathcal{F}}

\newcommand{\cM}{\mathcal{M}}
\newcommand{\cN}{\mathcal{N}}

\newcommand{\cT}{\mathcal{T}}
\newcommand{\cD}{\mathcal{D}}

\newcommand{\J}{\mathbb{J}}

\newcommand{\tdim}{\mbox{totdim}}

\newcommand{\fl}{\ensuremath{I}}

\newcommand{\Equot}{\ensuremath{E_{quot}}}
\newcommand{\Ered}{\ensuremath{E_{red}}}
\newcommand{\Mred}{\ensuremath{\underline{N}}}
\newcommand{\cMred}{\ensuremath{\underline{\cN}}}

\newcommand{\SP} [1]     {{\left\langle {{#1}} \right\rangle}}
\newcommand{\Cour}[1]      {[\![#1]\!]}
\newcommand{\tolabel}[1]{\stackrel{#1}{\to}}

\title[Graded geometry and generalized reduction]{Graded geometry and generalized reduction}

\author{H.~Bursztyn}
\address{Instituto de Matem\'atica Pura e Aplicada,
Estrada Dona Castorina 110, Rio de Janeiro, 22460-320, Brazil }
\email{henrique@impa.br}

\author{A.S.~Cattaneo}
\address{Institut f\"ur Mathematik,
Universit\"at Z\"urich-Irchel, Winterthurerstr. 190, CH-8057
Z\"urich, Switzerland} \email{cattaneo@math.uzh.ch}

\author{R.A.~Mehta}
\address{Department of Mathematical Sciences,
Smith College,
44 College Lane,
Northampton, MA 01063, U.S.A.}\email{rmehta@smith.edu}

 \author{M.~Zambon}
\address{KU Leuven, Department of Mathematics, Celestijnenlaan 200B box 2400, BE-3001 Leuven, Belgium.}
\email{marco.zambon@kuleuven.be}

\begin{document}
\begin{abstract}
 We present general reduction procedures for Courant, Dirac and generalized complex structures, in particular when a group of symmetries is acting. We do so by taking the graded symplectic viewpoint on Courant algebroids and carrying out graded symplectic reduction, both in the coisotropic and hamiltonian settings. Specializing the latter to the exact case, we recover in a systematic way the reduction schemes of Bursztyn-Cavalcanti-Gualtieri.
\end{abstract}

\maketitle
\setcounter{tocdepth}{2}%Hides subsubsections in TOC  
\tableofcontents

%%%%%%%%%%%%%%%%%%%%%%%%%%%%%%%%%%%%%%%%%%%%%%%%%%%%%%%%%%%%%%%%%%%%%
\section{Introduction}

% Courant algebroids

Courant algebroids, introduced in \cite{lwx}, have drawn much attention in recent years due to their many connections with (higher) geometry and mathematical physics
in such topics as Poisson--Lie theory, T-duality, sigma models, vertex algebras, shifted symplectic structures, see, e.g., \cite{Bressler07,TdualityCG,pymsaf,rogers,roytenberg:aksz,severa:TdualCA}.
A key feature of Courant algebroids is that they provide a unifying framework for various geometric structures, especially by means of their Dirac structures. 
In their original setting, Dirac structures in the standard Courant algebroid \cite{courant} were used to unify Poisson and presymplectic structures.
In the same vein, Courant algebroids are central ingredients in generalized complex geometry \cite{Gu2,Hitchin}, where a similar ``unification'' phenomenon allows  
complex and symplectic structures to be treated on equal footing.

% Courant reduction, related Dirac/GCS (mention related papers)

A reduction procedure for Courant algebroids in the presence of symmetries was developed in \cite{bcg,HKquot} as the foundational step for the reduction of
Dirac and generalized complex structures (for the latter see also  \cite{Hu,LT1,SX}). An important aspect of such Courant reduction is that symmetries were described by a new notion of 
``extended action'' by ``Courant algebras,'' rather than usual actions by Lie algebras or Lie groups. While this reduction scheme was motivated and guided by examples, the general procedure 
lacked a clear conceptual framework.

% motivation, place Courant reduction in a more general and natural framework...

The goal of the present work is to give a broad and systematic approach to the reduction of Courant algebroids and related geometric structures, and our main tool is 
the viewpoint to these objects in terms of {\em graded symplectic geometry}. Specifically, our starting point is the correspondence between Courant algebroids and degree 2
symplectic $\N$-manifolds  with a self-commuting degree $3$ function \cite{royt:graded,severa:sometitle}.

\bigskip

%\begin{table}[h!]
{\small{
\begin{center}
\begin{tabular}{ |c| c |}
\hline

pseudo-euclidean vector bundle $(E,\SP{\cdot,\cdot})$ 
& 
symplectic $\mathbb{N}$-manifold of degree 2  $(\cM,\omega)$ 
\\ 
\hline
 Courant structure $\rho$, $\Cour{\cdot,\cdot}$ & $\Theta\in C(\cM)_3$, $\{\Theta,\Theta\}=0$     \\  
 \hline 
\end{tabular}
\medskip
%\caption{include?} 
\end{center}
}}
%\end{table}

\medskip

Our strategy consists in expressing the graded analogues of usual reduction procedures in symplectic geometry in Courant-geometric  terms via this correspondence. This involves, as a key step, adding new lines to the above dictionary so as to include, on the graded symplectic side, usual ingredients 
of symplectic reduction such as coisotropic submanifolds, hamiltonian actions and momentum maps. Their classical geometric counterparts lead to natural ``coisotropic'' and  
``hamiltonian'' frameworks in which 
to carry out the reduction of Courant and related structures, and these are shown to include the constructions in \cite{bcg,HKquot,Za} 
as special cases.
 
\begin{remark} During the long writing process of this paper, some related topics were explored in separate works. 
In the context of $\mathbb N$-manifolds of degree 1, the 
analogous perspective of graded symplectic reduction leads to general reduction procedures for Poisson manifolds described in \cite{CaZa:graded,CaZa:supergeometric};
in \cite{Me:homotopy}, a more general framework for graded symplectic reduction is studied, based on actions by homotopy Poisson Lie groups; in the graded setting, coisotropic reduction relies on a version of the Frobenius theorem (on the integrability of involutive distributions) for $\mathbb{N}$-manifolds discussed in \cite{HenMiqRaj}.
\end{remark}

\subsection*{Statements of results and outline of the paper}
\subsubsection*{Geometrization of $\N$-manifolds of degree 2} We collect in \S \ref{sec:deg2man} foundational results about $\N$-manifolds of degree 2, starting with their basic
geometric description in terms of ordinary vector bundles.
Any $\N$-manifold $\cM$ of degree 2, with body $M$, is completely determined by a pair of vector bundles $E_1$ and $\tilde{E}$ over $M$ (whose duals codify functions of degrees 1 and 2) together with a vector-bundle map $\phi_E\colon\tilde{E} \twoheadrightarrow \wedge^2 E_1$ (whose dual accounts for the multiplication of degree 1 functions). Such triples $(E_1, \tilde{E}, \phi_E)$ can be naturally regarded as objects of a category, denoted by $\VB2$, which is shown to be equivalent to the category $\NMan$ of  $\N$-manifolds of degree 2,
\begin{equation}\label{eq:VB22Man}
\VB2 \; \rightleftharpoons \; \NMan,
\end{equation}
see Theorem~\ref{thm:equivcat}. We build on this equivalence to obtain classical geometric descriptions of submanifolds of $\N$-manifolds of degree 2 (Proposition~\ref{prop:subobjects} and Theorem~\ref{thm:submanifold}), as well as regular values of maps (Lemma~\ref{lem:regular}) and their inverse images (Proposition~\ref{prop:invc}).

In \S \ref{sec:symplectic} we focus on  \emph{symplectic}  $\N$-manifolds of degree $2$, and relate their known correspondence with pseudo-euclidean vector bundles \cite{royt:graded} (stated in Theorem \ref{thm:pseudoeuc})  with the equivalence in \eqref{eq:VB22Man}: any vector bundle carrying a pseudo-euclidean metric gives rise to an object in $\VB2$ by means of its Atiyah sequence; the corresponding $\N$-manifold of degree 2 carries a natural symplectic structure,
and any degree 2 symplectic $\N$-manifold can be identified with one of this type, see \eqref{eq:compdiag}.

\subsubsection*{Coisotropic reduction}
We consider coisotropic submanifolds of symplectic $\N$-mani\-folds of degree $2$ 
in \S\ref{subsec:coisosubm}, providing their geometric characterization
 in Theorem~\ref{thm:coisotropic}.
For a  pseudo-euclidean vector bundle $(E,\SP{\cdot,\cdot})$ over $M$, coisotropic submanifolds of the corresponding symplectic $\N$-manifold of degree $2$ are identified with quadruples $(N, K, F, \nabla)$ consisting of  a submanifold $N$ supporting an isotropic subbundle $(K\to N) \subset (E\to M)$ and an involutive subbundle $F\subseteq TN$, along with a flat metric $F$-connection $\nabla$ on the quotient vector bundle $(K^\perp/K)\to N$.
We refer to such quadruples $(N, K, F, \nabla)$ as  \emph{geometric coisotropic data} (Definition \ref{def:quadruples}). 
As a special case, we recover the correspondence between graded lagrangian submanifolds and  lagrangian subbundles $K\to N$ of $E$ (Corollary~\ref{cor:lag}) stated in \cite{severa:sometitle}. 
Geometric coisotropic data are the fundamental objects allowing constructions involving graded coisotropic submanifolds to be expressed in classical geometric terms.

In usual symplectic geometry, coisotropic submanifolds are often used to construct new symplectic manifolds via {\em reduction}: concretely, any coisotropic submanifold carries a ``null'' foliation (tangent to the kernel of the restriction of the symplectic form) whose leaf space, whenever smooth, is naturally equipped with a  symplectic structure. A more fundamental fact underlying this construction is that the {\em basic} functions (i.e., leafwise constant) of a coisotropic submanifold always form a {\em Poisson algebra}.  For coisotropic submanifolds of symplectic $\N$-manifolds of  degree $2$,  Proposition~\ref{prop:Cbas} gives a geometric description of their sheaves of basic functions, together with their Poisson brackets, in terms of the data $(N, K, F, \nabla)$.

In $\S$ \ref{sec:coisoreduction}, we consider the reduction
of a coisotropic submanifold $\cN$ of a   symplectic $\N$-manifold $\cM$ of  degree $2$, 
yielding (under suitable conditions) a new  symplectic $\N$-manifold $\cMred$ of degree $2$.
The classical geometric counterpart of this construction is explained in
Theorem \ref{thm:coisored}. The set-up is that of a pseudo-euclidean vector bundle $E$ with geometric coisotropic data 
$(N, K, F, \nabla)$ (corresponding to $\cM$ and $\cN$, respectively), with the additional requirements that $F$ defines a {\em simple foliation} on $N$ and that the $F$-connection $\nabla$ on $K^\perp/K$ has {\em trivial holonomy} (which are the conditions for $\underline{\cN}$ to exist). In this case, $N$ has a smooth leaf space $\underline{N}$ and $\nabla$ 
gives rise to a linear action on $K^\perp/K$ whose quotient is a 
 a pseudo-euclidean vector bundle $\Ered \to \Mred$ (see Lemma~\ref{lem:hol}),
$$
\xymatrix{ \frac{K^\perp}{K}  \ar[r]^{} \ar[d]_{} & \Ered \ar[d]^{} \\
N \ar[r]_{} &  \Mred.}
$$
It is shown in Theorem \ref{thm:coisored} that $\Ered$ is the pseudo-euclidean vector bundle
corresponding to the reduced symplectic $\N$-manifold of degree 2   $\cMred$.
 \bigskip

%\begin{table}[h!]
{\small{
\begin{center}
\begin{tabular}{ |c| c |}
\hline

pseudo-euclidean vector bundle $(E,\SP{\cdot,\cdot})$ 
& 
symplectic $\mathbb{N}$-manifold of degree 2 $(\cM,\omega)$
\\ 
\hline
$(N, K, F, \nabla)$ geometric coisotropic data  & $\cN \hookrightarrow \cM$  coisotropic submanifold  \\  
 \hline 
$\Ered\to \Mred$ 
& 
coisotropic reduction $\cMred$\\
\hline
\end{tabular}
\medskip
%\caption{include?} 
\end{center}
}}
%\end{table}

\smallskip

\begin{remark}\label{rem:constrank}
The classical geometric description of graded coisotropic reduction can be extended to more general graded submanifolds with null distributions of ``constant rank,'' but we restrict ourselves to the coisotropic setting for simplicity.
\end{remark}

Building on graded coisotropic reduction, we obtain in $\S$ \ref{sec:redCA} a
reduction procedure for Courant algebroids. As previously recalled, if $E$ is a pseudo-euclidean vector bundle with corresponding symplectic  $\N$-manifold of degree 2 $\cM$, a Courant algebroid structure on $E$ is equivalent to a {\em Courant function} on $\cM$, i.e., a degree 3 function $\Theta$ satisfying $\{\Theta,\Theta\}=0$. 
A function $S$ on $\cM$ is called {\em reducible} with respect to a coisotropic submanifold $\cN$ if it satisfies $\{S,\cI\}\subseteq \cI$, where $\cI$ is the vanishing ideal of $\cN$ (i.e., if the hamiltonian vector field of $S$ is tangent to $\cN$); in the context of graded coisotropic reduction
$$
\cN \to \cMred,
$$ 
reducibility ensures that $S$ descends to a function $S_{red}$ on $\cMred$. 
Courant reduction is based on the fact that whenever a Courant function $\Theta$ on $\cM$ is reducible, $\Theta_{red}$ is a Courant function of $\cMred$, which in turn corresponds to a Courant algebroid structure on the pseudo-euclidean vector bundle $\Ered\to \Mred$.

For a Courant algebroid structure on $E$ with anchor $\rho$ and bracket $\Cour{\cdot,\cdot}$, we show in Theorem~\ref{thm:reducible} that the 
reducibility of the corresponding Courant function with respect to a coisotropic submanifold defined by $(N, K, F, \nabla)$ amounts to the following conditions:
{\small
\begin{equation}\label{eq:reducible}
\rho(K^\perp)\subseteq TN, \quad \rho(K)\subseteq F, \quad [\rho(\Gamma^\st{flat}_\st{E,K^\perp}),\Gamma_\st{TM,F}]\subseteq \Gamma_\st{TM,F}, \quad 
\Cour{\Gamma^\st{flat}_\st{E,K^\perp},\Gamma^\st{flat}_\st{E,K^\perp}}
\subseteq \Gamma^\st{flat}_\st{E,K^\perp},
\end{equation}}
\hspace{-1.5mm}where  $\Gamma^\st{flat}_\st{E,K^\perp}$ denotes the sheaf of sections of $E$ whose restriction to $N$ lie in $K^\perp$ and project to a $\nabla$-flat section of $K^\perp/K$, and
$\Gamma_\st{TM,F}$ is the sheaf of vector fields on $M$ whose restriction to $N$ lie in $F$. In the case of a graded lagrangian submanifold, reducibility amounts to $K$ being a {\em Dirac structure} in $E$ supported on $N$ (see Def.~\ref{def:diracsupp}).

The reduction of Courant algebroids corresponding to graded coisotropic reduction of Courant functions is presented in Theorem~\ref{cor:courred}, and summarized below.

\begin{thm*}[Coisotropic Courant reduction]
  Let $E\to M$ be a Courant algebroid equipped with geometric coisotropic data $(N, K, F, \nabla)$ such that $F$ is simple and $\nabla$ has trivial holonomy. If the reducibility conditions \eqref{eq:reducible} hold,
then the pseudo-euclidean vector bundle $\Ered\to \Mred$ inherits a natural Courant algebroid structure.  \end{thm*}

This result extends reductions in \cite{LBMeinrenken,Za}, see Example~\ref{ex:casescoisored}(ii) and Example~\ref{ex:JSGK}.

In $\S$ \ref{sec:redGCDirac}, we use coisotropic Courant reduction as the basis for reduction schemes for {\em generalized complex}  and {\em Dirac} structures:
\begin{itemize}
\item Generalized complex structures can be viewed as quadratic functions on symplectic  $\N$-manifolds of degree 2 suitably compatible with a Courant function, see  \cite{grabowski}; the geometric characterization of their reducibility (Lemma~\ref{lem:Jreducible}) leads to a reduction procedure for generalized complex structures with respect to geometric coisotropic data, see Theorem~\ref{thm:redGCS}.

\item Dirac structures are lagrangian submanifolds of symplectic $\N$-mani\-folds of degree 2 for which the Courant function is reducible, see Cor.~\ref{prop:lagdirac}. 
Just as in usual symplectic geometry, graded lagrangian submanifolds can be reduced 
to coisotropic quotients upon a clean intersection condition (Theorem~\ref{thm:lagred}). By means of a geometric characterization of such clean intersections in the graded setting (Prop.~\ref{lem:clean}) we obtain a reduction procedure for Dirac structures with respect to geometric coisotropic data in Theorem~\ref{thm:diracred}.
\end{itemize}

In the process of expressing reduction constructions from graded geometry in classical geometric terms, a recurrent issue is that of relating the Atiyah algebroid of a given (pseudo-euclidean) vector bundle $A$ with the Atiyah algebroid of a quotient of $A$.  Appendix \ref{app:A} addresses this issue and presents the results needed in the paper.

\subsubsection*{Hamiltonian reduction}
In $\S$ \ref{sec:momap}, we consider graded symplectic reduction in the hamiltonian setting; its formulation in classical terms leads to a hamiltonian version of Courant reduction that recovers constructions in \cite{bcg,HKquot}. 

In ordinary symplectic geometry, a hamiltonian action of a Lie algebra $\g$ on a symplectic manifold $M$ is given by a Lie algebra map
$\mu^*\colon\g \to C^\infty(M)$,
known as a {\em comomentum map} (the dual map $\mu\colon M\to \g^*$ is the {\em momentum map}).
In the context of $\N$-manifolds of degree 2, we consider a graded Lie algebra $\tilde{\g}$ concentrated in degrees $-2, -1$ and $0$, 
$$
\tilde{\g}=\h[2]\oplus \A[1]\oplus \g,
$$
with a hamiltonian action on a symplectic $\N$-manifold of degree 2 $\cM$ given by a comomentum map
\begin{equation}\label{eq:comom}
\tilde{\mu}^\sharp\colon \tilde{\g}\to C(\cM)[2],
\end{equation}
see  $\S$ \ref{sec:hamred}.
With this setup one may construct new symplectic  $\N$-manifolds of degree 2 by means of 
a graded version of Marsden--Weinstein reduction. For simplicity, we will focus on reduction at momentum level zero, which can be regarded as an instance of coisotropic reduction.
Whenever $\cM$ is equipped with a Courant function $\Theta$, $C(\cM)[2]$ becomes a differential graded Lie algebra (DGLA)
with differential $\{\Theta,\cdot\}$; in this case  we also assume that $\tilde{\g}$ is a DGLA and that the comomentum map is a DGLA morphism. This ensures that $\Theta$ is reducible with respect to the coisotropic submanifold defined by the zero level of the momentum map, and hence defines a reduced Courant function $\Theta_{red}$ on the Marsden--Weinstein quotient $\cM_{red}$. 

For the classical geometric description of graded hamiltonian actions and reduction, we first notice that a DGLA $\tilde{\g}$ as above is equivalent to an ordinary Lie algebra $\g$, together with $\g$-modules $\A$ and $\h$, a symmetric bilinear map $\varpi\colon \mathfrak{a} \otimes \mathfrak{a} \to \h$ and operators
$\h \stackrel{\delta}{\to} \mathfrak{a} \stackrel{\delta}{\to} \g$ satisfying the conditions described in Propositions~\ref{prop:gla} and \ref{prop:dgla} (see ${\bf (B)}$ in $\S$ \ref{subsec:hamCAred}). The notion of a {\em hamiltonian action} on a Courant algebroid $E\to M$ in Definition~\ref{def:infHact} arises as the classical geometric counterpart of a comomentum map \eqref{eq:comom}; it is given in terms of $(\g,\A,\h, \varpi, \delta)$  and  consists  of 
a linear $\g$-action on $E$ and $\g$-equivariant maps
$$
\mu\colon M\to \h^*, \qquad \varrho\colon \A \to \Gamma(E)
$$
with suitable compatibility conditions (see ${\bf (C)}$ in $\S$ \ref{subsec:hamCAred}). Reduction procedures for Courant, Dirac and GC structures in this hamiltonian setting are presented in Theorem~\ref{thm:geomhamred}, all based on graded Marsden--Weinstein reduction. (Since the zero level of the graded momentum map is coisotropic, these procedures rely on coisotropic Courant reduction.)

\bigskip

{\small{
\begin{center}
\begin{tabular}{ |c| c |}
\hline

Hamiltonian action on Courant algebroid $E$
& 
Hamiltonian action on $\cM$, $\Theta \in C(\cM)_3$
\\ 
\hline
$\g$-action on $E$, $\mu\colon M\to \h^*$, $\varrho\colon \A\to \Gamma(E)$  & comomentum map $\tilde{\g}\to C(\cM)[2]$ \\  
 \hline 
Courant reduction $\Ered\to \mu^{-1}(0)/G$ 
& 
Marsden--Weinstein reduction $\cM_{red}$, $\Theta_{red}$\\
\hline
\end{tabular}
\medskip
%\caption{include?} 
\end{center}
}}

\medskip

We make contact with \cite{bcg,HKquot} by restricting our attention to hamiltonian actions of {\em exact} DGLAs on {\em exact} Courant algebroids. We show in Prop.~\ref{lem:exact} that exact DGLAs  bijectively correspond to the (exact) Courant algebras introduced in \cite{bcg} (up to a minor technical difference with no noticeable effect). Further we prove in Prop.~\ref{prop:bcg} (and Remark~\ref{rem:bcg}) that their hamiltonian actions on exact Courant algebroids are equivalent to the ``extended actions with momentum maps,'' used in the reduction schemes in \cite{bcg,HKquot}, of the corresponding Courant algebras.

\bigskip

{\small{
\begin{center}
\begin{tabular}{ |c| c |}
\hline
Exact DGLA $\tilde{\g}=\h[2]\oplus \A[1]\oplus \g$
& 
Exact Courant algebra $\A \to \g$
\\ 
\hline
Hamiltonian $\tilde{\g}$-action on (exact) $E\to M$   & Extended action of $\A\to \g$ on (exact) $E\to M$ \\  
 \hline 
%Courant reduction $\Ered\to \mu^{-1}(0)/G$ 
%& 
%Marsden-Weinstein reduction $\cM_{red}$, $\Theta_{red}$\\
%\hline
\end{tabular}
\medskip
%\caption{include?} 
\end{center}
}}

\medskip

This shows that our hamiltonian framework for Courant reduction, naturally derived from standard constructions in graded symplectic geometry, provides  a systematic approach to the reductions of Courant, Dirac and generalized complex structures, which  recovers the results in \cite{bcg,HKquot} when specialized to the exact case.

\subsection*{Notation and conventions}\label{subsec:notation}

For a $\mathbb{Z}$-graded vector space $V$, we define its degree shift $V[k]$ by $(V[k])_i=V_{k+i}$.

We will work in the category of smooth manifolds, usually denoted by $M$ and $N$.

For a sheaf of rings $\mathcal{A}$ over $M$, we denote its evaluation on
open subsets $U$ by $\mathcal{A}(U)$, its restriction by $\mathcal{A}|_U$, and its stalk at
$x$ by $\mathcal{A}|_x$.

For a manifold $M$, we denote by $C^ \infty_M$ its sheaf of algebras of (real-valued) smooth functions, and $C^\infty(M)=C^\infty_M(M)$. For a submanifold $N\subseteq M$, $I_N$ denotes its sheaf of vanishing ideals. 

For a smooth vector bundle $E\to M$, the  sheaf of sections is denoted by $\Gamma_E$, and $\Gamma(E)=\Gamma_E(M)$. Vector bundle maps cover the identity map unless stated otherwise.
Given a subbundle $K\to N$ of $E\to M$, we write $\Gamma_{\st{E,K}}$ for the subsheaf of sections of $E$ whose restriction to $N$ lie in $K$.
Derivations of $E\to M$ are denoted by $(X,D)$, with $D: \Gamma(E)\to \Gamma(E)$ and $X\in \mathfrak{X}(M)=\Gamma(TM)$ its symbol.

For a pseudo-euclidean vector bundle $E\to M$ (i.e., $E$ is equipped with a fiberwise pseudo-Riemannian metric), we denote by $\mathbb{A}_E$ its Atiyah algebroid (whose sections are derivations of $E$ that are compatible with the metric). 
We write $\Gamma^{\st{N}}_{\mathbb{A}_E}$ for the subsheaf of sections of ${\mathbb{A}_E}$ whose symbols are tangent to a submanifold $N$. For a subbundle $K\to N$ of $E\to M$, we denote by $\Gamma^{\st{N,K}}_{\mathbb{A}_E}$ the subsheaf of $\Gamma_{\mathbb{A}_E}$ whose sections  $(X,D)$ satisfy $X|_N\in TN$ and $D(\Gamma_{\st{E,K}})\subseteq \Gamma_{\st{E,K}}$. When $N=M$, we simplify notation to $\Gamma^{\st{K}}_{\mathbb{A}_E}$ for the sheaf of derivations 
$(X,D)$ with $D(\Gamma_K)\subseteq \Gamma_K$. We use a similar notation $\Gamma^{\st{K,L}}_{\mathbb{A}_E}$ when we have an additional subbundle $L\to M$ of $E\to M$, and impose the conditions $D(\Gamma_K)\subseteq \Gamma_K$ and $D(\Gamma_L)\subseteq \Gamma_L$.

We will make use of Einstein's convention for sums whenever there is no risk of confusion.

\noindent{\bf Acknowledgements:} H.B. acknowledges the continual financial support of CAPES, CNPq and FAPERJ.
A.S.C. acknowledges partial support of SNF Grant No. 200020-192080, of the Simons Collaboration on Global Categorical Symmetries, and of the COST Action 21109 (CaLISTA).  A.S.C.'s research was (partly) supported by the NCCR SwissMAP, funded by the Swiss National Science Foundation. R.M. acknowledges partial support from CNPq.
M.Z.~acknowledges partial support by Forschungskredit U. Z\"urich, by  Ciencia 2007 (Portugal), SB2006-0141, MICINN RYC-2009-04065
%MTM2011-22612 
and Severo Ochoa  SEV-2011-0087 (Spain), 
CNPq-PVE grant \\
 88881.030367/2013-01 (Brazil), 
 Methusalem grants METH/15/026 and  METH/21/03 - long term structural funding of the Flemish Government,  the FWO research projects G083118N, G0B3523N and G014726N, the FWO and FNRS under EOS projects G0H4518N and G0I2222N (Belgium). We are thankful to several institutions for hosting us during various stages of this project, including  Erwin Schr\"odinger Institut, IMPA, KU Leuven and U. Z\"urich. We thank A. Cabrera, M. Cueca, F. del Carpio-Marek, P. H. Carvalho Silva, and F. Valach  for fruitful discussions related to this work. We are grateful to the anonymous referees for their many valuable comments and suggestions.

%%%%%%%%%%%%%%%%%%%%%%%%%%%%%%%%%%%%%%%%%%%%%%%%%%%%%%%%%%%%%%%%%%%%%
\section{$\N$-manifolds of degree 2}\label{sec:deg2man}

In this section we present foundational results about  $\N$-manifolds in degree 2. We will consider $\N$-manifolds as in \cite{BonavolontaPoncin,rajthesis} (see also \cite{royt:graded,severa:sometitle}). Some parallel results in the theory of supermanifolds can be found, e.g., in  \cite{fiori,leites80}.

\subsection{$\N$-manifolds of degree 2 and their morphisms}
An \textit{$\N$-manifold of degree 2} $\cM$ of dimension $m_0|m_1|m_2$
is a pair $(M, C_{\cM})$, where $M$ (the \emph{body} of $\cM$) is a
dimension $m_0$ manifold and $C_{\cM}$ is a  sheaf of graded-commutative
algebras on $M$ satisfying the following local
property: any $x\in M$ admits an open neighborhood $U\subseteq M$
such that
\begin{equation}\label{eq:localtriv}
C_\cM|_U \cong C^\infty_U\otimes \wedge \reals^{m_1} \otimes \bigS
\reals^{m_2},
\end{equation}
where the  elements of $\reals^{m_q}$ are of degree $q$; the right-hand side of \eqref{eq:localtriv} denotes the sheafification of the pre-sheaf $V\mapsto C^\infty(V)\otimes \wedge \reals^{m_1} \otimes \bigS
\reals^{m_2}$, for  $V\subseteq U$ open. 
Denoting by $(C_{\cM})_q$ the subsheaf  of degree $q$ sections of $C_{\cM}$, by the local condition \eqref{eq:localtriv} we have that
$$
(C_{\cM})_q(U) = \oplus_{l+2k=q} C^\infty(U) \otimes \wedge^l \reals^{m_1} \otimes \bigS^k
\reals^{m_2},
$$
so there exist sections $\{x^i, e^\mu, p^I\}$ of $C_{\cM}$ over
$U$, where $\{x^i\}$ are usual coordinates on $U$, the elements
$\{e^\mu\}$ and $\{p^I\}$ are respectively of degree $1$ and $2$,
such that every homogeneous section of $C_{\cM}$ over $U$ can be expressed as a sum of
functions that are smooth in $x^i$ and polynomial in $e^\mu$ and
$p^I$. We refer to the local sections
$$
\{x^1,\ldots,x^{m_0}, e^1,\ldots,e^{m_1},p^1,\ldots, p^{m_2}\}
$$
as \textit{local coordinates} on $\cM$, 
and we think of $C_{\cM}$ as
the sheaf of ``smooth functions'' on $\cM$.

Note  that $(C_{\cM})_0$ is canonically isomorphic to
$C^\infty_M$. We use the notation $ \dim(\cM) = m_0|m_1|m_2$ and refer to $ \mathrm{totdim}(\cM) = m_0 + m_1 + m_2$ as the {\em total dimension}.

\begin{remark}\label{rem:lowerdegs}
$\mathbb{N}$-manifolds of degree 2 with dimensions of special type $m_0 | 0 | 0$ are naturally identified with ordinary manifolds of dimension $m_0$. Those of dimension type $m_0 | m_1 | 0$ are referred to as {\em $\mathbb{N}$-manifolds of degree 1}. \hfill $\diamond$
\end{remark}

\begin{remark}
Another approach to $\N$-manifolds regards them as supermanifolds (i.e., $\mathbb{Z}_2$-graded) carrying an action of the multiplicative monoid of real numbers in such a way that the degrees of the coordinates correspond to the weight of the action \cite{severa:sometitle,royt:graded}; see \cite[Sec.~4]{voronov}.
\end{remark}

If $\cM$ and $\cN$ are  $\N$-manifolds of degree 2, a \textit{morphism}
$\Psi \colon  \cN\to \cM$ is a pair $(\psi,\psi^\sharp)$, where $\psi\colon N \to
M$ is a smooth map and
$$
\psi^\sharp \colon  C_{\cM}\to \psi_* C_{\cN}
$$
is a morphism of sheaves over $M$; in particular, for each open
subset $U\subseteq M$, $\psi^\sharp_U \colon  C_{\cM}(U)\to
C_{\cN}(\psi^{-1}(U))$ is a morphism of graded algebras. 
 $\N$-manifolds of degree 2 and their morphisms define a category denoted by $\NMan$.

A standard
argument (see, e.g., \cite[Sec.~4.1]{varadarajan}) shows that the
restriction of $\psi^\sharp$ to degree zero functions agrees with
the pullback map
$$
\psi^* \colon C_M^\infty \to \psi_*C_N^\infty.
$$
An \textit{isomorphism} $\Psi \colon \cN\to \cM$ is a morphism that admits
an inverse; equivalently, $\psi$ is a diffeomorphism and
$\psi^\sharp$ induces a bijection of stalks at each $x \in M$.

Let $U\subset \mathbb{R}^{m_0}$ be an open set, and denote by
$\mathcal{U}$ the $\N$-manifold of degree 2 with body $U$ and structure
sheaf $C^\infty_U\otimes \wedge \mathbb{R}^{m_1}\otimes \bigS
\mathbb{R}^{m_2}$, described by coordinates $\{x^i,e^\mu,p^I\}$. If
$\cN$ is any $\N$-manifold of degree 2, then it is a simple matter to
check that any morphism $\cN \to \mathcal{U}$ is completely
determined by the choice of a map $\psi:N\to U$ as well as elements
$f^\mu, g^I \in C_{\cN}(N)$, of degrees 1 and 2, respectively;
indeed, the conditions
$$
\psi^\sharp(x^i)= x^i\circ \psi,\;\; \psi^\sharp(e^\mu)=f^\mu,\;\;
\psi^\sharp(p^I)=g^I,
$$
uniquely determine a morphism of sheaves $\psi^\sharp:
C_{\mathcal{U}}\to \psi_*C_{\cN}$.

%%%%%%%%%%%%%%%%%%%%%%%%%%%%%%%%%%%%%%%%%%%%%%%%%%%%%%%%%%%%%
\subsection{Vector-bundle description}\label{subsec:decomp}
As we explain in this section, $\N$-manifolds of degree 2 can be
completely described in terms of classical vector bundles. We will formulate this fact in Theorem~\ref{thm:equivcat} as an equivalence between a suitable category defined by vector-bundle data and the category of $\N$-manifolds of degree 2.
(A generalization of this result for $\N$-manifolds of higher degrees can be found in \cite{HenMiqRaj}.)

\begin{defi}
We denote by $\VB2$ the category whose objects are triples $(E_1,\widetilde{E},\phi_E)$, where $E_1\to M$ and $\widetilde{E}\to M$ are vector bundles and $\phi_E: \widetilde{E}\to \wedge^2E_1$ is a surjective vector-bundle map covering the identity on $M$, and morphisms $(E_1,\tilde{E},\phi_E)\to
(F_1,\tilde{F},\phi_F)$ are given by pairs $(\psi_1,\widetilde{\psi})$ of vector-bundle maps $\psi_1 \colon E_1\to
F_1$, $\widetilde{\psi} \colon \tilde{E}\to \tilde{F}$, covering a smooth map $\psi:
M\to N$, such that
\begin{equation}\label{eq:compat}
\wedge^2 \psi_1 \circ \phi_E=\phi_F\circ \widetilde{\psi}.
\end{equation}
\hfill$\diamond$
\end{defi}

Given an object $(E_1,\tilde{E},\phi_E)$ in $\VB2$, 
we set $E_2=\mathrm{ker}(\phi_E)$, so that we have a short exact sequence
\begin{equation}\label{deg2exact}
    0\longrightarrow E_2 {\longrightarrow} \tilde{E} \stackrel{\phi_E}{\longrightarrow} \wedge^2 E_1 \longrightarrow 0.
\end{equation}
We will use the injective map 
$$
\iota_E:= \phi_E^* \colon \wedge^2 E_1^* \to \widetilde{E}^*
$$ 
to view  $\wedge^2 E_1^*$ as a subbundle of $\widetilde{E}^*$, in such a way that $E_2^* = \widetilde{E}^*/\wedge^2 E_1^*$. 

We will see how to associate an $\N$-manifold of degree 2 to an object $(E_1,\widetilde{E},\phi_E)$. Consider the sheaves of graded algebras 
$\Gamma_{\wedge^\bullet E_1^*}$ and $\Gamma_{\mathrm{S}^\bullet\tilde{E}^*} =  \oplus_{l=0}^\infty \Gamma_{\mathrm{S}^l\tilde{E}^*}$, and define the sheaf of homogeneous ideals $I$ in $\Gamma_{\wedge^\bullet E_1^*}\otimes
\Gamma_{\mathrm{S}^\bullet\tilde{E}^*}$ generated by 
$$
1\otimes
\iota_E(T)- T\otimes 1,
$$ 
with $T$ a local section of  $\Gamma_{\wedge^2 E_1^*}$.
\begin{lemma}\label{lem:functor}\
\begin{itemize}
\item For an object $(E_1,\widetilde{E},\phi_E)$ in $\VB2$, the sheaf of graded algebras over $M$ given by
\begin{equation}\label{eq:CM}
C_\cM := (\Gamma_{\wedge^\bullet E_1^*}\otimes
\Gamma_{\mathrm{S}^\bullet\tilde{E}^*})/I.
\end{equation}
defines an $\N$-manifold of degree 2. 
\item If $\cM$ and $\cN$ are the $\N$-manifolds of degree 2 corresponding to the objects
$(E_1,\tilde{E},\phi_E)$ and $(F_1,\tilde{F},\phi_F)$, then any morphism 
$(E_1,\tilde{E},\phi_E)\to
(F_1,\tilde{F},\phi_F)$ uniquely 
 determines a morphism
$\cM \to \cN$. 
\end{itemize}
\end{lemma}
\begin{proof}
The choice of a splitting of \eqref{deg2exact} induces an isomorphism $\tilde{E}^*\cong   \wedge^2 E^*_1 \oplus E_2^*$ and identifies $C_\cM$ in \eqref{eq:CM} with $\Gamma_{\wedge^\bullet E_1^*}\otimes
\Gamma_{\mathrm{S}^\bullet E_2^*}$. Then local frames of $E_1^*$ and $E_2^*$ over an open subset $U\subseteq M$ establish an isomorphism as in \eqref{eq:localtriv}, showing that $\cM = (M,C_\cM)$ is an $\N$-manifold of degree 2.

Consider a morphism $(\psi_1, \tilde{\psi})$ from $(E_1,\tilde{E},\phi_E)$ to 
$(F_1,\tilde{F},\phi_F)$
covering a smooth map $\psi:
M\to N$. The vector-bundle maps $\psi^*F_1^*\to E_1^*$ and $\psi^* \tilde{F}^*\to \tilde{E}^*$, dual to  $\psi_1 \colon E_1\to F_1$, and $\widetilde{\psi} \colon \tilde{E}\to \tilde{F}$, define morphisms (of sheaves of $C^\infty_M$-modules)
$\psi^\sharp_1  \colon \psi^*\Gamma_{F_1^*}\to \Gamma_{E_1^*}$ and $\psi^\sharp_2  \colon \psi^*\Gamma_{\tilde{F}^*}\to \Gamma_{\tilde{E}^*}$,
which are equivalent  (by
adjunction) to morphisms 
$$
\psi^\sharp_1 \colon \Gamma_{F_1^*}\to
\psi_*\Gamma_{E_1^*}, \qquad \psi^\sharp_2 \colon \Gamma_{\tilde{F}^*}\to
\psi_*\Gamma_{\tilde{E}^*}.
$$ 

Let $\cM$ and $\cN$ be the $\N$-manifolds of degree 2 defined by the objects
$(E_1,\tilde{E},\phi_E)$ and $(F_1,\tilde{F},\phi_F)$
The morphisms $\psi_1^\sharp$ and $\psi_2^\sharp$ naturally
extend to morphisms $\Gamma_{\wedge F_1^*}\to \psi_*C_\cM$ and
$\Gamma_{\mathrm{S}\tilde{F}^*}\to \psi_*C_\cM$, leading to a
morphism $\Gamma_{\wedge F_1^*}\otimes
\Gamma_{\mathrm{S}^{\bullet}\tilde{F}^*}\to \psi_*C_\cM$. The compatibility
condition \eqref{eq:compat} between $\psi_1^\sharp$ and $\psi_2^\sharp$ 
guarantees that this morphism descends to a morphism $\psi^\sharp: C_\cN \to \psi_*C_\cM$.
%\nmt{Here $\cM$ denotes the $\N$-manifold of degree 2 defined by $(E_1,\tilde{E},\phi_E)$, and $\cN$ the one defined by $(F_1,\tilde{F},\phi_F)$.}
\end{proof}

Lemma \ref{lem:functor} defines a functor from $\VB2$ to the category of $\N$-manifolds of degree 2, denoted by
\begin{equation}\label{eq:F}
\mathcal{F}\colon \VB2 \to \NMan.
\end{equation} 
Note that, for $\cM = \mathcal{F}(E_1, \widetilde{E}, \phi_E)$, we have 
\begin{equation}\label{eq:vb}
(C_\cM)_1=\Gamma_{E_1^*}, \;\mbox{ and }\;
(C_\cM)_2=\Gamma_{\tilde{E}^*},
\end{equation}
and hence local coordinates for $\cM$ are given by 
local frames of $E_1^*$ and $E_2^*$ over some open subset of $M$.

\begin{thm}\label{thm:equivcat}
The functor $\mathcal{F}\colon \VB2 \to \NMan$ is an equivalence of categories. 
\end{thm}

\begin{proof}
Let  $\cM$ be an $\N$-manifold of degree 2.  By \eqref{eq:localtriv},
the subsheaf $(C_{\cM})_1$
of degree $1$ functions is a locally free sheaf of
$C^\infty_M$-modules, so it is isomorphic to the sheaf of sections $\Gamma_{E_1^*}$
of the dual of some vector bundle $E_1 \to M$.
In coordinates $\{x^i, e^\mu,
p^I\}$ on $\cM$, we may take $\{e^\mu\}$ as a local frame for $E_1^*$.
Similarly, $(C_{\cM})_2$ can be identified with
$\Gamma_{\tilde{E}^*}$ for some vector bundle $\tilde{E} \to M$, and
a local frame for $\tilde{E}^*$ is given by $\{p^I, e^\mu e^\nu :
\mu < \nu \}$. Since $(C_\cM)_1 . (C_\cM)_1\subseteq (C_\cM)_2$,
there is a natural inclusion map
\begin{equation}\label{eq:iE}
\iota_E \colon \wedge^2 E_1^* \hookrightarrow \tilde{E}^*,
\end{equation}
whose dual map $\phi_E:=\iota_E^* \colon \tilde{E}\twoheadrightarrow \wedge^2
E_1$ is surjective. Note that $\cM \cong \mathcal{F}(E_1,\tilde{E},\phi_E)$, showing that $\mathcal{F}$ is essentially surjective.

For $\N$-manifolds $\cM = \mathcal{F}(E_1,\tilde{E},\phi_E)$ and $\cN = \mathcal{F}(F_1,\tilde{F},\phi_F)$, a morphism $(\psi,\psi^\sharp): \cM \to \cN$ defines, by restriction,
morphisms of sheaves of $C_N^\infty$-modules $\psi_i^\sharp \colon (C_\cN)_i\to \psi_*(C_\cM)_i$, $i=1,2$, such that
$\psi_2^{\sharp}
(f f')=
\psi_1^{\sharp}(f)
\psi_1^{\sharp}(f')$ for any local sections $f$, $f'$ of $(C_\cN)_1$. With the identifications \eqref{eq:vb}, the morphisms $\psi^\sharp_1$ and $\psi^\sharp_2$ correspond, through duality, to vector-bundle maps $\psi_1: E_1\to F_1$ and $\tilde{\psi}: \tilde{E}\to \tilde{F}$ satisfying \eqref{eq:compat}, in such a way that $(\psi,\psi^\sharp)=\mathcal{F}(\psi_1,\tilde{\psi})$.
This construction establishes a bijection between the set of morphisms from $(E_1,\tilde{E},\phi_E)$ to $(F_1,\tilde{F},\phi_F)$ and morphisms $\cM\to \cN$, which shows that $\mathcal{F}$ is fully faithful. 
\end{proof}

See  \cite[Section~6]{MR3689898} for a different perspective on Theorem~\ref{thm:equivcat}.

\begin{remark}
 $\mathbb N$-manifolds of degree $2$ admit an alternative geometric description as special types of double vector bundles; see \cite{delCarpioThesis,Jotz_Lean_2018,Li-BlandThesis}. 
Following \cite{delCarpioThesis}, the connection with Theorem~\ref{thm:equivcat} can be understood via the short exact sequences \eqref{deg2exact}, which are closely related to the ``DVB exact sequences'' characterizing double vector bundles given in \cite{CLS14}.
 \hfill $\diamond$
\end{remark}

We illustrate the above correspondence with simple examples.

\begin{ex}\label{ex:1Man}
Any vector bundle $E\to M$ defines an object in $\VB2$ with $E_1=E$, $\widetilde{E}=\wedge^2 E_1$ and $\phi_E = \mathrm{Id}_{\wedge^2 E_1}$. 
In this way the category of vector bundles sits in $\VB2$ as a (full) subcategory. The restriction of $\cF$ to this subcategory recovers the well-known equivalence between vector bundles and $\mathbb{N}$-manifolds of degree 1. The image of $E\to M$ under this equivalence is denoted by $\cM = E[1]$ and satisfies $C_\cM = \Gamma_{\wedge E^*}$.  
\hfill $\diamond$
\end{ex}

\begin{ex}[The split case]\label{ex:split}
A pair of vector bundles $E_1\to M$ and $E_2\to M$ defines an object in $\VB2$ with $\tilde{E} = E_2\oplus \wedge^2 E_1$ and $\phi_E\colon \tilde{E} \to \wedge^2 E_1$ the natural projection.  We refer to an object of this type as {\em split}, and keep the same terminology for its corresponding $\mathbb{N}$-manifold of degree 2  via $\cF$; in this case the sheaf $C_\cM$ (see \eqref{eq:CM}) has a simpler description in terms of $E_1$ and $E_2$ as 
$$
C_\cM = \Gamma_{\wedge E_1^*} \otimes
\Gamma_{\mathrm{S}E_2^*}.
$$
Given an arbitrary object $(E_1, \tilde{E}, \phi_E)$, 
the fact that one can always choose a splitting of the exact sequence \eqref{deg2exact} implies that any $\mathbb{N}$-manifold of degree 2 is isomorphic to one that is split, though this identification is generally noncanonical, see \cite{BonavolontaPoncin}.
\hfill $\diamond$
\end{ex}

\begin{ex}[{The tangent bundle}]\label{ex:tangent}
As in clasical geometry, there is a tangent functor on $\NMan$, that associates to 
each $\mathbb{N}$-manifold $\cM$ of degree 2 and body $M$ its tangent bundle $T\cM$, which is
an $\mathbb{N}$-manifold of degree 2 and body $TM$, see e.g. \cite[Prop.~2.3.23]{rajthesis} and \cite[Ex.~2.13]{cuecaschnitzer}. Following Theorem~\ref{thm:equivcat}, we will describe the counterpart of the tangent functor in $\VB2$.

Let  $(E_1,\widetilde{E},\phi_E)$ be an object in $\VB2$, corresponding to $\cM$.  To define the object in $\VB2$ representing its tangent bundle, we consider the canonical surjective map (of vector bundles over $TM$)
$$
q: \wedge^2_{TM} TE_1\rightarrow T (\wedge^2 E_1),
$$
obtained by differentiating the natural skew-symmetric bilinear map $E_1 \times _M E_1 \rightarrow \wedge^2 E_1$, $(e,e')\mapsto e\wedge e'$, and define the vector bundle $\widetilde{E}_T \to TM$ as the fibred product
$$
\widetilde{E}_T :=  T\tilde{E}\times_{T (\wedge^2 E_1)} \wedge^2_{TM} TE_1
$$
with respect to $T\phi_E$ and $q$.
Note that the natural projection $\mathrm{pr}_2: \widetilde{E}_T\to  \wedge^2_{TM} TE_1$ is surjective, and therefore the triple  
\begin{equation}\label{eq:tangtriple}
(TE_1, \widetilde{E}_T, \mathrm{pr}_2)
\end{equation}
is an object in $\VB2$. One can see from the short exact sequence
\begin{equation*}\label{eq:Tseq}
   0 \longrightarrow TE_2 \longrightarrow T\tilde{E}\overset{T\phi_E}{\longrightarrow} T (\wedge^2 E_1) \longrightarrow 0 
\end{equation*}
 of vector bundles over $TM$ (obtained from \eqref{deg2exact} via the tangent functor) that $\widetilde{E}_T$ fits into the short exact sequence
\begin{equation}\label{eq:seqT}
0 \longrightarrow TE_2 \longrightarrow \widetilde{E}_T \overset{\mathrm{pr}_2}{\longrightarrow} \wedge^2_{TM} TE_1 \longrightarrow 0.    
\end{equation}
Moreover, any splitting  $s$  of $\phi_E\colon \tilde{E}\to \wedge^2 E_1$
canonically induces a splitting of $\mathrm{pr}_2: \widetilde{E}_T\to \wedge^2_{TM}TE_1$ given by
$$
\wedge^2_{TM} TE_1 \rightarrow \widetilde{E}_T = T\widetilde E \times_{T (\wedge^2 E_1)} \wedge^2_{TM} TE_1, \quad \;\; v \mapsto (Ts(q(v)),v).
$$  

In particular, when $\cM$ is split as in Example~\ref{ex:split}, the object $(TE_1, \widetilde{E}_T, \mathrm{pr}_2)$ is also split, defined by $TE_1\to TM$ and $TE_2\to TM$. This implies that, in the split case, $(TE_1, \widetilde{E}_T, \mathrm{pr}_2)$ corresponds to $T\cM$, and the same holds in general since any $\mathbb{N}$-manifold of degree 2 is isomorphic to a split one.
\hfill $\diamond$
\end{ex}

\begin{ex}\label{ex:gradedVS}
Any graded vector space $V=V_{-2}\oplus V_{-1} \oplus V_0$, concentrated in degrees $-2$, $-1$ and $0$, may be regarded as a split $\mathbb{N}$-manifold of degree 2 corresponding to the vector bundles
$$
E_1= (V_{-1}\times V_0) \to V_0, \qquad E_2= (V_{-2}\times V_0) \to V_0.
$$
\hfill $\diamond$
\end{ex}

%%%%%%%%%%%%%%%%%%%%%%%%%%%%%%%%%%%%%%%%%%%%%%%%%%%%%%%%%%%%%%%%%%%%%%
%\subsection{\texorpdfstring{Submanifolds of degree $2$ $\N$-manifolds}{Submanifolds of degree 2 N-manifolds}}\label{sec:Nsub}
\subsection{Submanifolds}\label{subsec:Nsub}

Let $\cM$ be an $\N$-manifold of degree 2 with dimension $m_0|m_1|m_2$.  %$\dim(\cM)= m_0|m_1|m_2$. 
A \textit{submanifold} of $\cM$ is defined by an $\N$-manifold of degree 2
$\cN$ of dimension $n_0|n_1|n_2$ with $N\subseteq M$, and a morphism
$(\iota,\iota^\sharp):\cN \to \cM$ such that $\iota:N\hookrightarrow
M$ is the inclusion map and the following local condition holds: any
$x\in N$ admits a neighborhood $U$ in $M$ with local coordinates
$\{x^i, e^\mu, p^I\}$ with respect to which the map
\begin{equation}\label{eq:localisharp}
\iota^\sharp|_U \colon C_{\cM}|_U \to C_{\cN}|_{U\cap N}
\end{equation}
is given by
\begin{align}
& \iota^\sharp (x^i)=\iota^*(x^i) =0, \;\;\; 1\leq i\leq  r_0, \label{eq:sub1}\\
& \iota^\sharp(e^\mu)=0,\;\qquad \qquad \; 1\leq \mu \leq r_1,\label{eq:sub2}\\
& \iota^\sharp (p^I)=0,\;\qquad\qquad \; 1\leq I\leq
r_2,\label{eq:sub3}
\end{align}
where $r_l=m_l-n_l$, $l=0,1,2$, and in such a way that
$$
\{\iota^\sharp(x^i),\iota^\sharp(e^\mu),\iota^\sharp(p^I)\},\;\;
r_0<i\leq m_0,\; r_1<\mu \leq m_1, \;r_2<I \leq m_2,
$$
are local coordinates for $\cN$ over $U\cap N$. Such local
coordinates $\{x^i,e^\mu,p^I\}$ of $\cM$ on $U$ are called
\textit{adapted to $\cN$}. In particular, note that $N\subseteq M$
is an embedded submanifold. We refer to $r_0|r_1|r_2$ as the
\textit{codimension} of $\cN$.

Let $(\iota,j^\sharp):\cN'\to \cM$ be another submanifold with $N' =
N$. We consider the submanifolds $\cN$ and $\cN'$ to be
\textit{equivalent} if there is an isomorphism of sheaves
$\psi^\sharp \colon C_{\cN}\to C_{\cN'}$ for which the induced isomorphism
$\psi^\sharp \colon \iota_*C_{\cN}\to \iota_*C_{\cN'}$ satisfies
$\psi^\sharp\circ \iota^\sharp = j^\sharp$. We will make no
distinction between equivalent submanifolds, and we will keep the
term \textit{submanifold} to refer to an equivalence class\footnote{Even though we assume that submanifolds satisfy $N\subseteq M$, the notion of equivalence is needed since, in contrast with ordinary geometry, the map of sheaves is not determined by the inclusion map $\iota: N\to M$ of bodies.}.

There is a geometric description of submanifolds in light of Theorem~\ref{thm:equivcat}.
Let $(E_1,\tilde{E},\phi_E)$ be an object in $\VB2$. By a \textit{subobject}
of $(E_1,\tilde{E},\phi_E)$  we mean an object $(F_1,\tilde{F},\phi_F)$
in $\VB2$, where $F_1$ and $\tilde{F}$ are vector bundles over $N$, equipped with
a morphism into $(E_1,\tilde{E},\phi_E)$, defined by $j_1 \colon F_1\to
E_1$ and ${\tilde{j}} \colon \tilde{F}\to \tilde{E}$, so that $j_1$ and ${\tilde{j}}$ 
are fiberwise injective and cover an embedding $\iota \colon N\to M$. We will naturally identify subobjects of $(E_1,\tilde{E},\phi_E)$ which are isomorphic through an isomorphism that commutes with their respective morphisms into $(E_1,\tilde{E},\phi_E)$. In this way, we can always assume that a subobject $(F_1,\tilde{F},\phi_F)$ is such that $F_1\subseteq E_1|_N$, $\tilde{F}\subseteq \tilde{E}|_N$, and $\phi_E(\tilde{F})=\wedge^2 F_1\subseteq \wedge^2 E_1$, so that $\phi_F = \phi_E |_{\tilde{F}}$.

Suppose that $\cM$ corresponds to $(E_1,\tilde{E},\phi_E)$ via the functor \eqref{eq:F}.

\begin{prop}\label{prop:subobjects}
Submanifolds of $\cM$ are equivalent to either of the following geometric data:
\begin{itemize}
    \item Subobjects of
 $(E_1,\tilde{E},\phi_E)$ in $\VB2$, i.e., pairs of vector subbundles $F_1\subseteq E_1|_N$, $\tilde{F}\subseteq \tilde{E}|_N$, such that $\phi_E(\tilde{F})=\wedge^2 F_1\subseteq \wedge^2 E_1$; 
    \item Pairs of vector subbundles $K_1\subseteq E_1^*$ and $\tilde{K}\subseteq \tilde{E}^*$ over a submanifold $N\subseteq M$ satisfying $\tilde{K}\cap \wedge^2E_1^*|_N = K_1\wedge E_1^*|_N$.
\end{itemize}
The explicit correspondence between the two sets of geometric data is given by $K_1=\mathrm{Ann}(F_1)$ and $\tilde{K}=\mathrm{Ann}(\tilde{F})$.
\end{prop}

\begin{proof}
It is a direct verification that, through the equivalence in Theorem~\ref{thm:equivcat}, subobjects of
 $(E_1,\tilde{E},\phi_E)$ in $\VB2$ are in one-to-one correspondence with submanifolds of $\cM$.
 
Given a pair of vector subbundles $F_1\subseteq E_1$ and $\tilde{F}\subseteq \tilde{E}$ over a submanifold $N\subseteq M$, let
$K_1=\mathrm{Ann}(F_1)$ and $\tilde{K}=\mathrm{Ann}(\tilde{F})$.
Then the condition $\phi_E(\tilde{F})=\wedge^2 F_1$  holds 
if and only if  $\tilde{K}\cap \wedge^2E_1^*|_N = K_1\wedge E_1^*|_N$.
 \end{proof}

In the special case of an object defined by a vector bundle $E\to M$ as in Example~\ref{ex:1Man}, the previous proposition recovers the correspondence between vector subbundles $F\to N$ of $E$ and submanifolds $\cN=F[1]$ of $\cM= E[1]$. For a split  $\mathbb{N}$-manifold of degree 2  $\cM$ corresponding to a pair of vector bundles $E_1$ and $E_2$ as in Example~\ref{ex:split}, a corresponding pair of vector subbundles defines a split submanifold of $\cM$, but not every submanifold of $\cM$ is split.

%%%%%%%%%%%%%%%%%%%%%%%%%%%%%%%%%%%%%%%%%%%%%%%%%%%%%%%%%%%%%%%%%%%%%%
\subsubsection{\underline{Submanifolds and regular ideals}}\label{subsubsec:ideal}\
Let $\cM$ be an $\mathbb{N}$-manifold of degree 2.
Any submanifold $(\iota,\iota^\sharp)\colon \cN\to
\cM$ gives rise to a subsheaf of homogeneous
ideals
\begin{equation}\label{eq:ideal}
\mathcal{I}=\ker(\iota^\sharp)\subseteq C_{\cM},
\end{equation}
called the {\em sheaf of vanishing ideals} of $\cN$.
Note that  $\cI_0 = \cI \cap (C_{\cM})_0 \subseteq C^\infty_M$ coincides with $I_N$, the vanishing ideal of $N\subseteq M$, the body of $\cN$. 
The sheaf of functions on $\cN$ is recovered from $\cI$ via  
$$
C_\cN = \iota^{-1}(C_\cM/\cI).
$$

We will now characterize the sheaves of ideals in $C_\cM$ that arise as vanishing ideals of submanifolds. 

Given a subsheaf of ideals $\mathcal{I} \subseteq C_{\cM}$, we call it \textit{regular}  
if 
\begin{equation}\label{eq:vanI}
\cI_0 := \cI \cap (C_{\cM})_0 =   I_N
\end{equation} 
for a 
subset $N\subseteq M$ and the following local condition holds: There exist $r_j \in
\{1,\ldots,m_j\}$, $j=0, 1, 2$, such that any $x \in N$ has a
neighborhood $U\subset M$ with local coordinates $\{ x^i, e^\mu,
p^I\}$ satisfying the property that $\{ x^1,\ldots, x^{r_0},
e^1,\ldots, e^{r_1}, p^1,\ldots, p^{r_2}\}$ generates the sheaf of
ideals $\cI|_U$. Note that $N$ is not uniquely specified by $\cI$, since different subsets of $M$ may have the same vanishing ideal. But the local condition implies that any $N \subseteq M$ satisfying \eqref{eq:vanI} is an embedded submanifold.

\begin{prop}\label{prop:1-1}
For a submanifold $\cN$ of $\cM$, its sheaf of vanishing ideals is regular.
Moreover, given a submanifold $N\subseteq M$, this correspondence establishes a  bijection between
submanifolds of $\cM$ with body $N$ and
regular sheaves of ideals $\cI\subseteq C_{\cM}$ 
 such that
$\cI_0 =   I_N$.
\end{prop}

This result will follow from the geometric description of regular sheaves of ideals in the next lemma. 

Suppose that $\cM$ corresponds to the object $(E_1, \widetilde{E}, \phi_E)$ in $\VB2$.

\begin{lemma}\label{lem:geomideal}
For a given submanifold $N\subseteq M$,
regular sheaves of ideals  $\cI\subseteq C_{\cM}$ with $\cI_0=I_N$ are equivalent to the following geometric data: vector subbundles $K_1\subseteq E_1^*$ and
$\tilde{K}\subseteq \tilde{E}^*$ over $N$ satisfying $\tilde{K}\cap \wedge^2E_1^*|_N = K_1\wedge E_1^*|_N$.
\end{lemma}

The explicit correspondence is determined by the conditions  
\begin{equation}\label{eq:I123}
\cI_0 = \fl_N, \quad \cI_1 = \Gamma_\st{E_1^*,K_1}, \; \mbox{ and } \;\; \cI_2 =
\Gamma_\st{\tilde{E}^*,\widetilde{K}},
\end{equation} 
where $\mathcal{I}_k = \mathcal{I}\cap (C_\cM)_k$, for $k=0, 1, 2$. Recall from the end of Section~\ref{subsec:notation}  that  $\Gamma_\st{E_1^*,K_1}$ denotes the sheaf of sections of the vector bundle $E_1^*$ which restrict to sections of $K_1$ over $N$, and similarly for $\Gamma_\st{\tilde{E}^*,\widetilde{K}}$. Note also that a regular sheaf of ideals 
$\cI$ coincides with the sheaf of ideals
generated by $\cI_0 + \cI_1 + \cI_2$ as a result of its local property.
 
\begin{proof}[Proof of Lemma~\ref{lem:geomideal}]
We first check that any  regular sheaf of ideals $\cI$ with $\cI_0=I_N$ is given as in \eqref{eq:I123}.  The local property of $\cI$
can be rephrased as follows: any $x \in
N$ has a neighborhood $U\subseteq M$ where $C_{\cM}$ has local
coordinates $$
\{ x^1,\ldots,x^{m_0},
e^1,\ldots,e^{m_1},p^1,\ldots,p^{m_2}\}
$$ such that the sheaves of
$C^\infty_M$-modules $\cI_0$, $\cI_1$ and $\cI_2$ are generated by
\begin{align}
&\{x^1,\ldots,x^{r_0}\},\label{eq:0-gen}\\
&\{e^1,\ldots,e^{r_1}, x^ie^\mu   \}, \;\; i=1,\ldots, r_0,\label{eq:1-gen}\\
&\{p^1,\ldots,p^{r_2}, e^\mu e^\gamma, x^ip^I, x^i e^\mu e^\nu\},
\;\; i=1,\ldots, r_0,\; \; \gamma=1,\ldots, r_1,\label{eq:2-gen}
\end{align}
respectively, where $\mu$, $\nu$, and $I$ range through all indices.

The pullback
sheaf $\iota^{*}(C(\cM)_1)$ with respect to the inclusion  $\iota \colon  N\hookrightarrow M$ is the sheaf of sections
$\Gamma_{E_1^*|_{N}}$, and the local description \eqref{eq:1-gen} of $\cI_1$ around
points in $N$ shows that $\iota^{*}(\cI_1) = \Gamma_{K_1}$
for a vector subbundle $K_1\subseteq E_1^*|_{N}$.
Note that a section of $E_1$  over a small neighborhood $U\subseteq M$ of any $x\in N$ vanishes along $N\cap U$ if and only if it is a local section of $\cI_1$, so  $\cI_1 = \Gamma_\st{E_1^*,K_1}$. By   similar arguments,
there is a vector subbundle $\widetilde{K}\subseteq \widetilde{E}^*|_N$ such that
$\cI_2 = \Gamma_\st{\widetilde{E}^*,\widetilde{K}}$.
The description of $\cI_2$ around points in $N$ given by
the generators \eqref{eq:2-gen} shows that  $\tilde{K}\cap
\wedge^2 {E}^*|_N = K_1\wedge E_1^*|_N$. 
Conversely, we can reverse the arguments to show that, given $K_1$ and $\widetilde{K}$ as in the statement,  we obtain  a corresponding regular sheaf of ideals by means of the conditions \eqref{eq:I123}. 
\end{proof}

To conclude the proof of Proposition~\ref{prop:1-1}, note that if a submanifold $\cN$ of $\cM$ corresponds to a subobject $(F_1, \widetilde{F}, \phi_F)$, then its   sheaf of vanishing ideals $\cI$ corresponds, as in Lemma~\ref{lem:geomideal}, to the subbundles $K_1=\mathrm{Ann}(F_1)$ and  $\tilde{K}=\mathrm{Ann}(\tilde{F})$. Hence
Proposition~\ref{prop:1-1} follows directly from the equivalences in Prop. \ref{prop:subobjects}.

%%%%%%%%%%%%%%%%%%%%%%%%%%%%%%%%%%%%%%%%%%%%%%%%%%%%%%%%%%%%%%%%%%
\subsubsection{\underline{Another geometric characterization of submanifolds}}\

We will describe yet another geometric characterization of submanifolds of $\N$-manifolds of degree 2; compared to the data in Lemma~\ref{lem:geomideal}, we will express the subbundle $\tilde{K}\subseteq \tilde{E}^*$
in terms of a subbundle of the vector bundle $E_2^*=\tilde{E}^*/\wedge^2 E_1^*$ (see  \eqref{deg2exact}), whose sheaf of sections can be interpreted as functions of ``pure'' degree $2$. This alternative description will be convenient for the geometric characterization of coisotropic submanifolds in $\S$ \ref{subsubsec:coisogeom} (see Theorem~\ref{thm:coisotropic}).

\begin{lemma}\label{lem:Ktil}
Let  $K_1\subseteq E_1^*$ and
$\tilde{K}\subseteq \tilde{E}^*$ be vector subbundles over a submanifold $N\subseteq M$. Consider
the following diagram of natural projections:
\begin{equation}\label{eq:diag}
\xymatrix{ \tilde{E}^*|_{N}  \ar[rr]^{\pi} \ar[dr]_{\pi''} & &
E_2^*|_{N}\\
&\frac{\tilde{E}^*|_{N}}{K_1\wedge E_1^*|_{N}} \ar[ur]_{\pi'}. & }
\end{equation}
The following are equivalent:
\begin{itemize}

\item[$(a)$] $\tilde{K}\cap \wedge^2E_1^*|_N = K_1\wedge E_1^*|_N$.

\item[$(b)$] There is a vector subbundle $K_2\subseteq
E_2^*|_N$ and a vector-bundle map $\phi \colon K_2 \to
\frac{\tilde{E}^*|_{N}}{K_1\wedge E_1^*|_{N}}$ satisfying $\pi'\circ
\phi = Id$ such that $\widetilde{K}=(\pi'')^{-1}(\phi(K_2))$.

\end{itemize}
\end{lemma}

\begin{proof}
Note that $\tilde{K}$ is of the form $(\pi'')^{-1}(\phi(K_2))$ as in
$(b)$ if and only if \begin{itemize}
    \item $\pi(\tilde{K})$ has constant rank (so we can
set $K_2=\pi(\tilde{K})$), 
\item $\ker(\pi'')=K_1\wedge E_1^*|_N \subseteq
\tilde{K}$, and 
\item $\pi'|_{\pi''(\tilde{K})}$ is injective (so that it
is an isomorphism onto $K_2$, and $\phi$ is its inverse).
\end{itemize}
The injectivity of $\pi'|_{\pi''(\tilde{K})}$ amounts to the condition
$$
\pi'(\pi''(k))=\pi(k)=0 \implies k\in \ker(\pi'')= K_1\wedge
E_1^*|_N,
$$
for all $k\in \tilde{K}$; i.e., $\ker(\pi)\cap
\tilde{K}=\wedge^2E_1^*|_N\cap \tilde{K} = K_1\wedge E_1^*|_N$. It
directly follows that $(a)$ and $(b)$ are equivalent.
\end{proof}

From Lemma~\ref{lem:Ktil} and the results in the previous subsection (see 
Prop.~\ref{prop:1-1} and Lemma~\ref{lem:geomideal}), we obtain a geometric characterization of submanifolds that will be useful in Section \ref{subsec:coisosubm}.

\begin{thm} \label{thm:submanifold}
Submanifolds of $\cM$ of codimension $r_0|r_1|r_2$ are equivalent to quadruples $(N,K_1,K_2,\phi)$, where 
\begin{itemize}
\item $N\subseteq M$ is
a submanifold of codimension $r_0$, 
\item $K_1\subseteq E_1^*|_{N}$
and $K_2\subseteq E_2^*|_{N}=\frac{\tilde{E}^*|_N}{\wedge^2 E_1^*|_N
}$
are vector subbundles of ranks $r_1$ and $r_2$, respectively, and
\item $\phi \colon {K_2 \to \cfrac{\tilde{E}^*|_{N}}{K_1\wedge E_1^*|_{N}}}$
is a vector bundle map
such that $\pi'\circ \phi = Id$, for $\pi'$
defined in \eqref{eq:diag}.
\end{itemize}

\end{thm}

We recall how to make the correspondence explicit. For a given quadruple $(N, K_1, K_2, \phi)$, we set 
$$
\tilde{K}= (\pi'')^{-1}(\phi(K_2)),
$$ 
as in Lemma~\ref{lem:Ktil}. Let 
$$
F_1=\mathrm{Ann}(K_1)\subseteq E_1|_N,\;\; \widetilde{F}=\mathrm{Ann}(\widetilde{K})\subseteq \tilde{E}|_N,\;\;\;
\phi_F = \phi_E|_{\widetilde{F}}\colon \widetilde{F}\to \wedge^2 F_1.
$$
Then the submanifold $\cN \hookrightarrow \cM$ corresponding to $(N, K_1, K_2, \phi)$ is the one defined by $(F_1, \widetilde{F}, \phi_F)$ via
$(C_\cN)_1 = \Gamma_{F^*_1}$ and $(C_\cN)_2 =
\Gamma_{\tilde{F}^*}$. 
From another perspective,  the regular sheaf of vanishing ideals $\cI$ representing the submanifold corresponding to
$(N, K_1, K_2, \phi)$ is determined as in \eqref{eq:I123}.

%%%%%%%%%%%%%%%%%%%%%%%%%%%%%%%%%%%%%%%%%%%%%%%%%%%%%%%%%%%%%%%%%%%%%%%%%%%%%
\subsection{Tangent bundle and differential calculus}\label{subsec:tangent}

In this section we collect some basic results on the differential
calculus on $\N$-manifolds of degree 2. Analogous results for
supermanifolds can be found e.g. in
\cite{BruzzoCianci,fiori,varadarajan}.

We will make use of Einstein's convention for sums whenever there is no risk of confusion.

We will need the notion of derivation of a graded algebra (relative to a morphism). Given graded algebras $A$, $B$, and a morphism $\phi: A\to B$, a {\em degree $q$ derivation on $A$ relative to $\phi$} is a linear map $X: A\to B$ of degree $q$ satisfying 
$$
X(a_1a_2) = X(a_1) \phi(a_2) + (-1)^{q|a_1|}\phi(a_1) X(a_2).
$$
The  set of such derivations is an $A$-module denoted by $\mathrm{Der}_\phi(A,B)$. When $\phi$ is the identity map $A\to A$, we use the simplified notation $\mathrm{Der}(A)$. We have analogous definitions when $\mathcal{A}$ and $\mathcal{B}$ are sheaves of graded algebras over $M$ and $\phi: \mathcal{A}\to \mathcal{B}$ is a morphism of sheaves. In this case $\mathrm{Der}_\phi(\mathcal{A},\mathcal{B})$ is a sheaf of $\mathcal{A}$-modules over $M$.

%%%%%%%%%%%%%%%%%%%%%%%%%%%%%%%%%%%%%%%%%%%%%%%%%%%%%%%%%%%%%
\subsubsection{\underline{Vector fields and tangent vectors}}\

Let $\cM$ be an $\N$-manifold of degree 2. A \textit{degree $q$ vector
field} $X$ on $\cM$ is a degree $q$ derivation of the sheaf
$C_{\cM}$; this means, as recalled above, that for each open subset $U\subseteq M$, $X$ defines a
linear map $C_{\cM}(U)\to C_{\cM}(U)$ such that for any $f,g \in
C_{\cM}(U)$, where $f$ is homogeneous, $|X(f)| = |f| + q$ and
$$
X(fg) = X(f)g + (-1)^{q|f|}f X(g).
$$
The sheaf of derivations $\mathrm{Der}(C_\cM)$ is called the {\em tangent sheaf} of $\cM$
and denoted
by $\mathcal{T}\cM$. We will also use the notation $\mathfrak{X}(\cM)= \mathcal{T}\cM(M)$.

Given two homogeneous vector fields $X$ and
$Y$, of degrees $|X|$ and $|Y|$, their Lie bracket is defined as the graded commutator
$$
[X,Y]
\defequal XY - (-1)^{|X||Y|}YX,
$$
so, besides being a sheaf of $C_\cM$-modules, $\mathcal{T}\cM$ is also a sheaf of graded Lie algebras.

In local coordinates $\{x^i,e^\mu, p^I\}$ on $\cM$, one has the
associated vector fields
\begin{equation}\label{eq:der}
\frac{\partial}{\partial x^i}, \;\; \frac{\partial}{\partial e^\mu},
\;\; \frac{\partial}{\partial p^I},
\end{equation}
 of degrees $0$, $-1$ and $-2$,
respectively; with respect to these coordinates, any degree $q$
vector field $X$ can be written in the form
\begin{equation*}
    X = a^i \pdiff{}{x^i} + b^\mu \pdiff{}{e^\mu} + c^I \pdiff{}{p^I},
\end{equation*}
where $|a^i| = q$, $|b^\mu| = q+1$, and $|c^I| = q+2$, showing that
the sheaf of modules $\mathcal{T}\cM$ is locally free. Note that the space of vector fields of degree $q$ is trivial for $q<-2$.

As discussed e.g. in \cite[Prop.~2.3.23]{rajthesis} and \cite[Ex.~2.13]{cuecaschnitzer}, one can regard the tangent bundle of $\cM$ as a vector bundle in the realm of $\N$-manifolds and define vector fields via its sections; see  Example~\ref{ex:tangent}.

\begin{ex}\label{ex:VFindeg1}
For an $\mathbb{N}$-manifold of degree 1  $\cM= E[1]$ (see Example~\ref{ex:1Man}), so that $C_\cM= \Gamma_{\wedge^\bullet E^*}$, degree 0 vector fields  agree with {\em derivations} of $E\to M$ (or, equivalently, $E^*\to M$, by duality), i.e., pairs $(X,D)$ where $X$ is a (ordinary) vector field on $M$ and $D\colon \Gamma(E)\to \Gamma(E)$ is $\mathbb{R}$-linear and satisfies
$D(fe)= (\pounds_Xf)e + f D(e)$, for $f\in C^ \infty(M)$.
Degree -1 vector fields are $C^ \infty(M)$-linear maps $\Gamma(E^*)\to C^ \infty(M)$, which can be seen as sections of $E$. See \cite[Lemma 1.6]{ZZL} for more details.
\hfill $\diamond$
\end{ex}

\begin{remark}\label{rem:geometricVF}
For an $\mathbb{N}$-manifold of degree 2, the  geometric description of vector fields in degrees $0$, $-1$ and $-2$ in terms of the corresponding object $(E_1, \tilde{E}, \phi_E)$ in $\VB2$ can be found in  \cite[\S 4.3]{HenMiqRaj} (see also \cite[$\S$ 3.5]{delCarpioThesis}).
\hfill $\diamond$
\end{remark}

\begin{remark}\label{rem:Qstructures}
A degree 1 vector field $Q$ on $\cM$ satisfying $[Q,Q]=2Q^2=0$ is called a {\em $Q$-structure}, or {\em homological vector field}. For an $\mathbb{N}$-manifold of degree 1 $\cM= E[1]$ (see Example~\ref{ex:1Man}), 
a $Q$-structure is equivalent to a Lie algebroid structure on $E\to M$ \cite{vaintrob} ($Q$ is just the Lie algebroid differential on $\Gamma(\wedge^\bullet E^*)$). For $\mathbb{N}$-manifolds of degree 2, $Q$-structures (also known as {\em Lie 2-algebroids}) have been described in \cite{delCarpioThesis, JotzMatchedPairs, Li-BlandThesis}. (See \cite{Behrend:2020tn} for more on differential graded manifolds.)
\hfill $\diamond$
\end{remark}

Let $C_\cM|_x$ be the stalk of $C_\cM$ at a point $x\in M$. A
\textit{tangent vector} to $\cM$ at a  point $x\in M$ is
 a derivation of $C_\cM|_x$ relative to the morphism $C_\cM|_x\to \mathbb{R}$, $f|_x \mapsto f_0(x)$ that evaluates the degree zero component of $f|_x$ at $x$. 
In other words, a homogeneous tangent vector at $x$ is a linear map $X_x  \colon C_{\cM}|_x \to \mathbb{R}$, of degree $|X_x|$, satisfying
\begin{equation}\label{eq:derprop}
X_x (f g|_x) = X_x(f|_x) g_0(x) + (-1)^{|X_x||f|}f_0(x)X_x(g|_x).
\end{equation}

Any $X\in \mathcal{T}\cM(U)$ defines a tangent vector $X_x$ at each
$x\in U$ by
\begin{equation}\label{eq:tangent}
X_x(f|_x) = (X(f))_0(x).
\end{equation}
Whenever there is no risk of confusion, we may simplify the notation
and write $f$ instead of $f|_x$. 

In terms of local coordinates $\{x^i,e^\mu, p^I\}$ around $x\in M$,
the derivations \eqref{eq:der} define tangent vectors at $x$,
denoted by
\begin{equation}\label{eq:derx}
\left( \frac{\partial}{\partial x^i} \right)_x, \;\;
\left(\frac{\partial}{\partial e^\mu} \right)_x, \;\;
\left(\frac{\partial}{\partial p^I}\right)_x,
\end{equation}
of degrees $0$, $-1$, and $-2$, respectively. The space of tangent
vectors at $x\in M$ is a graded vector space over $\R$, denoted by
$\mathcal{T}_x\cM$, having \eqref{eq:derx} as a basis. In
particular, note that if $X$ is a homogeneous local section of
$\mathcal{T}\cM$ and $X_x\neq 0$ at some $x$, then $|X|\in
\{-2,-1,0\}$.

\begin{prop}\label{prop:li}
Let $U\subseteq M$ be open and consider
$$
X_1,\ldots,X_{k_0},Y_1,\ldots,Y_{k_1},Z_1,\ldots,Z_{k_2} \in
\mathcal{T}\cM(U),
$$
where $|X_j|=0, |Y_\nu|=-1,|Z_J|=-2$. If $(X_1)_x,\ldots,(Z_{k_2})_x
\in \mathcal{T}_x\cM$ are linearly independent for all $x\in U$,
then $X_1,\ldots,Z_{k_2}$ are linearly independent over $C_\cM(U)$.
\end{prop}

Note that the converse does not hold: the vector field $X=
p\frac{\partial}{\partial p}$ of degree $0$ is linearly independent
over $C_\cM$, but $X_x=0$.

\begin{proof}
To show that $X_1,\ldots,Z_{k_2} \in \mathcal{T}\cM(U)$ are linearly
independent, it suffices to consider $U$ admitting coordinates
$\{x^i,e^\mu,p^I\}$. Since each $(X_j)_x$ (resp. $(Y_\nu)_x$, resp.
$(Z_J)_x$) is a linear combination of $(\partial/\partial x^i)_x$
(resp. $(\partial/\partial e^\mu)_x$, resp. $(\partial/\partial
p^J)_x$), we see that the linear independence of
$(X_1)_x,\ldots,(Z_{k_2})_x$ is equivalent to the conditions
\begin{align}
& f^j (X_j)_x(x^i)=0 \qquad  \forall \;\; i=1,\ldots,m_0,\;  \implies f^j=0,  \label{eq:lix1}\\
& g^\nu (Y_\nu)_x(e^\mu)=0 \qquad  \forall\; \mu=1,\ldots,m_1,\;  \implies g^\nu=0,\label{eq:lix2}\\
& h^J (Z_J)_x(p^I)=0 \qquad  \forall \; I=1,\ldots,m_2,\; \implies
h^J=0,\label{eq:lix3}
\end{align}
where $f^j, g^\nu, h^J \in \mathbb{R}$.

Consider the condition
\begin{equation}\label{eq:li}
a^j X_j + b^\nu Y_\nu + c^J Z_J = 0,
\end{equation}
where $a^j, b^\nu, c^J \in C_\cM(U)$. By 
applying the left-hand
side of this equation on $x^i$, and since there are no functions of negative degrees, we see that
\begin{equation}\label{eq:li1}
a^j (X_j(x^i)) = 0,
\;\;\;
i=1,\ldots,m_0.
\end{equation}
Since $a^j$ is a polynomial in $e^\mu$ and $p^I$ with coefficients
in smooth functions of $x^i$, it suffices to assume that $a^j$ is a monomial, i.e.,  of
the form $f^j(x)e^{\mu_1}\ldots e^{\mu_{r}}p^{I_1}\ldots p^{I_{s}}$,
for fixed $\mu_1\leq\ldots\leq \mu_r$ and $I_1\leq\ldots\leq I_s$.
It follows from \eqref{eq:li1} that
$$
(f^j X_j(x^i)) (x) = f^j(x)(X_j)_x(x^i)=0, \;\;\; \forall x\in U,
$$
and \eqref{eq:lix1} implies that $f^j = 0$, so $a^j=0$. Similarly,
applying \eqref{eq:li} on $e^\mu$ and using \eqref{eq:lix2}, one
concludes that $b^\nu=0$. Finally, applying \eqref{eq:li} on
$p^I$ and using \eqref{eq:lix3} one verifies that $c^J=0$.
\end{proof}

We now give an alternative characterization of $\mathcal{T}_x\cM$
for later use. For a fixed $x\in M$, let us consider the ideal
\begin{equation}\label{eq:I(x)}
\cI_{(x)}:= \{f\in C_\cM|_x\,,\, f_0(x)=0\}\subseteq C_\cM|_x,
\end{equation}
and the graded vector space $\cI_{(x)}/\cI_{(x)}^2$.
The derivation
property \eqref{eq:derprop} for an element $X_x\in \mathcal{T}_x\cM$
implies that $X_x|_{\cI_{(x)}^2}=0$, so there is a natural
degree-preserving map
\begin{equation}\label{eq:natmap}
\mathcal{T}_x\cM \longrightarrow \left
({\cI_{(x)}}/{\cI_{(x)}^2}\right )^*.
\end{equation}

\begin{prop}
The map \eqref{eq:natmap} is an isomorphism of graded vector spaces.
\end{prop}

\begin{proof}
Denote by $[f]$ the class of $f\in \cI_{(x)}$ in
$\cI_{(x)}/\cI_{(x)}^2$. Then, for each $\alpha \in
({\cI_{(x)}}/{\cI_{(x)}^2})^*$, 
$$
C_{\cM}|_x \to \mathbb{R},\;\; f \mapsto \alpha([f-f_0(x)]),
$$
defines an element in $\mathcal{T}_x\cM$, denoted by $X_x^\alpha$.
The map $({\cI_{(x)}}/{\cI_{(x)}^2})^*\to \mathcal{T}_x\cM$, $\alpha
\mapsto X^\alpha_x$, is the inverse of \eqref{eq:natmap}.
\end{proof}

%%%%%%%%%%%%%%%%%%%%%%%%%%%%%%%%%%%%%%%%%%%%%%%%%%%%%%%%%%%%%%%%%%%%%%
\subsubsection{\underline{Vector fields along submanifolds}}\

In classical differential geometry, for a submanifold $\iota: N\hookrightarrow M$, the tangent bundle $TN$ is a subbundle of $TM|_N=\iota^*TM$, and
the natural restriction map 
\begin{equation}\label{eq:restcl}
\mathfrak{X}(M)\to \Gamma(\iota^*TM), \quad X\mapsto X|_N,
\end{equation}
maps  vector fields on $M$ that are tangent to $N$ onto $\mathfrak{X}(N)$.
We will extend these basic facts to the context of $\N$-manifolds of degree 2 for later use in the discussion of coisotropic submanifolds in $\S$ \ref{subsec:coisosubm} (see Prop.~\ref{prop:distr}).

Let $(\iota,\iota^\sharp):\cN \to \cM$ be a 
submanifold of codimension $r_0|r_1|r_2$ with sheaf of vanishing ideals $\cI=\ker(\iota^\sharp)$. 
Consider the morphism
$$
\iota^\flat: \iota^{-1}C_\cM \to C_\cN.
$$
corresponding to $\iota^\sharp: C_\cM\to \iota_* C_\cN$ (by adjunction, see \cite[$\S$ 2.3]{kash-schapira}).

As a sheaf of $C_\cN$-modules over $N$, the pullback $\iota^*\cT\cM$ is identified with the sheaf of derivations of $\iota^{-1}C_\cM$ relative to $\iota^\flat$,
$$
\iota^*\cT\cM = \mathrm{Der}_{\iota^\flat}(\iota^{-1} C_\cM,C_\cN).
$$
Note that,  when $N=\{x\}$, this recovers the tangent space $\cT_x\cM$. 

From the short exact sequence 
$$
0\to \iota^{-1} \cI \to \iota^{-1}C_\cM \stackrel{\iota^\flat}{\to} C_\cN \to 0,
$$
we see that the sheaf of derivations $\mathrm{Der}(C_\cN)=\cT \cN$
is identified with the subsheaf of $\iota^*\cT\cM$ given by sections that vanish on  $\iota^{-1}\cI$ (via $Y \mapsto Y\circ \iota^\flat$); in this way, we view
$$
\cT\cN\subseteq \iota^*\cT\cM.
$$

To define the analogue of the restriction map \eqref{eq:restcl}, it is convenient to regard $\iota^*\cT\cM$ as a sheaf over $M$ supported on $N$ by considering its direct image $\iota_*(\iota^*\cT\cM)$, which is given by the sheaf of derivations of $C_\cM$ relative to $\iota^\sharp: C_\cM\to \iota_*C_\cN$.
The  corresponding ``restriction'' map is given by
\begin{equation}\label{eq:rsharp}
r^\sharp: \cT \cM \to \iota_*(\iota^*\cT\cM) =  \mathrm{Der}_{\iota^\sharp}(C_\cM,\iota_*C_\cN), \quad X\mapsto \iota^\sharp \circ X.
\end{equation}
Its kernel, which consists of vector fields $X: C_\cM \to C_\cM$ satisfying $X(C_\cM)\subseteq \cI$, is denoted by $\cI.\mathcal{T}\cM$. We have an exact sequence
\begin{equation}\label{eq:exseq0}
0 \longrightarrow \cI.\mathcal{T}\cM \longrightarrow \mathcal{T}\cM
\stackrel{r^\sharp}{\longrightarrow} \iota_*(\iota^* \mathcal{T}\cM)
 \longrightarrow 0.
\end{equation}

In analogy with classical geometry, one defines vector fields on $\cM$ that are ``tangent to $\cN$'' by means of the subsheaf of $C_\cM$-modules
$\mathcal{T}_\cI\subseteq \mathcal{T}\cM$ given on open subsets
$U\subseteq M$ by
\begin{equation}\label{eq:TI}
\mathcal{T}_\cI (U) := \{ X\in \mathcal{T}\cM(U)\,|\,
X(\cI(U))\subseteq \cI(U) \}.
\end{equation}
Note that $\mathcal{T}_\cI$ is also a subsheaf of Lie subalgebras of $\cT \cM$,
$$
[\mathcal{T}_\cI,\mathcal{T}_\cI]\subseteq \mathcal{T}_\cI.
$$

In local coordinates $\{x^i,e^\mu,p^i\}$ on $U\subseteq M$ adapted
to $\cN$, note that $ \cI.\mathcal{T}\cM$  is given by the span over $\cI$ of $\{
{\partial}/{\partial x^i}, {\partial}/{\partial e^\mu},
{\partial}/{\partial p^I} \}$, and
\begin{equation}\label{eq:ITM}
\mathcal{T}_\cI |_U = \Span_{C_\cM|_U} \left \{
\frac{\partial}{\partial x^i},\frac{\partial}{\partial
e^\mu},\frac{\partial}{\partial p^I} \right \} + \cI .
\mathcal{T}\cM |_U,
\end{equation}
for $i=r_0+1,\ldots,m_0$, $\mu=r_1+1,\ldots,m_1$,
$I=r_2+1,\ldots,m_2$.

Regarding $\cT\cN$ as a subsheaf of $\iota^*\cT \cM$ as above leads to an identification
\begin{equation*}\label{eq:backTI}
\iota^*\cT_\cI = \cT\cN \subseteq \iota^* \cT \cM.
\end{equation*}

On the other hand, the image of $\cT_\cI$ under the map $r^\sharp$ in \eqref{eq:rsharp} is the subsheaf of 
$\mathrm{Der}_{\iota^\sharp}(C_\cM,\iota_*C_\cN)$ given by derivations $Y: C_\cM\to \iota_*C_\cN$
(relative to $\iota^\sharp$) satisfying $Y(\cI)=0$; recalling that $C_\cM/\cI \cong \iota_*C_\cN$, such derivations are naturally identified with derivations of $\iota_*C_\cN$. 
In conclusion,  \eqref{eq:rsharp} restricts to a map
\begin{equation}\label{eq:rest1}
\mathcal{T}_\cI\stackrel{r^\sharp}{\to} \{Y\in \mathrm{Der}_{\iota^\sharp}(C_\cM,\iota_*C_\cN),\, Y(\cI)=0\} \cong \mathrm{Der}(\iota_*C_\cN) = \iota_*\cT\cN
\end{equation}
that preserves Lie brackets and  fits into the following exact sequence:
\begin{equation}\label{eq:exseq}
0 \longrightarrow \cI.\mathcal{T}\cM \longrightarrow \mathcal{T}_\cI
\stackrel{}{\longrightarrow} \iota_*\mathcal{T}\cN
 \longrightarrow 0.
\end{equation}

For each $x\in N$, the subsheaf $\cT_\cI\subseteq \cT\cM$ defines a graded subspace
$(\cT_\cI)_x \subseteq \cT_x\cM$ given by
$$
(\cT_\cI)_x:= \{X_x \,:\, X\in \cT_\cI(U)\} = \{ X_x \in \cT_x\cM\,:\, X_x(\cI|_x)=0\}, 
$$
where $U$ is a small neighborhood of $x$. The map \eqref{eq:rest1} induces a graded linear map 
\begin{equation}\label{eq:isomTN}
(\cT_\cI)_x \to T_x\cN, \qquad X_x\mapsto (r^\sharp(X))_x,
\end{equation}
which coincides with the natural isomorphism
$$
\{ X_x \in \cT_x\cM\,:\, X_x(\cI|_x)=0\} \cong T_x\cN 
$$
arising from the fact that $C_\cN|_x = C_\cM|_x/\cI|_x$.

%%%%%%%%%%%%%%%%%%%%%%%%%%%%%%%%%%%%%%%%%%%%%%%%%%%%%%%%%%%%%%%%%%%%%%%%%
\subsubsection{\underline{Tangent maps and local normal forms}\label{subsubsec:tanmap}}\
 Any morphism $\Psi=(\psi,\psi^\sharp):
\cM\to \cN$ induces a \textit{tangent map} $(d\Psi)_x :
\mathcal{T}_x\cM \to \mathcal{T}_{\psi(x)}\cN$ by
$$
(d\Psi)_x(X_x) \colon C_\cN |_{\psi(x)} \to \mathbb{R},\;\;
(d\Psi)_x(X_x)(f) := X_x (\psi^\sharp f).
$$

Let $\{ x^j, e^\nu, p^J \}$ be local coordinates around $x_0$, and
$\{\overline{x}^i, \overline{e}^\mu, \overline{p}^I \}$ be local
coordinates around $\psi(x_0)$. With respect to these coordinates,
the morphism $\Psi$ can be written as
\begin{equation}\label{eq:Psi}
\psi^\sharp \overline{x}^i  = \psi^* \overline{x}^i = \psi^i(x),\;\;
\psi^\sharp\overline{e}^\mu = f^\mu_{\nu}(x) e^\nu,\;\; \psi^\sharp
\overline{p}^I = g_{\mu\nu}^I(x) e^\mu e^\nu + h^I_J(x) p^J.
\end{equation}

By considering the bases of $\mathcal{T}_{x_0}\cM$ and
$\mathcal{T}_{\psi(x_0)}\cN$ relative to the choice of coordinates (as in
\eqref{eq:derx}), one obtains the matrix expression
\begin{equation}\label{eq:dpsi}
(d\Psi)_{x_0}= \left (\begin{matrix} 
\left(\frac{\partial \psi^i}{\partial x^j}(x_0)\right) 
& 0 & 0\\
0 & \left(f^\mu_\nu(x_0)\right)^t & 0\\
0 & 0 & \left(h^I_J(x_0)\right)^t
\end{matrix}\right ).
\end{equation}

Suppose that $\cM$ and $\cN$ correspond to $(E_1, \widetilde{E}, \phi_E)$ and $(F_1, \widetilde{F}, \phi_F)$, respectively, objects in $\VB2$. Let 
$$
\psi_1\colon E_1\to F_1, \qquad \widetilde{\psi}\colon \widetilde{E}\to \widetilde{F}
$$ 
be the vector bundle maps (covering $\psi\colon M\to N$) corresponding to $\Psi$. By the compatibility \eqref{eq:compat}, $\widetilde{\psi}$ takes $E_2=\mathrm{ker}(\phi_E)$ to $F_2=\mathrm{ker}(\phi_F)$, and we denote the restricted map by 
$$
\psi_2\colon E_2\to F_2.
$$ 
By viewing $\{e^\nu\}$ and $\{p^J\}$ as local frames for $E_1^*$ and $E_2^*$, and $\{\overline{e}^\mu\}$ and $\{\overline{p}^I\}$ as local frames for $F_1^*$ and $F_2^*$, respectively,  we obtain the matrix expressions
$(\psi_1)_{x_0}^*= (f_\nu^\mu(x_0))$ and $(\psi_2)_{x_0}^* = (h^I_J(x_0))$, i.e.,
$$
(\psi_1)_{x_0} = (f_\nu^\mu(x_0))^t, \qquad (\psi_2)_{x_0} = (h^I_J(x_0))^t.
$$
Comparing with \eqref{eq:dpsi}, the next result readily follows.

\begin{prop}\label{prop:tmap}
For $x\in M$, the tangent map $(d\Psi)_x :
\mathcal{T}_x\cM \to \mathcal{T}_{\psi(x)}\cN$ is injective (resp. surjective) if and only if the maps $(d\psi)_x\colon T_xM \to T_{\psi(x)}N$, $(\psi_1)_x\colon E_1|_x \to F_1|_{\psi(x)} $ and $(\psi_2)_x\colon E_2|_x \to F_2|_{\psi(x)} $ are injective (resp. surjective). 
\end{prop}

By simple diagram chasing, we see that the injectivity of $(\psi_1)_x$ and $(\psi_2)_x$ implies that of $(\widetilde{\psi})_x$, and similarly for surjectivity.

We now show how to use the previous proposition to  derive versions of  the inverse function theorem, and local normal form of submersions and immersions, for $\N$-manifolds of degree 2.

 We call $\Psi$ a \textit{local isomorphism} around $x\in M$ if there
are open neighborhoods of $x$ in $M$ and $\psi(x)$ in $N$, denoted
by $U$ and $V$, so that $\psi:U\to V$ is a diffeomorphism and
$\psi^\sharp \colon C_\cN |_V \to \psi_*(C_\cM |_U)$ is an isomorphism of
sheaves. 

\begin{cor}\label{cor:invfthm}
Let $\Psi=(\psi,\psi^\sharp) \colon \cM\to \cN$ be a morphism and $x_0 \in
M$. Then $(d\Psi)_{x_0} \colon \mathcal{T}_{x_0}\cM \to
\mathcal{T}_{\psi(x_0)}\cN$ is an isomorphism if and only if $\Psi$
is a local isomorphism around $x_0$.
\end{cor}

\begin{proof}
By the equivalence in Thm.~\ref{thm:equivcat}, the condition that $\Psi$
is a local isomorphism around $x_0$ is equivalent to the existence of neighborhoods $U$ of $x_0$ and $V$ of $\psi(x_0)$ such that $\psi:U \to V$ is a diffeomorphism, and $\psi_1\colon E_1|_U \to F_1|_V$ and $\psi_2\colon E_2|_U \to F_2|_V$ are vector bundle isomorphisms. By the inverse function theorem for $\psi$, the existence of such neighborhoods is, in turn, equivalent to $(d\psi)_{x_0}$, $(\psi_1)_{x_0}$ and $(\psi_2)_{x_0}$ being isomorphisms, which is the same as $(d\Psi)_{x_0}$ being an isomorphism by Prop. \ref{prop:tmap}.
\end{proof}

A morphism $\Psi=(\psi,\psi^\sharp) \colon \cM \to \cN$ is a
\textit{submersion at a point $x_0\in M$} if $(d\Psi)_{x_0}:
\mathcal{T}_{x_0}\cM \to \mathcal{T}_{\psi(x_0)}\cN$ is onto, and it is a {\em submersion} if it is a submersion at every point. We will say that $\Psi$ is a {\em surjective submersion} if it is a submersion and $\psi$ is surjective.

\begin{cor}\label{cor:submersion}
Suppose that $\Psi=(\psi,\psi^\sharp) \colon \cM \to \cN$ is a submersion
at $x_0\in M$. Then there exist local coordinates $\{ x^j, e^\nu,
p^J \}$ around $x_0$, and $\{\overline{x}^i, \overline{e}^\mu,
\overline{p}^I \}$ around $\psi(x_0)$, with respect to which $\Psi$
has the form
\begin{equation}\label{eq:sub}
\psi^\sharp \overline{x}^i = \psi^i(x)= x^i,\;\; \psi^\sharp
\overline{e}^\mu = e^\mu,\;\; \psi^\sharp \overline{p}^I = p^I.
\end{equation}
\end{cor}

\begin{proof}
By Prop. \ref{prop:tmap}, the fact that $\Psi$ is a submersion at $x_0$ is the same as the surjectivity of $(d\psi)_{x_0}$, $(\psi_1)_{x_0}$ and $(\psi_2)_{x_0}$. By the usual local normal form of submersions, there are neighborhoods of $x_0$ and $\psi(x_0)$, with coordinates $\{x^j\}$ and $\{\overline{x}^i\}$ satisfying  the first condition in \eqref{eq:sub}. Let us view the coordinates $\{e^\nu\}$ and $\{\overline{e}^\mu\}$ as local frames of $E_1^*$ around $x_0$ and $F_1^*$ around $\psi(x_0)$, respectively, and recall that the map $\psi^\sharp$ in degree 1 is identified with $\psi_1^*\colon \psi^*F_1\to E_1$. We see that the local injectivity of $\psi_1^*$ around $x_0$ (which follows from the surjectivity of $(\psi_1)_{x_0}$) ensures that, given any local frame $\{\overline{e}^\mu\}$, one can extend the set of sections $\{\psi_1^* \overline{e}^\mu\}$ to a local frame $\{ e^\nu\}$ of $E_1^*$, and hence the second condition in \eqref{eq:sub} is satisfied. The same argument using the local injectivity of $\psi_2^*$ shows the analogous result for degree 2 coordinates.
\end{proof}

One can obtain a local normal form for immersions with similar arguments.

%%%%%%%%%%%%%%%%%%%%%%%%%%%%%%%%%

\subsubsection{\underline{Regular values}\label{subsubsec:regvalue}}\
We present a graded analogue of the standard result in differential geometry that the inverse image of a regular value of a smooth map is a submanifold.

Let $\Psi=(\psi,\psi^\sharp):\cM \to \cN$ be a morphism of $\N$-manifolds of degree 2.

Just as in classical geometry, we say that $c
\in N$ is a \textit{regular value} of $\Psi$ if $(d\Psi)_x$ is onto for all $x\in M$ 
with $\psi(x)=c$. 

For $c\in N$, let $\cI_c\subseteq C_\cN$ be the subsheaf of
vanishing ideals defined by $c \in N$,
$$
\cI_c(V)=\{f\in C_\cN(V)\;|\; f_0(c)=0\mbox{ if } c\in V\},  \;\;\;
V\subseteq N \mbox{ open}.
$$
%Denote by $\cI_{\Psi^{-1}(c)}$  the sheaf of ideals  generated by $\psi^\sharp (\cI_c)$ in $C_\cM$.

Denote by 
$$
\cI:=\SP{\psi^\sharp (\cI_c)}\subseteq C_\cM
$$ 
the sheaf of ideals generated by $\psi^\sharp (\cI_c)$ in $C_\cM$.

\begin{prop}\label{prop:invc}
If $c$ is a regular value of $\Psi$,  then the sheaf of ideals $\cI\subseteq C_\cM$ is regular.
\end{prop}

\begin{proof}
As in the classical case, the proof follows from the local form of submersions (Cor.~\ref{cor:submersion}). 

By Prop.~\ref{prop:tmap}, we know that $c\in N$ is a regular value for $\psi$, which implies that, in degree zero, $\cI$ coincides with the vanishing ideal of the submanifold $S=\psi^{-1}(c)\subseteq M$,
$$
\cI_0 = I_S.
$$

Let $x\in S$.
Since $c$ is a regular value for $\Psi$, by Cor.~\ref{cor:submersion} there exist local coordinates $\{ x^j, e^\nu,
p^J \}$ centered at $x$, and $\{\overline{x}^i, \overline{e}^\mu,
\overline{p}^I \}$ centered at $\psi(x)=c$, such that
\begin{equation}\label{eq:subcoor}
\psi^\sharp \overline{x}^i = \psi^* \overline{x}^i= x^i,\;\; \psi^\sharp
\overline{e}^\mu = e^\mu,\;\; \psi^\sharp \overline{p}^I = p^I.
\end{equation}
In a sufficiently small neighborhood of $c$, a section $f$ of $\cI_c$ has the form
$f= a_i \overline{x}^i + b_\mu \overline{e}^\mu + c_I \overline{p}^I$, so
$$
\psi^\sharp f = (\psi^*a_i) x^i + (\psi^*b_\mu) e^\mu + (\psi^*c_I)p^I.
$$
It follows that, around $x\in S$, $\cI$ is generated by the coordinate functions $\{x^i, e^\mu, p^I\}$, and therefore  $\cI$ is regular.
\end{proof}

By Prop.~\ref{prop:1-1}, $\cI=\SP{\psi^\sharp (\cI_c)}\subseteq C_\cM$ is the sheaf of vanishing ideals of a submanifold of $\cM$ that we denote  by 
$$
\Psi^{-1}(c).
$$

To obtain the geometric characterization of $\Psi^{-1}(c)$ in terms of vector bundles as in $\S$ \ref{subsubsec:ideal}, we keep the notation of the previous subsection, with $\cM$ corresponding to $(E_1, \widetilde{E}, \phi_E)$ and $\cN$ to $(F_1, \widetilde{F}, \phi_F)$,
and observe the following immediate consequence of Prop.~\ref{prop:tmap}.

\begin{lemma} \label{lem:regular}
A point $c \in N$ is a regular value of $\Psi\colon \cM \to \cN$ if and only if
\begin{itemize}
    \item[(a)] $c$ is a regular value for $\psi\colon M\to N$,
    \item[(b)] the maps  $\psi_i\colon E_i|_{\psi^{-1}(c)}\to F_i|_c$, $i=1,2$, are fiberwise surjective. (Equivalently, for each $x\in \psi^{-1}(c)$, the maps $(\psi_i)^*_x\colon F_i^*|_c\to E_i^*|_x$, $i=1,2$, are injective.)
\end{itemize}
\end{lemma}

As mentioned after Prop. \ref{prop:tmap}, the second condition above implies that the map 
$$
\widetilde{\psi}\colon \widetilde{E}|_{\psi^{-1}(c)}\to \widetilde{F}|_c
$$ 
is also fiberwise surjective. 

Given a regular value $c$ of $\Psi$, we can therefore consider
\begin{itemize}
\item the submanifold $S=\psi^{-1}(c) \subseteq N$, and
\item the subbundles $K_1 \subseteq E_1^*$ and ${K'} \subseteq \widetilde{E}^*$ over $S$, given by 
$$
K_1|_x = (\psi_1)_x^*(F_1^*|_c), \quad \mbox{ and } \quad {K'}|_x = (\widetilde{\psi})_x^*(\widetilde{F}^*|_c).
$$
\end{itemize}

Using the special coordinates \eqref{eq:subcoor} around points in the submanifold $\iota: S \hookrightarrow M$, we see that
$$
\Gamma_{K_1}= \iota^* \cI_1, \qquad \Gamma_{\widetilde{K}}=\iota^*\cI_2,
$$
where $\widetilde{K}$ is the vector subbundle of $\widetilde{E}^*|_S$ given by
\begin{equation*}\label{eq:regvalueKt}
\widetilde{K} := K' + K_1\wedge E_1^*|_S. 
\end{equation*}

Following $\S$ \ref{subsubsec:ideal} (see Lemma~\ref{lem:geomideal}), 
and with the notation introduced at the end of Section~\ref{subsec:notation}, 
we conclude

\begin{prop}\label{prop:invcgeom}
Let $c$ be a regular value of $\Psi$. Then the sheaf of vanishing ideals $\cI=\SP{\psi^\sharp (\cI_c)}\subseteq C_\cM$ of the submanifold $\Psi^{-1}(c)$ is determined by the vector bundles $K_1$ and $\tilde{K}$ via
$$
\cI_0 = \fl_S, \quad \cI_1 = \Gamma_\st{E_1^*,K_1}, \; \mbox{ and } \;\; \cI_2 =
\Gamma_\st{\tilde{E}^*,\widetilde{K}}.
$$
\end{prop}

%%%%%%%%%%%%%%%%%%%%%%%%%%%%%%%%%%%%%%%%%%%%%%%
\subsection{Distributions and the Frobenius theorem}\label{subsec:frobenius}

Let $\cM$ be an $\N$-manifold of degree 2 of dimension $m_0|m_1|m_2$. A
\textit{distribution} of rank $d_0|d_1|d_2$ is a graded subsheaf of
$C_\cM$-modules $\mathcal{D}$ of $\mathcal{T}\cM$ satisfying the
following local property: around any point in $M$ there is a
neighborhood $U$ such that $\mathcal{D}|_U$ is generated by vector
fields $X_1,\ldots,X_{d_0}$ of degree $0$, $Y_1,\ldots,Y_{d_1}$ of
degree $-1$, and $Z_1,\ldots,Z_{d_2}$ of degree $-2$, and such that
for each $x\in U$,
$$
(X_1)_x,\ldots,(X_{d_0})_x, (Y_1)_x,\ldots,(Y_{d_1})_x,
(Z_1)_x,\ldots, (Z_{d_2})_x
$$
are linearly independent elements of $\mathcal{T}_x\cM$.
 
A distribution $\mathcal{D}$ is \textit{involutive} if its spaces of
local sections are closed under the Lie bracket on $\mathcal{T}\cM$.
We have the following generalization of the Frobenius
theorem.

\begin{thm}\label{thm:frob}
Let $\mathcal{D}$ be an involutive distribution of rank
$d_0|d_1|d_2$ on $\cM$. Then any point in $M$ has a neighborhood $U$
with local coordinates $\{ x^i,e^\mu,p^I\}$ such that
$\mathcal{D}|_U$ is spanned by
$$
\frac{\partial}{\partial x^1},\ldots,\frac{\partial}{\partial
x^{d_0}},\frac{\partial}{\partial
e^1},\ldots,\frac{\partial}{\partial
e^{d_1}},\frac{\partial}{\partial
p^1},\ldots,\frac{\partial}{\partial p^{d_2}}.
$$
\end{thm}

%%%%

This is proven in \cite{HenMiqRaj} (see, e.g.,  \cite{fiori,varadarajan} for the case of supermanifolds).

%%%%%%%%%%%%%%%%%%%%%%%%%%%%%%%%%%%%%%%%%%%%%%%%%%%%%%%%%%%%%%%%%%%%%%
\section{Symplectic $\N$-manifolds of degree $2$}\label{sec:symplectic}

\subsection{Poisson brackets} 
Let $\cM$ be an $\N$-manifold of degree 2 of dimension $m_0|m_1|m_2$. A
\textit{Poisson structure} on $\cM$ of degree $q$ is an
$\mathbb{R}$-bilinear operation
$$
\{\cdot,\cdot\}:C_\cM\times C_\cM \to C_\cM
$$
such that, for each $U\subseteq M$ open and homogeneous elements
$f\in (C_\cM(U))_k$, $g\in (C_\cM(U))_l$, $h\in (C_\cM(U))_m$, the
following holds:
\begin{itemize}
\item[(br1)] $\{f,g\}\in (C_\cM(U))_{k+l+q}$,
\item[(br2)] $\{f,g\}=-(-1)^{(k+q)(l+q)}\{g,f\}$,
\item[(br3)] $\{f,gh\} = \{f,g\}h + (-1)^{(k+q)l}g\{f,h\}$,
\item[(br4)] $\{f,\{g,h\}\}=\{\{f,g\}, h\}+(-1)^{(k+q)(l+q)}\{g,\{f,h\}\}$
\end{itemize}
 
In particular, any global section $f\in C(\cM)_k$ defines a vector
field on $\cM$,
$$
X_f := \{f, \cdot\},
$$ 
 of degree $k + q$, called the
\textit{hamiltonian vector field} of $f$.  A simple consequence of (br4) is that, for $f\in C(\cM)_k$ and $g\in
C(\cM)_l$, we have
\begin{equation}\label{eq:commutator}
X_{\{f,g\}} = [X_f, X_g].
\end{equation}

For each $x\in M$ and $f\in C_\cM|_x$ of degree $k$, there is
similarly an element $(X_f)_x \in \mathcal{T}_x\cM$ of degree $k+q$
defined by
$$
(X_f)_x(g) = \{f,g\}_0(x), \;\;\, \mbox{ for } g \in C_\cM|_x,
$$
where we keep the notation as in \eqref{eq:tangent}. 
It follows from (br2) and (br3) that if $f\in \cI_{(x)}^2$ (see
\eqref{eq:I(x)}), then $(X_f)_x=0$. So there is an induced
degree-preserving map
\begin{equation}\label{eq:sharp}
(\mathcal{T}_x\cM)^*\cong \cI_{(x)}/\cI_{(x)}^2 \longrightarrow
T_x\cM [q],\; [f]\mapsto (X_f)_x.
\end{equation}
Here ``$[q]$'' denotes the degree shift by $q$, as recalled in $\S$ \ref{subsec:notation}.
We say that a Poisson bracket $\{\cdot,\cdot\}$ is \textit{non-degenerate} 
when the
map \eqref{eq:sharp} is an isomorphism for all $x\in M$.
 
 We will be interested in Poisson brackets of degree $q=-2$.

\begin{prop}\label{prop:nondegpoisson}
Let $\cM$ be equipped with a Poisson bracket $\{\cdot,\cdot\}$ of degree $q=-2$, and 
let $\{x^i,e^\mu,p^I\}$ be local coordinates on $\cM$ around $x\in
M$. Then the map \eqref{eq:sharp} is an isomorphism if 
and only if $m_0=m_2$ and
\begin{equation}\label{eq:nodeg}
\mathrm{det}(\{p^I,x^j\}(x))\neq 0,\;\; \mathrm{det}(\{e^\mu,e^\nu\}(x)) \neq 0.
\end{equation}
\end{prop}

\begin{proof}
Relative to the bases $[x^i], [e^\mu], [p^I]$ of
$\cI_{(x)}/\cI_{(x)}^2$ and \eqref{eq:derx} of $\mathcal{T}_x\cM$,
with $i=1,\ldots,m_0$, $\mu=1,\ldots,m_1$, $I=1,\ldots,m_2$, the
linear map \eqref{eq:sharp} is given by the matrix
\begin{equation}\label{eq:matrix}
\left ( \begin{matrix} \{x^i,x^j\}_0(x) & \{e^\mu,x^j\}_0(x)&\{p^I,x^j\}_0(x)\\
\{x^i,e^\nu\}_0(x) & \{e^\mu,e^\nu\}_0(x)&\{p^I,e^\nu\}_0(x)\\
\{x^i,p^J\}_0(x) & \{e^\mu,p^J\}_0(x)&\{p^I,p^J\}_0(x)
\end{matrix} \right ),
\end{equation}
which, for $q=-2$, takes the form
$$
\left ( \begin{matrix} 0 & 0 &\{p^I,x^j\}(x)\\
0 & \{e^\mu,e^\nu\}(x)& 0\\
\{x^i,p^J\}(x) & 0 & 0
\end{matrix} \right ).
$$
This matrix is invertible if and only if $m_0=m_2$
and \eqref{eq:nodeg} holds.
\end{proof}

\begin{defi}We will refer to an $\N$-manifold of degree 2 equipped with a
non-degenerate Poisson bracket of degree $-2$ as a
\textit{symplectic} $\N$-manifold of degree 2 (or as a degree 2 symplectic $\N$-manifold). 
A {\em symplectomorphism} between degree 2 symplectic manifolds is an isomorphism of $\N$-manifolds of degree 2 that preserves the Poisson brackets.
\end{defi}

\begin{cor}\label{cor:LI} Let $\cM$ be a symplectic $\N$-manifold of degree 2 and
$\{x^i,e^\mu,p^I\}$ be local coordinates. Then the local vector
fields $\{x^i,\cdot\}$, $\{e^\mu,\cdot \}$, $\{ p^I,\cdot\}$ are
linearly independent.  
\end{cor}

\begin{proof}
The nondegeneracy conditions in \eqref{eq:nodeg} imply that, for every $x\in M$ in the domain of the local coordinates, the
tangent vectors $(X_I)_x := \{p^I,\cdot\}(x)$, $(Y_\mu)_x:=\{ e^\mu,
\cdot\}(x)$, $(Z_i)_x:= \{ x^i,\cdot\}(x)$ are linearly independent
in $\mathcal{T}_x\cM$ (see \eqref{eq:lix1}, \eqref{eq:lix2},
\eqref{eq:lix3}). The result now follows from Proposition
~\ref{prop:li}.
\end{proof}

%%%%%%%%%%%%%%%%%%%%%%%%%%%%%%%%%%%%%%%%%%%%%%%%%%%%%%%%%%%%%%%%%%%%%%%%%%%%
\subsection{Equivalence with pseudo-euclidean vector bundles}\label{subsec:pseudo}

Let $E\to M$ be a vector bundle. A \textit{pseudo-euclidean}
structure on $E$ is a fiberwise bilinear pairing $\SP{\cdot,\cdot}$
which is symmetric and non-degenerate.

Following
\cite[Thm.~3.3]{royt:graded}, pseudo-euclidean vector bundles are
``equivalent'' to symplectic  $\N$-manifolds of degree 2. For a precise formulation of this statement, consider the category $\mathrm{PsEuc}$, whose objects are pseudo-euclidean vector bundles and morphism are vector-bundle isomorphisms preserving the pseudo-euclidean structures, and let $\mathrm{Symp2Man}$ be the category of degree 2 symplectic  $\N$-manifolds, with morphisms given by symplectomorphisms. 

\begin{thm}[cf. \cite{royt:graded}, Thm.~3.3]\label{thm:pseudoeuc}
There is an equivalence of categories 
$$
\mathrm{PsEuc} \stackrel{\sim}{\to} \mathrm{Symp2Man}.
$$
\end{thm}

This yields, in particular, a bijective correspondence between isomorphism classes of pseudo-euclidean vector bundles and symplectic 
$\N$-manifolds of degree 2.

For illustrative purposes, we will show how 
this well-known result relates to the geometric description of $\N$-manifolds of degree 2 in Theorem~\ref{thm:equivcat}. Indeed, as we will now see, 
there is a natural faithful functor
\begin{equation}\label{eq:faithfunc}
\mathcal{A}: \mathrm{PsEuc} \to \VB2,
\end{equation}
such that the equivalence in Thm~\ref{thm:pseudoeuc} can be viewed as a canonical ``lift''
of the equivalence functor $\mathcal{F}$ in \eqref{eq:F},
\begin{equation}\label{eq:compdiag}
\xymatrix{
\mathrm{PsEuc} \ar[r] \ar[d]_{\mathcal{A}} & \mathrm{Symp2Man} \ar[d] \\
\VB2 \ar[r]_{\mathcal{F}} & \NMan,
}
\end{equation}
where the vertical map on the right is the forgetful functor $ \mathrm{Symp2Man}\to \NMan$.

Recall that any
pseudo-euclidean vector bundle $(E, \SP{\cdot,\cdot})$ over $M$ has
an associated \textit{Atiyah Lie algebroid}
$$
\mathbb{A}_E \to M,
$$
see, e.g., \cite{cannasweinstein,mac:lie} (following \cite{AtiyahAlgebroid}).

As a vector bundle, $\mathbb{A}_E$ is characterized by the property that its sheaf of sections  $\Gamma_{\mathbb{A}_E}$ agrees with the sheaf of {\em derivations} (also called {\em covariant differential operators}) of $E$ that preserve the pseudo-euclidean metric. I.e.,
 for each open subset $U\subseteq M$,
elements in $\Gamma_{\mathbb{A}_E}(U) = \Gamma(\mathbb{A}_E|_U)$ are pairs $(X,D)$, where
$X\in \mathfrak{X}(U)$ and $D\colon \Gamma(E|_U)\to
\Gamma(E|_U)$ is an $\mathbb{R}$-linear endomorphism satisfying
\begin{align}\label{eq:CDO}
D(f e) &= (\pounds_X f) e + f D(e),\\
\pounds_X\SP{e,e'}&=\SP{D(e),e'}+\SP{e,D(e')},\label{eq:CDOpairing}
\end{align}
for $f\in C^\infty(U)$, $e,e' \in \Gamma(E|_U)$.
 The Lie bracket on
$\Gamma_{\mathbb{A}_E}$ is given by commutators and the anchor $\sigma$ of $\mathbb{A}_E$ (also called the {\em symbol map}) is the projection
$(X,D)\mapsto X$. 
The Lie algebroid $\mathbb{A}_E$ is transitive and
fits into the exact sequence
\begin{equation}\label{eq:atiyahseq}
0\longrightarrow \mathfrak{so}(E)\cong \wedge^2E^* \longrightarrow
\mathbb{A}_E \stackrel{\sigma}{\longrightarrow} TM \longrightarrow 0,
\end{equation}
where $\mathfrak{so}(E) \subseteq \mathrm{End}(E)$ is the space of
skew-symmetric endomorphisms
$$
T:E\to E,\;\; \SP{Te,e'}=-\SP{e,Te'},
$$
that is naturally identified with $\wedge^2 E^*$. 

One refers to \eqref{eq:atiyahseq} as the  \textit{Atiyah sequence} of $(E,\SP{\cdot,\cdot})$.

We now define the functor $\mathcal{A}$ in \eqref{eq:faithfunc}:
At the level of objects, it  assigns to
 a pseudo-euclidean vector bundle $(E, \SP{\cdot,\cdot})$ the object $(E_1, \widetilde{E}, \phi_E)$ in $\VB2$ given by
\begin{equation}\label{eq:pseuVB2}
E_1=E, \quad\;\; \widetilde{E}= \mathbb{A}_E^*, \quad \;\; \phi_E\colon \widetilde{E}\twoheadrightarrow \wedge^2 E_1,
\end{equation}
where $\phi_E$ is the map dual to the inclusion $ \wedge^2E^* \to
\mathbb{A}_E$ in \eqref{eq:atiyahseq}; At the level of morphisms, any isomorphism of pseudo-euclidean vector bundles (covering a diffeomorphism of the base manifolds) canonically induces an isomorphism of the corresponding Atiyah sequences, and hence an isomorphism between the corresponding objects in $\VB2$.

The next result will relate the functor $\mathcal{A}$ to symplectic structures in degree 2.

\begin{lemma}\label{lem:symp}
Let $\cM$ be an $\N$-manifold of degree 2
corresponding to an object  $(E_1, \widetilde{E}, \phi_E)$ in $\VB2$.
A symplectic structure on $\cM$
is equivalent to the following additional data  
on $(E_1, \widetilde{E}, \phi_E)$:
\begin{itemize}
\item a pseudo-euclidean metric on $E_1^*$ (or, equivalently, on $E_1$), 
\item a vector-bundle isomorphism $\widetilde{E}^* \cong \mathbb{A}_{E_1}$ commuting with the inclusions $\wedge^2 E_1^* \to \widetilde{E}^*$ and $\wedge^2 E_1^* \to \mathbb{A}_{E_1}$.
\end{itemize}
\end{lemma}

\begin{remark}
A general Poisson bracket of degree $-2$ (not necessarily symplectic) on $\cM$ gives rise to a (possibly degenerate) symmetric pairing on $E_1^*$, a
Lie-algebroid structure on $\widetilde{E}^*$ and a  morphism from $\widetilde{E}^*$ to the Lie algebroid of derivations of $E_1^*$ that preserve the pairing, see \cite[Thm.~6.1]{delCarpioThesis}. The nondegeneracy condition in the symplectic case makes this morphism into an isomorphism, leading to the simplified formulation of Lemma~\ref{lem:symp}. \hfill $\diamond$
\end{remark}

\begin{proof}
Let $\cM$ be equipped with a symplectic structure given by the Poisson
bracket
$$
\{\cdot,\cdot\} \colon (C_{\cM})_k\times (C_{\cM})_l\to (C_{\cM})_{k+l-2}.
$$
We will obtain the data described in the lemma by restricting  this bracket to functions of degrees 0, 1 and 2.

The restriction to degree 1 functions,
$$
\{\cdot,\cdot \} \colon \Gamma_{E_1^*}\times \Gamma_{E_1^*}\to
C^\infty_M,
$$
is symmetric and $C^\infty_M$-linear by properties
(br2) and (br3), and non-degenerate by the
second condition in \eqref{eq:nodeg}. So it defines a pseudo-euclidean structure on $E_1^*$, or equivalently on $E_1$, through the identification of $E_1$ with $E_1^*$ via the pseudo-euclidean metric.

The Poisson brackets of functions in degrees 2 and 0 define a
map
$$
\Gamma({\widetilde{E}^*})\to \mathfrak{X}(M),\;\;
\widetilde{e}\mapsto \left(\{\widetilde{e},\cdot\}:C^\infty(M)\to
C^\infty(M)\right),
$$
which is  $C^\infty(M)$-linear, i.e.,
it comes from a vector-bundle map
%\begin{equation}\label{eq:anchorpsi}
$$
\rho\colon \widetilde{E}^*\to TM.
$$
%\end{equation}
The Poisson brackets of functions in degrees 2 and 1 now lead to a map
%\begin{equation}\label{eq:atiyahid}
$\Gamma({\widetilde{E}^*}) \to \Gamma({\mathbb{A}_{E_1^*}})$, 
$\widetilde{e}\mapsto (X, D)$,
%\end{equation}
where $X=\{\widetilde{e},\cdot\} \colon C^\infty(M)\to C^\infty(M)$ and
\begin{equation}\label{eq:XD}
%X=\{\widetilde{e},\cdot\} \colon C^\infty(M)\to C^\infty(M),\;\; 
D= \{\widetilde{e},\cdot\}:\Gamma(E_1^*)\to \Gamma(E_1^*).
\end{equation}
It is a direct verification that this map comes from a vector-bundle map $\tilde{E}^* \to \mathbb{A}_{E_1^*}$. By means of the identification $E_1\cong E_1^*$, we alternatively obtain a map
\begin{equation}\label{eq:deg2id}
\tilde{E}^* {\to} \mathbb{A}_{E_1}.
\end{equation}
It follows from condition (br3) that this map commutes with the inclusions $\wedge^2 E_1^* \to \tilde{E}^*$ and $\wedge^2 E_1^* \to \mathbb{A}_{E_1}$.
Finally, the first nondegeneracy condition in \eqref{eq:nodeg} implies 
the exactness of the sequence 
\begin{equation}\label{eq:exact1}
0 {\longrightarrow} \wedge^2E^* \stackrel{\phi_E^*}{\longrightarrow}  \widetilde{E}^*
\stackrel{\rho}{\longrightarrow} TM \longrightarrow 0.
\end{equation}
The fact that the morphism \eqref{eq:deg2id} commutes with the inclusions and projections of the exact sequences \eqref{eq:exact1} and \eqref{eq:atiyahseq} implies that it is an isomorphism. This shows that a symplectic structure on $\cM$ gives rise to the data in the lemma.

For the converse, we reverse the arguments and use the additional data on $(E_1, \tilde{E}, \phi_E)$ to define the nontrivial Poisson-bracket relations involving functions on $\cM$ of degrees 0, 1 and 2.  (Note that the Poisson bracket on degree 2 functions is identified with the Lie bracket on $\Gamma_{\mathbb{A}_{E_1}}$.) Since $C_\cM$ is locally generated by such functions, it follows from  (br3) that there is a unique Poisson bracket on $C_\cM$ of degree $-2$ satisfying these relations, which is moreover non-degenerate. 
\end{proof}

Let $(E, \SP{\cdot,\cdot})$ be a pseudo-euclidean vector bundle. Its image in $\VB2$
under $\mathcal{A}$ is given in \eqref{eq:pseuVB2}, so the corresponding $\N$-manifold $\cM = \mathcal{F}(\mathcal{A}(E, \SP{\cdot,\cdot}))$ satisfies
\begin{equation}\label{eq:Mdeg12}
(C_\cM)_1 = \Gamma_E,\qquad (C_\cM)_2 = \Gamma_{\mathbb{A}_E}.
\end{equation}
It follows from Lemma~\ref{lem:symp} that $\cM$ is equipped with a natural symplectic structure which, as shown in the proof of that lemma, is determined by the following bracket relations:
\begin{equation}\label{eq:brkrel}
\begin{cases}
\{e_1,e_2\}&=\SP{e_1,e_2}, \\
\{\widetilde{e},f\}&=\pounds_Xf,  \\ 
\{\widetilde{e},e\}&=D(e),  \\
\{\widetilde{e}_1, \widetilde{e}_2\}&=([X_1,X_2],[D_1,D_2]), 
\end{cases}
\end{equation}
for $f\in C^\infty(M)$, $e, e_i \in \Gamma(E)$, $\widetilde{e}=(X,D), \widetilde{e}_i=(X_i,D_i) \in \Gamma(\mathbb{A}_E)$, $i=1,2$.

We can now complete the proof of the theorem.

\begin{proof}[Proof of Theorem~\ref{thm:pseudoeuc}]
Consider the functor given by the composition of \eqref{eq:faithfunc} with \eqref{eq:F},
$$
\mathcal{F}\circ \mathcal{A}: \mathrm{PsEuc} \to \NMan.
$$ 
We will check that it admits a lift with respect to the forgetful functor $ \mathrm{Symp2Man}\to \NMan$, as indicated in \eqref{eq:compdiag}.

Given a pseudo-euclidean vector bundle $(E, \SP{\cdot,\cdot})$, 
we saw above that 
$$
\cM = \mathcal{F}(\mathcal{A}(E, \SP{\cdot,\cdot}))
$$
carries a natural symplectic structure determined by \eqref{eq:brkrel}.
Moreover,  any isomorphism $T: E\to E'$ of pseudo-euclidean vector bundles is assigned to an isomorphism $\mathcal{F}(\mathcal{A}(T))$ between the corresponding $\N$-manifolds of degree 2 that naturally preserves
 the bracket relations \eqref{eq:brkrel}, so it is a symplectomorphism.
We conclude that $\mathcal{F}\circ \mathcal{A}$ admits a canonical lift to a  fully-faithful functor 
  $$
 \widehat{\cF}:   \mathrm{PsEuc}\to \mathrm{Symp2Man}.
  $$ 

Since $\mathcal{F}$ is essentially surjective, any object in $\mathrm{Symp2Man}$ is symplectomorphic to one of the form $(\cM, \{\cdot,\cdot\})$, with
$\cM=\mathcal{F}(E_1,\widetilde{E},\phi_E)$. Lemma~\ref{lem:symp} shows that the symplectic structure on $\cM$ induces a pseudo-euclidean metric $\SP{\cdot,\cdot}$ on $E_1$ in such a way that $\cM$ is symplectomorphic to $\mathcal{F}(\mathcal{A}(E,\SP{\cdot,\cdot}))$ with its natural symplectic form, i.e., to an object in the image of $\widehat{\cF}$, so this functor is essentially surjective.
\end{proof}

\begin{ex}\label{ex:cotangent}
Given a vector bundle $A\to M$, one can define a pseudo-euclidean vector bundle $E:=A\oplus A^*$, with metric
$$
\SP{(a_1,\xi_1),(a_2, \xi_2)}=\xi_2(a_1)+ \xi_1(a_2).
$$
The corresponding degree 2 symplectic $\mathbb{N}$-manifold  is denoted by $T^*[2]A[1]$  and plays the role of the cotangent bundle of $A[1]$ (as in Example~\ref{ex:1Man}) in the context of $\mathbb{N}$-manifolds of degree 2. The projection $T^*[2]A[1]\to A[1]$ is defined
by the natural inclusion
$$
\kappa^\sharp\colon C_{A[1]} \to C_{T^*[2]A[1]},
$$
given by the identity map in degree 0 and 
$$
(C_{A[1]})_1=\Gamma_{A^*} \hookrightarrow \Gamma_{0\oplus A^*}\subseteq \Gamma_{A\oplus A^*} =  (C_{T^*[2]A[1]})_1.
$$
For graded cotangent bundles in higher degrees, see \cite{Cueca21}.
\hfill $\diamond$
\end{ex}

%%%%%%%%%%%%%%%%%%%%%%%%%%%%%%%%%%%%%%%%%%%%%%%%%%%%%%%%%%%%%%%%%%%%%%%%%%%%
\section{Coisotropic submanifolds}\label{subsec:coisosubm}

For a symplectic manifold $M$, a submanifold $ N\hookrightarrow M$ is called {\em coisotropic} if its sheaf of vanishing ideals is closed under the Poisson bracket,
$$
\{I_N,I_N\}\subseteq I_N,
$$
or, equivalently, if $TN^\omega\subseteq TN$, where $TN^\omega$ is the symplectic orthogonal to $TN$. A coisotropic submanifold $N$ carries a foliation with leaves tangent to $TN^\omega$, and a (local) function on $N$ is called {\em basic} if it is constant along the leaves. An important property of $N$ is that the sheaf of basic functions $(C^\infty_N)_{bas}$ inherits a natural Poisson bracket given by
$$
\{f,g\}_{bas}= \{\tilde{f},\tilde{g}\}|_N,
$$
where $f,g$ are (local) basic functions and $\tilde{f}$, $\tilde{g}$ are arbitrary (local) extensions to $M$. In this section we will present versions of these results for symplectic $\mathbb N$-manifolds of degree 2.

\begin{remark}
The previous discussion, as well as the results in this section, can be extended to the broader class of submanifolds $\iota\colon N \hookrightarrow M$ for which $TN\cap TN^\omega$, which coincides with the
kernel distribution of $\iota^*\omega$, has constant rank. We will restrict ourselves to coisotropic submanifolds for simplicity. \hfill $\diamond$
\end{remark}

\subsection{Two viewpoints}\label{subsec:twoviews}
Let $\cM$ be a symplectic  $\N$-mani\-fold of degree 2, with Poisson
bracket $\{\cdot,\cdot\}$. If $(\iota,\iota^\sharp):\cN\to \cM$ is a 
submanifold with sheaf of vanishing ideals $\cI\subseteq
C_\cM$, we denote by $\mathfrak{N}_\cI \subseteq C_\cM$ its sheaf of
\textit{Lie normalizers}, given for each open subset $U\subseteq M$ by
\begin{equation}\label{eq:normalizer}
\mathfrak{N}_\cI(U) := \{ f\in C_\cM(U)\,|\, \{f,\cI(U)\}\subseteq
\cI(U)  \}.
\end{equation}
Note that local sections of
$\mathfrak{N}_\cI$ are closed under the Poisson bracket, so
$\mathfrak{N}_\cI$ is a sheaf of Poisson algebras.

 We say that $\cN$ is \textit{coisotropic} if
$\cI\subseteq \mathfrak{N}_\cI$,;
i.e., if
\begin{equation}\label{eq:coiso}
\{\cI,\cI\}\subseteq \cI.
\end{equation}

%%%%%%%%%%%%%%%%%%%%%%%%%%%%%%%%%%%%%%%%%%%%%%%%%%%%%%%%%%%%%%%%%%%%

 One can also approach coisotropic submanifolds through the usual geometric condition in terms of symplectic orthogonals, 
 ``$\cT\cN^\omega \subseteq \cT \cN$''. 
 To make sense of the symplectic orthogonal distribution $\cT\cN^\omega$, we  follow the classical viewpoint that regards it as the distribution locally generated by hamiltonian vector fields corresponding to functions in $\cI$.

Let $\mathcal{T}_{\cI^\omega}\subseteq \mathcal{T}\cM$ be the
subsheaf of $C_\cM$-modules given, for each open subset $U\subseteq M$, by vector fields $X\in \cT\cM(U)$ with the property that any $x \in U$ admits a neighborhood where $X$ can be written as a combination of hamiltonian vector fields $\{f,\cdot\}$, with $f$ a section of $\cI$ in that neighborhood.
Consider the sheaf of
$C_\cN$-modules given by
\begin{equation}\label{eq:TNw}
\mathcal{T}\cN^\omega := \iota^* \mathcal{T}_{\cI^\omega}\subseteq
\iota^* \mathcal{T}\cM.
\end{equation}

Let $\mathcal{T}_\cI \subseteq \mathcal{T}\cM$ be the subsheaf
defined in \eqref{eq:TI}.

\begin{lemma}\label{lem:taucond}
The submanifold $\cN$ is coisotropic if and only if
$\mathcal{T}_{\cI^\omega}\subseteq \mathcal{T}_\cI$.
\end{lemma}

\begin{proof}
Suppose that $\mathcal{T}_{\cI^\omega}\subseteq \mathcal{T}_\cI$,
and let $f\in \cI(U)$, $U\subseteq M$ open. Then $\{f,\cdot\} \in
\mathcal{T}_{\cI^\omega}(U)\subseteq \mathcal{T}_\cI(U)$, which
implies that $\{f,\cI(U)\}\subseteq \cI(U)$. Hence $\cI(U)\subseteq
\mathfrak{N}_\cI(U)$ for all open $U$, i.e., $\cN$ is coisotropic.

Conversely, let $X\in \mathcal{T}_{\cI^\omega}(U)$. To show that
$X\in \mathcal{T}_\cI(U)$, it suffices to verify that any $x\in U$
admits an open neighborhood $W$ in $U$ for which $X|_W\in
\mathcal{T}_\cI(W)$. By the definition of
$\mathcal{T}_{\cI^\omega}$, one can always find such $W$ where
$$
X|_W = a^l \{f_l,\cdot\},
$$
for $f_l\in \cI(W)$ and $a^l \in C^{\infty}(W)$; then the condition $\{\cI,\cI\}\subseteq \cI$
implies that $X|_W\in \mathcal{T}_\cI(W)$.
\end{proof}

\begin{lemma}\label{lem:tauIwlocal}
Let $\{x^i,e^\mu,p^I\}$ be local coordinates on an open subset
$U\subseteq M$ adapted to $\cN$. Then 
$$
\mathcal{T}_{\cI^\omega}(U) =
\Span_{C_\cM(U)}\{\{x^i,\cdot\},\{e^\mu,\cdot\},\{p^I,\cdot\}\},
$$
for $i=1,\ldots,r_0$, $\mu=1,\ldots,r_1$, $I=1,\ldots,r_2$, where
$r_0|r_1|r_2 = \codim(\cN)$.
\end{lemma}

\begin{proof}
Any point in $U$ admits a neighborhood $W$ such that
$$
\mathcal{T}_{\cI^\omega}(W) =
\Span_{C_\cM(W)}\{\{x^i,\cdot\},\{e^\mu,\cdot\},\{p^I,\cdot\}\},
$$
for $1\leq i\leq r_0$, $1\leq \mu\leq r_1$, $1\leq I\leq r_2$, since
such $x^i, e^\mu, p^I$ are generators of $\cI(U)$. Let $\{W_k\}$ be an open cover of $U$ such that $\mathcal{T}_{\cI^\omega}(W_k) $ is given as above. Let $X \in \mathcal{T}_{\cI^\omega}(U)$. Then, for each $k$, $X|_{W_k} =
a^i_k\{x^i,\cdot\} + b^\mu_k\{e^\mu,\cdot\}+ c^I_k\{p^I,\cdot\}$
for unique sections $a^i_k$, $b^\mu_k$, $c^I_k$ in $C_\cM(W_k)$, see Corollary~\ref{cor:LI}.
It also follows that, on overlaps  $W_l\cap W_k$, we have that  
 $a^i_l=a^i_k$,
$b^\mu_l=b^\mu_k$ and $c^I_l=c^I_k$. So by the
gluing property of sheaves we can find $a^i$, $b^\mu$, $c^I$   such
that $X= a^i\{x^i,\cdot\} + b^\mu\{e^\mu,\cdot\}+ c^I\{p^I,\cdot\}$.
\end{proof}

\begin{prop}\label{prop:distr}
A submanifold $\cN\to \cM$ is coisotropic if and only if 
$$
\mathcal{T}\cN^\omega\subseteq \mathcal{T}\cN. 
$$ 
In this case,
$\mathcal{T}\cN^\omega$ is an involutive distribution of
$\mathcal{T}\cN$ with rank $r_0 | r_1|r_2 = \mathrm{codim}(\cN)$.
\end{prop}

The distribution $ \mathcal{T}\cN^\omega\subseteq \mathcal{T}\cN$ is
called the \textit{null distribution} of the coisotropic
submanifold $\cN$.

\begin{proof}

By Lemma~\ref{lem:taucond}, the first assertion amounts to checking
that
$$
\mathcal{T}_{\cI^\omega}\subseteq \mathcal{T}_\cI \iff \cT\cN^\omega \subseteq \cT \cN.
%\iota^*\mathcal{T}_{\cI^\omega}\subseteq \iota^*\mathcal{T}_\cI.
$$

Note that the restriction morphism \eqref{eq:rsharp} maps 
$\mathcal{T}_{\cI^\omega}$ to $\iota_*\cT\cN^\omega$. Since it maps $\mathcal{T}_\cI$ to $\iota_*\cT\cN$, we see that $\mathcal{T}_{\cI^\omega}\subseteq \mathcal{T}_\cI $ implies that 
$\iota_*\cT\cN^\omega \subseteq \iota_*\cT \cN$, and therefore that $\cT\cN^\omega \subseteq \cT \cN$.
For the reverse direction, 
let $X$ be a section of $\mathcal{T}_{\cI^\omega}$ over a small open subset $U\subseteq M$.
Assuming that $\cT\cN^\omega \subseteq \cT \cN$, the  image of $X$ under \eqref{eq:rsharp} lies in
$\iota_* \cT\cN(U)$, and hence there is $Y\in
\mathcal{T}_{\cI}(U)$ such that $X-Y\in \cI . \mathcal{T}\cM(U)$,
see \eqref{eq:exseq0}. Since $\cI .
\mathcal{T}\cM(U)\subseteq \mathcal{T}_{\cI}(U)$, it follows that
$X\in \mathcal{T}_{\cI}(U)$.

  For the second part of the proposition, consider local coordinates of $\cM$
over $U$ adapted to $\cN$, and let $V=U\cap N$. 
Lemma~\ref{lem:tauIwlocal} implies that
$$
\mathcal{T}\cN^\omega |_V = \Span_{C_\cN|_V}\{X_I,Y_\mu,Z_i\},
$$
for $i=1,\ldots,r_0$, $\mu=1,\ldots,r_1$, $I=1,\ldots,r_2$, where
$X_I$, $Y_\mu$ and $Z_i$ are the images of the local sections
$\{p^I,\cdot\}$, $\{e^\mu,\cdot\}$ and $\{x^i,\cdot\}$ of
$\mathcal{T}_{\cI^\omega}|_U\subseteq \mathcal{T}_\cI|_U$ under \eqref{eq:rest1},
respectively. We have already observed that $\{p^I,\cdot\}(x),
\{e^\mu,\cdot\}(x)$ and $\{x^i,\cdot\}(x)$ are linearly independent
in $\mathcal{T}_x\cM$ for all $x\in U$ (see the proof of
Corollary~\ref{cor:LI}); it follows that $(X_I)_x, (Y_\mu)_x, (Z_i)_x \in
\mathcal{T}_x\cN$ are linearly independent for all $x\in V$ as a
consequence of the isomorphism \eqref{eq:isomTN}.
Hence $\mathcal{T}\cN^\omega\subseteq \mathcal{T}\cN$ is a
distribution.

It follows from \eqref{eq:commutator} and the coisotropicity
condition \eqref{eq:coiso} that the Lie bracket of any two of
$\{x^i,\cdot\}$, $\{e^\mu,\cdot\}$ and $\{p^I,\cdot\}$,
$i=1,\ldots,r_0$, $\mu=1,\ldots,r_1$, $I=1,\ldots,r_2$, lies again
in $\mathcal{T}_{\cI^\omega}|_U$. The involutivity of
$\mathcal{T}\cN^\omega$ follows from the fact that \eqref{eq:rest1}
preserves Lie brackets.
\end{proof}

%%%%%%%%%%%%%%%%%%%%%%%%%%%%%%%%%%%%%%%%%%%%%%%%%%%%%%%%%%%%%%%%%%%%%%%%%%%
\subsection{Geometric description of coisotropic submanifolds} 
\label{subsec:geomdescrcoiso}
We now pre\-sent a description of coisotropic submanifolds in classical geometric terms which builds on Theorem~\ref{thm:submanifold}, see Theorem~\ref{thm:coisotropic} below. We begin with some
technical observations.

%%%%%%%
\subsubsection{\underline{Preliminaries}} 

Let $(E,\SP{\cdot,\cdot})$ be a pseudo-euclidean vector bundle over
$M$, and let $\iota \colon N\hookrightarrow M$ be a submanifold.
The restriction $\iota^*E = E|_{N}$ is a pseudo-euclidean vector
bundle over $N$. Let $K\subset E|_N$ be an isotropic subbundle, so  
\begin{equation}\label{eq:Ered}
\Equot:=K^\perp/K
\end{equation}
is also a pseudo-euclidean vector bundle over $N$.
We now collect some results relating the
Atiyah algebroids of $E$, $E|_N$ and $\Equot$, following Appendix~\ref{app:A}.

Let $\Gamma_{\mathbb{A}_E}^\st{N}$ be the subsheaf of $\Gamma_{\mathbb{A}_E}$ given by metric-preserving derivations whose symbols are tangent to $N$. I.e., for each open subset $U\subseteq M$,
$$
\Gamma_{\mathbb{A}_E}^\st{N}(U) = \{ (X,D)\in \Gamma_{\mathbb{A}_E}(U)\,|\, X|_x\in T_xN,\, \forall x\in U\cap N \}.
$$
Then there is a natural restriction map 
\begin{equation}\label{eq:restrict}
    \Gamma_{\mathbb{A}_E}^\st{N}\twoheadrightarrow \iota_* \Gamma_{\mathbb{A}_{E|_\st{N}}}
\end{equation}
that takes $(X,D)\in  \Gamma_{\mathbb{A}_E}^\st{N}(U)$ to $(X_N,D_N) \in \Gamma_{\mathbb{A}_{E|_\st{N}}}(U\cap N)$, where 
$X_N:=X|_{U\cap N}$, and $D_N$ is defined by 
$$
D_N(e|_{U\cap N}):= D(e)|_{U\cap N}.
$$
It is clear that \eqref{eq:restrict} is onto. If we now consider the subsheaf $\Gamma_{\mathbb{A}_E}^\st{N,K}\subseteq \Gamma_{\mathbb{A}_E}^\st{N}$ defined by 
\begin{equation}\label{eq:newONK}
\Gamma_{\mathbb{A}_E}^\st{N,K}(U)=\{(X,D)\in \Gamma_{\mathbb{A}_E}^\st{N}(U)\,|\, D(\Gamma_{\st{E,K}}(U))\subseteq \Gamma_{\st{E,K}}(U)\},
\end{equation}
then the map \eqref{eq:restrict} restricts to a map
\begin{equation}\label{eq:restrictK}
    \Gamma_{\mathbb{A}_E}^\st{N,K} \twoheadrightarrow \iota_* \Gamma^\st{K}_{\mathbb{A}_{{E|_N}}},
\end{equation}
where $\Gamma^\st{K}_{\mathbb{A}_{{E|_N}}}$ is the subsheaf $\Gamma_{\mathbb{A}_{{E|_N}}}$ given by metric-preserving derivations of $E|_N$ that preserve $\Gamma_{K}$. 

Since sections of $\Gamma^\st{K}_{\mathbb{A}_{{E|_N}}}$ also  preserve $\Gamma_{K^\perp}$ (by compatibility with the pseudo-euclidean metric), we have a natural map
\begin{equation}\label{eq:maptoquot}
    \Gamma^\st{K}_{\mathbb{A}_{{E|_N}}}\twoheadrightarrow \Gamma_{\mathbb{A}_{E_{quot}}},
\end{equation}
defined on sections over any open $V\subseteq N$ by $(X,D)\mapsto (X,[D])$, where $[D]\colon \Gamma_{E_{quot}}(V)\to \Gamma_{E_{quot}}(V)$ is given by
$$
[D]([e]) := [D(e)],
$$
with $e \mapsto [e]$ denoting the projection map $\Gamma_{K^\perp}(V)\to \Gamma_{E_{quot}}(V)$. The fact that \eqref{eq:maptoquot} is onto is shown in Prop.~\ref{prop:atiyahonto}, and we have an exact sequence (see \eqref{eq:exactatiyahquot})
\begin{equation}\label{eq:exact}
0\to \Gamma_{K\wedge E|_N} \to \Gamma^\st{K}_{\mathbb{A}_{{E|_N}}}  \to
\Gamma_{\mathbb{A}_{\Equot}} \to 0.
\end{equation}

%%%%%%%%%%%%%%%%%%%%%%%%%
\subsubsection{\underline{Geometric characterization}}\label{subsubsec:coisogeom}
We now proceed to the geometric description of coisotropic submanifolds of a
degree 2 symplectic  $\N$-manifold $\cM$. As explained in
$\S$ \ref{subsec:pseudo}, $\cM$ corresponds to a pseudo-euclidean vector bundle $E$, so that, in terms of the associated object in $\VB2$, 
we have identifications
$$
E=E_1\cong E_1^*, \;\; \widetilde{E}^* = \mathbb{A}_E, \;\;
E_2^*= TM,
$$
where $E_2 = \mathrm{ker}(\phi_E)$ and $E_1\cong E_1^*$ via the pseudo-euclidean metric.
The description of an arbitrary submanifold $\cN$ of $\cM$ in
Theorem \ref{thm:submanifold} amounts to the following geometric
data: a submanifold $N \subseteq M$, subbundles $K_1
\subseteq E|_{N}$ and $K_2\subseteq TM|_{N}$, as well as a bundle
map
\begin{equation}\label{eq:phimap}
\phi \colon K_2 \to \frac{\mathbb{A}_E|_N}{K_1\wedge E|_N}
\end{equation}
such that $\pi'\circ \phi = Id$. Here the map $\pi'$ is obtained by
factoring the anchor (symbol) map $\sigma\colon \mathbb{A}_E \to TM$, whose kernel is
$\wedge^2 E$ (see \eqref{eq:exact}), as the composition
\begin{equation}\label{eq:composition}
\mathbb{A}_E|_{N} \stackrel{\pi''}{\longrightarrow}
\frac{\mathbb{A}_E|_N}{K_1\wedge E|_N}
\stackrel{\pi'}{\longrightarrow}
TM|_{N}=\frac{\mathbb{A}_E|_N}{\wedge^2 E|_N},
\end{equation}
cf. \eqref{eq:diag}. Recall (see \eqref{eq:I123} and Thm.~\ref{thm:submanifold}) that the sheaf of vanishing ideals $\cI$ corresponding to
$\cN$ is generated by its homogeneous components in degrees $0$, $1$
and $2$, 
\begin{equation}\label{eq:I012}
\cI_0 = \fl_N,\;\; \cI_1 = \Gamma_\st{E,K_1},\;\; \cI_2 =
\Gamma_\st{\mathbb{A}_E,\widetilde{K}},
\end{equation}
where
\begin{equation}\label{eq:Ktilde2}
\widetilde{K}:=(\pi'')^{-1}(\phi(K_2))\subseteq \mathbb{A}_E|_N.
\end{equation}

\begin{thm}\label{thm:coisotropic}
Coisotropic
submanifolds $\cN$ of codimension $r_0|r_1|r_2$ are equivalent to quadruples
$(N,K,F,\nabla)$, where
\begin{itemize}
\item $\iota \colon N\hookrightarrow M$ is a 
submanifold of codimension $r_0$,
\item $K$ is an isotropic subbundle of
$E|_{N}$ of rank $r_1$, 
\item $F\subseteq TN$ is a regular, integrable
distribution of rank $r_2$, 
\item $\nabla$ is a flat  metric 
$F$-connection 
on the vector bundle $\Equot=K^\perp/K \to N$.
\end{itemize}
\end{thm}

Recall that a metric 
$F$-connection on the pseudo-Euclidean vector bundle $\Equot$ can be defined as a vector bundle map  
\begin{equation}\label{eq:nabla}
\nabla\colon F \to \mathbb{A}_{\Equot}
\end{equation}
such that $\sigma\circ \nabla = \mathrm{Id}_F$. Note that usual (metric) connections correspond to the case where $F=TN$; in general, $F$-connections are referred to as {\em partial} connections.
Flatness corresponds to the map \eqref{eq:nabla} being a Lie algebroid morphism (i.e., preserving Lie brackets on spaces of sections). 
 
\begin{defi} \label{def:quadruples}Given a pseudo-euclidean vector bundle $E\to M$, we will refer to quadruples $(N, K, F, \nabla)$ as in Thm.~\ref{thm:coisotropic} as {\em geometric coisotropic data}.
\end{defi}

As the proof of the theorem will show, see below, the correspondence between coisotropic submanifolds and geometric coisotropic data is given explicitly as follows. 

By composing the maps \eqref{eq:restrictK} and \eqref{eq:maptoquot}, we obtain a map
$$
 \Gamma_{\mathbb{A}_E}^\st{N,K} \to \iota_*\Gamma_{\mathbb{A}_{E_{quot}}}, \quad (X,D) \mapsto (X_N, [D_N]),
$$
where $\Gamma_{\mathbb{A}_E}^\st{N,K}$ was defined in \eqref{eq:newONK}.
In one direction, 
for geometric coisotropic data $(N, K, F, \nabla)$, the corresponding coisotropic submanifold $\cN$ has body $N$ 
and sheaf of vanishing ideals $\cI$ 
specified by the conditions
\begin{equation}\label{eq:I1I2}
\cI_0 = \fl_N, \qquad \cI_1 = \Gamma_\st{E,K},
\end{equation}
and $\cI_2$ given,
on each open subset $U\subseteq M$, by 
\begin{equation}\label{eq:I2claim}
\cI_2(U) = \{(X,D)\in \Gamma_{\mathbb{A}_E}^\st{N,K}(U)\,|\, X_N \in \Gamma_F({U\cap N}),\,
[D_N]=\nabla_{X_N}\}.
\end{equation}

\begin{remark}\label{rem:flat}
The condition $[D_N]=\nabla_{X_N}$ in \eqref{eq:I2claim} can be equivalently written as $[D_N](\Gamma_{E_{quot}}^\st{flat})=0$, where $\Gamma_{E_{quot}}^\st{flat}$ is the sheaf of $\nabla$-flat sections of $E_{quot}$; indeed, 
 in this case, the difference $[D_N]-\nabla_{X_N}$ must vanish since it is $C^\infty_N$-linear and vanishes on $\nabla$-flat sections, which locally generate $\Gamma_{E_{quot}}$. \hfill $\diamond$
\end{remark}

In the opposite direction, given a coisotropic submanifold $\cN$, we extract the corresponding  data $(N, K, F, \nabla)$ from its sheaf of vanishing ideals $\cI$ as follows:  $\iota: N \hookrightarrow M$ is the body of $\cN$; the vector bundle $K\to N$ is defined by the condition
$\Gamma_K=\iota^*\cI_1$;
the distribution $F\subseteq TN$ is defined by 
$$
\Gamma_F=\sigma(\iota^*\cI_2),
$$
where $\sigma:\Gamma_{\mathbb{A}_{E|_N}}\to \Gamma_{TN}$ is the symbol map; and the $F$-connection $\nabla$ is defined by 
\begin{equation}\label{eq:Fnabla}
\nabla_{Y}[e|_N] = [D(e)|_N],
\end{equation}
 for all sections $e$ of $E$ with $e|_N\subset {K^\perp}$ and sections   $Y$ of $F$, and where $D$ is such that
$(X,D)$ is a section of $\cI_2\subseteq \Gamma_{\mathbb{A}_E}$ with $X|_N=Y$.

\begin{proof}[Proof of Thm.~\ref{thm:coisotropic}]
 Let $\mathcal{N}$ be a coisotropic submanifold. Let $(N,K_1,K_2,\phi)$ be the geometric data associated with the submanifold $\cN$ as in Theorem~\ref{thm:submanifold} (the precise correspondence was recalled  just before Theorem \ref{thm:coisotropic}). We will check that this quadruple is equivalent to the one in the theorem.
Indeed, the condition for $\cN$ being
coisotropic is equivalent
to
\begin{itemize}
\item[(a)] $\{\cI_1,\cI_1 \}\subseteq \cI_0,$
\item[(b)] $\{\cI_2,\cI_0\}\subseteq \cI_0,$
\item[(c)] $\{\cI_2,\cI_1 \}\subseteq \cI_1,$
\item[(d)] $\{\cI_2,\cI_2 \}\subseteq \cI_2.$
\end{itemize}

Since $\cI_0 = \fl_N$ and $\cI_1=\Gamma_\st{E,K_1}$, (a) means that if $e_1,
e_2 \in \Gamma_\st{E,K_1}(U)$, then
$$
0=\{e_1,e_2\}|_{U\cap N}=\SP{e_1|_{U\cap N},e_2|_{U\cap N}};
$$
i.e., $K_1\subset E|_{N}$ is isotropic with respect to
$\SP{\cdot,\cdot}$. We set $K:=K_1$.

If $(X,D)\in \cI_2(U)\subseteq \Gamma_{\mathbb{A}_E}(U)$ then (b)
says that $\pounds_X(\fl_N(U))\subseteq \fl_N(U)$, while (c)
says that $D(\Gamma_\st{E,K}(U))\subseteq \Gamma_\st{E,K}(U)$. I.e., (b) and (c) are
equivalent to
\begin{equation}\label{eq:I2inO}
\cI_2\subseteq \Gamma_{\mathbb{A}_E}^\st{N,K},
\end{equation}
see \eqref{eq:newONK}. Since $\cI_2=\Gamma_\st{\mathbb{A}_E,\widetilde{K}}$, we see that
\begin{equation}\label{eq:tildeKinO}
\Gamma_{\widetilde{K}} = \iota^*\cI_2 \subseteq
\iota^*\Gamma_{\mathbb{A}_E}^\st{N,K} = \Gamma_{\mathbb{A}_{E|_N}}^\st{K}.
\end{equation}
In particular, $\widetilde{K}\subseteq \mathbb{A}_{E|_N}\subseteq
\mathbb{A}_E|_N$. Since $K_2 = \widetilde{K}/\wedge^2 E|_N$, it follows from the Atiyah sequence for $E|_N$,
$$
0\to \wedge^2 E|_N \to \mathbb{A}_{E|_N}\to TN \to 0,
$$
that $K_2 \subseteq TN\subseteq TM|_N$. Let $F:=K_2$.

By the exact sequence \eqref{eq:exact}, we have an identification 
$$
\Gamma_{\mathbb{A}_{\Equot}} =
\pi''(\Gamma_{\mathbb{A}_{E|_N}}^\st{K})\subseteq
\frac{\Gamma_{\mathbb{A}_{E}|_N}}{\Gamma_{K\wedge E|_N}}.
$$
We claim that the image of the map $\phi$ in \eqref{eq:phimap} lies in $\mathbb{A}_{E_{quot}} \subseteq \frac{\mathbb{A}_E|_N}{K\wedge E|_N}$. Indeed, 
by \eqref{eq:Ktilde2}, we know that $\phi(F) = \pi''(\widetilde{K})$, and $\pi''(\widetilde{K})\subseteq \mathbb{A}_{E_{quot}}$ since  $\Gamma_{\widetilde{K}} \subseteq \Gamma_{\mathbb{A}_{E|_N}}^\st{K}$, see \eqref{eq:tildeKinO}.
 So $\phi$ is
equivalent to a $C^\infty_M$-linear map $\Gamma_F\to
\Gamma_{\mathbb{A}_{\Equot}}$, and this is precisely a metric
 $F$-connection. Let us denote it by $\nabla$, so we write $\phi(Y)=\nabla_Y$ for $Y\in F$. 

We can now describe $\cI_2 = \Gamma_\st{\mathbb{A}_E,\widetilde{K}}$ as follows. For any open $U\subseteq M$, an element $(X,D)\in \Gamma^{\st{N,K}}_{\mathbb{A}_E}(U)$ is in $\cI_2(U)$ if and only if its restriction $(X_N, D_N) \in \Gamma^{\st{K}}_{\mathbb{A}_{E|_N}}(V)$, where $V=U\cap N$, lies in $\Gamma_{\widetilde{K}}(V)$, which, by  \eqref{eq:Ktilde2} and \eqref{eq:composition}, is equivalent to the condition
$$
\pi''(X_N,D_N) = (X_N, [D_N]) = \phi(X_N) = \nabla_{X_N},
$$
from where the description of $\cI_2$ in \eqref{eq:I2claim} follows.

Finally, condition (d) is equivalent to 
$$
[\Gamma_F,\Gamma_F]\subseteq \Gamma_F, \;\mbox{ and }
\nabla_{[X_N,Y_N]}=\nabla_{X_N}\nabla_{Y_N}-\nabla_{Y_N}\nabla_{X_N},
$$
for all $X_N,Y_N\in \Gamma_F(V)$, $V\subseteq N$ open; i.e., 
$F$ is an integrable distribution and $\nabla$ is flat.
\end{proof}

\begin{ex}\label{ex:bott}
For a manifold $M$, we consider the pseudo-euclidean vector bundle $E:=TM\oplus T^*M$ over $M$, cf. Example~\ref{ex:cotangent}. Then a pair $(N,F)$, where $N\subseteq M$ is a submanifold and $F\subseteq TN$ is an involutive subbundle, naturally gives rise to geometric coisotropic data $(N, K, F, \nabla)$ as follows. We set $K=F\oplus \mathrm{Ann}(TN) \subseteq E|_N$, noticing that we have a natural identification
$$
\frac{K^\perp}{K} =  \frac{TN}{F} \oplus \left ( \frac{TN}{F} \right )^*.
$$
Recall that $TN/F$ carries a canonical flat $F$-connection $\nabla^{Bott}$, known as the {\em Bott connection}, given by
$$
\nabla^{Bott}_Z (\overline{Y}) = \overline{[Z, Y]},
$$
where $\overline{Y}$ denotes the projection of a section $Y$ of $TN$ in $TN/F$.
We then set $\nabla:= \nabla^{Bott} \oplus (\nabla^{Bott})^*$, which is a metric, flat $F$-connection on $K^\perp/K$.
\hfill $\diamond$
\end{ex}

The next example describes the graded analogue of the well-known fact that the annihilator of an involutive distribution on a manifold defines a coisotropic submanifold of its cotangent bundle.

\begin{ex}\label{ex:coisocotangent}
Let $A\to M$ be a vector bundle, and consider the degree 2 symplectic $\mathbb{N}$-manifold  $\cM=T^*[2]A[1]$
corresponding to pseudo-euclidean vector bundle $A\oplus A^*$ (see Example~\ref{ex:cotangent}). 
As in the classical case, one can view vector fields on $A[1]$ as ``fiberwise linear'' functions on $T^*[2]A[1]$: we have a natural (injective) morphism of sheaves of $C_{A[1]}$-modules
\begin{equation}\label{eq:fiberlinsh}
\mathcal{T}({A[1]})\to C_{\cM}[2],
\end{equation}
given by (cf. Example~\ref{ex:VFindeg1}) :
\begin{align*}
     &(\mathcal{T}({A[1]}))_{-1}=\Gamma_A \to    \Gamma_{ A \oplus 0}\subset
     \Gamma_{A\oplus A^*}= (C_{\cM})_1,\\
     & (\mathcal{T}({A[1]}))_{0} \to  (C_{\cM})_2,
 \end{align*}
 where the first map is the inclusion, and the second map takes a derivation $(X,D)$ of $A|_U$ to $(X, D\oplus D^*) \in \Gamma({\mathbb{A}_{A\oplus A^*}|_U})= (C_{\cM})_2(U)$. This map is a morphism of graded Lie algebras (with respect to the Poisson bracket on $C_\cM$).

Any involutive distribution $\mathcal{D}$ on $A[1]$ (in the sense of $\S$ \ref{subsec:frobenius})
generates a regular sheaf of ideals $\mathcal{I}$ in $C_\cM$ via \eqref{eq:fiberlinsh} that satisfies $\{\mathcal{I},\mathcal{I}\}\subseteq \mathcal{I}$, so it defines a coisotropic submanifold $\cN_\cD$ of $\cM$; by analogy with the classical case, $\cN_\cD$ should be thought of as the ``annihilator of $\mathcal{D}$'' in $T^*[2]A[1]$.

In classical geometrical terms, as shown in \cite[Lemmas~3.1, 3.2]{ZZL2}, an involutive distribution $\cD$ on $A[1]$ is given by
\begin{itemize}
 \item a subbundle $B\to M$ of $A$,
\item an involutive distribution $F$ on $M$,
\item a flat $F$-connection $\overline{\nabla}$ on the vector bundle  $A/B\to M$.
\end{itemize}
The coisotropic data $(N, K, F, \nabla)$ corresponding to $\cN_\cD$ are $N=M$, the isotropic subbundle $K=B\oplus 0 \subseteq A\oplus A^*$, the involutive distribution $F$, and the metric flat $F$-connection $\nabla$ on $$K^{\perp}/K=(A\oplus \mathrm{Ann}(B))/(B\oplus 0)=(A/B)\oplus (A/B)^*$$ given by $\overline{\nabla}\oplus \overline{\nabla}^*$. Note that Example~\ref{ex:bott} can be seen as a special case of this example when $N=M$.
\hfill $\diamond$
\end{ex}

We will use the following general remark to obtain natural examples of geometric coisotropic data by means of 
equivariant vector bundles. 

\begin{remark}[Equivariant bundles and partial connections]\label{rem:actionconnection}
For a Lie group $G$, the infinitesimal counterpart of a $G$-action on a vector bundle $A\to N$ by vector-bundle automorphisms is a representation of the action Lie algebroid $\g\ltimes N$ on $A$. Assuming that the $G$-action on $N$ is (locally) free, and letting $F\subseteq TN$ be the distribution tangent to the $G$-orbits, we have an identification of $F$ with $\g\ltimes N$ as Lie algebroids. So a representation of $\g\ltimes N$ on $A$ is equivalent to a flat $F$-connection $\nabla$ on $A$. 
\hfill $\diamond$
\end{remark}

\begin{ex}[Coisotropic data from invariant subbundles]\label{ex:coisodataaction} Let $E\to M$ be a $G$-equivariant vector bundle with an invariant pseudo-euclidean structure. Let $K\to N$ be a $G$-invariant isotropic subbundle of $E\to M$. When the $G$-action on $N$ is (locally) free, we obtain geometric coisotropic data $(N, K, F, \nabla)$ as follows. The $G$-action on $E$ also preserves $K^\perp\to N$, hence $E_{quot}=K^\perp/K $ is $G$-equivariant, and its induced pseudo-euclidean structure is $G$-invariant. We then set $F$ to be the orbit distribution on $N$, and $\nabla$ is the flat metric  $F$-connection corresponding to $G$-action on $E_{quot}$, as in Remark~\ref{rem:actionconnection}.
\hfill $\diamond$
\end{ex}

%%%%%%%%%%%%%%%%%%%%%%%%%%%%%%%%%%%%%%%%%%%%%%%%%%%%%%%%%%%%%%%%%%%%

\subsection{Basic functions} \label{subsec:basic}

Let $\cN$ be a coisotropic submanifold of a degree 2 symplectic $\mathbb{N}$-manifold $\cM$, with $\dim(\cN)=n_0|n_1|n_2$ and
$\codim(\cN)=r_0|r_1|r_2$, and vanishing ideal $\cI$.
Consider the subsheaf of algebras 
\begin{equation}\label{eq:cbas}
    (C_\cN)_{bas} := \iota^{-1}(\mathfrak{N}_\cI/\cI)\subseteq C_\cN.
\end{equation}
More concretely, considering local coordinates on $U\subseteq M$
adapted to $\cN$ and the map $\iota^\sharp \colon C_\cM(U)\to C_\cN(V)$,
where $V=U\cap N$, we see that
\begin{equation}\label{eq:bas}
(C_\cN)_{bas}(V)=\iota^\sharp(\mathfrak{N}_\cI(U)).
\end{equation}
We will refer to $(C_\cN)_{bas}$ as the sheaf of {\em basic functions} on $\cN$ relative to the null distribution, as indicated by property (b) of the next result.

\begin{prop}\label{prop:basic} The following holds:
\begin{itemize}
\item[(a)] $(C_\cN)_{bas}$ is a sheaf of Poisson algebras.
% (with bracket of degree -2).
\item[(b)]
For $V \subseteq N$ open,
$$
(C_\cN)_{bas}(V) = \{ f\in C_\cN(V)\,|\, X(f)=0 \; \forall\, X\in
\mathcal{T}\cN^\omega(V) \}.
$$
\end{itemize}
\end{prop}

\begin{proof}
By the local characterization of $(C_\cN)_{bas}$ in \eqref{eq:bas},
we may write any section $f$ in $(C_\cN)_{bas}(V)$ as
$\iota^\sharp(\tilde{f})$, for $\tilde{f}\in \mathfrak{N}_\cI(U)$. Using
the fact that $\{\mathfrak{N}_\cI,\cI\}\subseteq \cI$, we see that
\begin{equation}\label{eq:poisbasic}
 \{f,g\}_{bas}:=\iota^\sharp\{\tilde{f},\tilde{g} \}   
\end{equation}
is a
well-defined Poisson bracket (of degree $-2$) on $(C_\cN)_{bas}(V)$.
Since $\iota^\sharp$ is a map of sheaves $C_\cM\to \iota_*C_\cN$, it
is a simple matter to verify that this Poisson bracket globalizes to
the sheaf $(C_\cN)_{bas}$.

To prove (b), it is enough to consider the case where $U\subseteq M$
has local coordinates $\{x^i,e^\mu,p^I\}$ adapted to $\cN$ and
$V=U\cap N$. Let $f=\iota^\sharp(\tilde{f}) \in C_\cN(V)$, for
$\tilde{f}\in C_\cM(U)$. The condition that $X(f)=0$ for all $X\in
\mathcal{T}\cN^\omega(V)$ is equivalent to
$$
\{x^i, \tilde{f}\}\in \cI(U), \;\; \{e^\mu, \tilde{f}\}\in
\cI(U),\;\; \{p^I, \tilde{f}\}\in \cI(U),
$$
for $1\leq i\leq r_0$, $1\leq \mu\leq r_1$, $1\leq I\leq r_2$, which
means that $\tilde{f}\in \mathfrak{N}_\cI(U)$, i.e., $f\in
(C_\cN)_{bas}(V)$.
\end{proof}

\begin{cor}\label{cor:basic}
Around any point in $N$, there are local coordinates
\begin{equation}\label{eq:coord}
\{x^1,\ldots,x^{n_0},e^1,\ldots,e^{n_1},p^1,\ldots,p^{n_2}\}
\end{equation}
on $\cN$ with respect to which sections of $(C_\cN)_{bas}$ are the
sections of $C_\cN$ which depend only on $x^{j>r_0}$, $e^{\nu>r_1}$,
and $p^{J>r_2}$; in particular, $(C_\cN)_{bas}$ is locally generated
by $((C_\cN)_{bas})_0$, $((C_\cN)_{bas})_1$, and
$((C_\cN)_{bas})_2$.
\end{cor}

\begin{proof}
By Frobenius' theorem (Theorem~\ref{thm:frob}), there are local
coordinates \eqref{eq:coord} with respect to which the distribution
$\mathcal{T}\cN^\omega$ is spanned by $\partial/\partial x^j$,
$\partial/\partial e^\nu$, $\partial/\partial p^J$,
$j=1,\ldots,r_0$, $\nu=1,\ldots,r_1$, $J=1,\ldots,r_2$. The result
now follows from Prop.~\ref{prop:basic}(b).
\end{proof}

We will now give a description of the sheaf of basic functions in terms of the geometric data $(N, K, F, \nabla)$ corresponding to $\cN$, as in Thm.~\ref{thm:coisotropic}. Consider  the sheaves $I_N$,
$\Gamma_\st{TM,F}$, $\Gamma_\st{E,K^\perp}$, $\Gamma^\st{N,K}_\st{\mathbb{A}_E}$, and
$\Gamma_\st{E,K^\perp}^\st{flat}$, given by
\begin{equation}\label{eq:Snabla}
\Gamma_\st{E,K^\perp}^\st{flat}(U):= \{e\in \Gamma_\st{E,K^\perp}(U)\,|\,\nabla
[e|_{U\cap N}]=0  \},
\end{equation}
for $U\subseteq M$ open. 

We start with a description of the sheaf of Lie normalizers $\mathfrak{N}_\cI$.

\begin{lemma}\label{lem:NI} 
For $U\subseteq M$ open,
$\mathfrak{N}_\cI$ satisfies
\begin{align*}
(\mathfrak{N}_\cI)_0(U) &= \{f\in C_M^\infty(U)\,|\, X(f)\in \fl_N(U)\;\;\forall \,X\in \Gamma_\st{TM,F}(U) \},\\
(\mathfrak{N}_\cI)_1(U) &= \Gamma_\st{E,K^\perp}^\st{flat}(U),\\
(\mathfrak{N}_\cI)_2(U) &= \{(X,D)\in \Gamma^\st{N,K}_{\mathbb{A}_E}(U)\,|\,
[X,\Gamma_\st{TM,F}(U)]\subseteq \Gamma_\st{TM,F}(U),\,\\
&\qquad\qquad \qquad \qquad \qquad \;\;
D(\Gamma_\st{E,K^\perp}^\st{flat}(U))\subseteq \Gamma_\st{E,K^\perp}^\st{flat}(U) \}.
\end{align*}
\end{lemma}
  
\begin{proof} 
Recall from Theorem~\ref{thm:coisotropic} that $\cI_0=\fl_N$,
$\cI_1=\Gamma_\st{E,K}$, and $\cI_2$ is given by \eqref{eq:I2claim}. To show
that a local section $g$ of $C_\cM$ belongs to $\mathfrak{N}_\cI$,
it suffices to check that
$\{g,\cI_i\}\subseteq \cI$, for $i=0,1,2$. Note that, as a consequence of the surjectivity of \eqref{eq:restrictK} and \eqref{eq:maptoquot}, the map
\begin{equation}\label{eq:surj}
\cI_2\to \Gamma_F,
\end{equation}
given on sections over $U\subseteq M$ by $(X,D) \mapsto
X_N=X|_{U\cap N}$, is onto.

If $f\in C^\infty_M(U)$ and $(X,D)\in \cI_2(U)$, then
$\{f,(X,D)\}\in \cI(U)$ means that $X(f)\in \fl_N(U)$. But, by the
surjectivity of \eqref{eq:surj}, $X\in \Gamma_{TM}(U)$ is the symbol
of an element in $\cI_2(U)$ if and only if $X\in \Gamma_\st{TM,F}(U)$. This
provides the desired description of $(\mathfrak{N}_\cI)_0$.

If $e\in \Gamma_E(U)$, then $\{e,\cI_1(U)\}\subseteq \cI(U)$ is
equivalent to $\SP{e,e'}=\fl_N(U)$ for all $e'\in \Gamma_\st{E,K}(U)$, i.e.,
$e\in \Gamma_\st{E,K^\perp}(U)$. The condition $\{e,\cI_2(U)\}\subseteq
\cI(U)$ amounts to $D(e)\in \Gamma_\st{E,K}(U)$ whenever $(X,D)\in \cI_2(U)$.
This is equivalent to $[D_N]([e|_{U\cap N}])=
\nabla_{X_N}([e|_{U\cap N}]) =0$, which in turn means that $e \in
\Gamma_\st{E,K^\perp}^\st{flat}(U)$, since \eqref{eq:surj} is onto.

Let $(X,D)\in \Gamma_{\mathbb{A}_E}(U)$. The conditions $\{ (X,D),
\cI_i(U)\}\subseteq \cI(U)$, $i=0,1$, simply say that $(X,D)\in
\Gamma_{\mathbb{A}_E}^\st{N,K}(U)$. Condition $\{ (X,D), \cI_2(U)  \}\subseteq \cI(U)$ is
equivalent, using the surjectivity of \eqref{eq:surj}, to
\begin{equation}\label{eq:lastcond}
[X,\Gamma_\st{TM,F}(U)]\subseteq \Gamma_\st{TM,F}(U),\;\;
[[D_N],\nabla_{Y}]=\nabla_{[X_N,Y]},
\end{equation}
for all $Y\in \Gamma_F(V)$, where $V=U\cap N$. .
Now note that the second condition in \eqref{eq:lastcond}
is equivalent to $[D_N]:\Gamma_{\Equot}(U)\to \Gamma_{\Equot}(U)$
preserving $\nabla$-flat sections, which is the same as
$D(\Gamma_\st{E,K^\perp}^\st{flat}(U))\subseteq \Gamma_\st{E,K^\perp}^\st{flat}(U)$.
\end{proof}

We will now use the previous result to obtain a geometric description of $(C_\cN)_{bas}$. For that, let us consider the sheaf of basic functions on $N$ with respect to $F\subseteq TN$, denoted by $(C_N^\infty)_{bas}\subseteq C_N^\infty$; i.e.,  
for each open subset $V\subseteq N$, 
\begin{equation}\label{eq:basicfns}
(C_N^\infty)_{bas}(V)=\{f\in C^\infty_N(V)\;|\; X(f)=0\, \, \forall
X\in \Gamma_F(V)\}.
\end{equation}
Recall that $\Gamma^\st{flat}_{E_{quot}}$ denotes the sheaf of $\nabla$-flat sections of $E_{quot}$, which is naturally a sheaf of $(C_N^\infty)_{bas}$-modules.

We also need to consider the vector bundle
\begin{equation}\label{eq:Anabla}
\mathbb{A}_{E_{quot}}^\nabla:= \mathbb{A}_{E_{quot}}/\nabla(F)
\end{equation}
over $N$,  where the vector subbundle $\nabla(F)\subseteq \mathbb{A}_{E_{quot}}$ is the image of \eqref{eq:nabla}.
For a section $(Y,\Delta)$ of  $\mathbb{A}_{E_{quot}}$, we denote its class in $\mathbb{A}_{E_{quot}}^\nabla$ by $\overline{(Y,\Delta)}$.

As explained in $\S$ \ref{subsec:atiyahquot}, the flat $F$-connection  $\nabla$ induces a flat $F$-connection $\overline{\nabla}$ on $\mathbb{A}_{E_{quot}}^\nabla$ by
$$
\overline{\nabla}_Z(\overline{(Y,\Delta)})= \overline{([Z,Y],[\nabla_Z, \Delta])}, 
$$
where $Z$ is a section of $F$, in such a way that a section $\overline{(Y,\Delta)}$ is flat if and only if   
 \begin{equation}\label{eq:flat}
 [Y,\Gamma_F]\subseteq \Gamma_F, \qquad \Delta(\Gamma^\st{flat}_{E_{quot}})\subseteq \Gamma^\st{flat}_{E_{quot}}.
 \end{equation}

\begin{prop}\label{prop:Cbas} For a coisotropic submanifold $\cN$ corresponding to $(N, K, F, \nabla)$, the sheaf $(C_\cN)_{bas}$ is such that 
$$
((C_\cN)_{bas})_0 = (C_N^\infty)_{bas}, \quad ((C_\cN)_{bas})_1 = \Gamma^\st{flat}_{E_{quot}}, \quad ((C_\cN)_{bas})_2 = \Gamma^\st{flat}_{\mathbb{A}^\nabla_{E_{quot}}}.
$$
\end{prop}

\begin{proof}
Using Lemma \ref{lem:NI} as well as \ref{eq:I1I2} and \eqref{eq:I2claim} (see also Remark~\ref{rem:flat}), we directly obtain the equalities in degrees 0 and 1, and a description of basic functions in degree 2 as follows: for $V\subseteq N$ open, 
$((C_\cN)_{bas})_2(V)$ is given by operators 
$(Y,D) \in \Gamma^{\st{K}}_{\mathbb{A}_{E|_N}}(V)$ satisfying
$$
[Y, \Gamma_F(V)]\subseteq \Gamma_F(V), \quad  
[D](\Gamma^\st{flat}_{E_{quot}}(V))\subseteq \Gamma^\st{flat}_{E_{quot}}(V), 
$$
modulo those such that $Y\in \Gamma_F(V)$ and $[D](\Gamma_{\Equot}^\st{flat}(V))=0$ (or, equivalently, $[D]=\nabla_{Y}$). We denote the equivalence class of $(Y,D)$ by $\overline{(Y,D)}$.
We then have a natural map 
\begin{equation}\label{eq:basic2}
((C_\cN)_{bas})_2 \to \Gamma^\st{flat}_{\mathbb{A}^\nabla_{E_{quot}}}, \quad \overline{(Y,D)} \mapsto \overline{(Y, [D])},
\end{equation}
which is readily  seen to be injective. Its surjectivity follows from the surjectivity of \eqref{eq:maptoquot}, so we have an isomorphism.
\end{proof}

The Poisson bracket relations in $(C_\cN)_{bas}$ involving functions
of degree 0, 1 and 2 (cf. Prop.~\ref{prop:basic}, part (a)) are
computed to be as follows: for $f\in (C^\infty_N)_{bas}(V)$, $e, e'
\in \Gamma^\st{flat}_{\Equot}(V)$, and
$\overline{(Y, \Delta)}, \overline{(Y', \Delta')} \in \Gamma^\st{flat}_{\mathbb{A}^\nabla_{E_{quot}}}$, we have
\begin{equation}\label{eq:basicprel}
\begin{aligned}
& \{e,e'\} = \SP{e,e'},\\
& \{ \overline{(Y, \Delta)}, f  \} = \pounds_Y(f), \\
& \{ \overline{(Y, \Delta)}, e \} = \Delta(e), \\
& \{ \overline{(Y, \Delta)}, \overline{(Y', \Delta')} \} =
\overline{([Y, Y'],[\Delta, \Delta'])}.
\end{aligned}
\end{equation}

%%%%%%%%%%%%%%%%%%%%%%%%%%%%%%%%%%%%%%%%%%%%%%%%%%%%%%%%%%%%%%%%%%%%
\subsection{Lagrangian submanifolds}\label{subsec:lagrangians}

A submanifold $\cN$ of a degree 2 symplectic $\N$-manifold
$\cM$ is called \textit{lagrangian} if
\begin{equation}\label{eq:lagdist}
\mathcal{T}\cN^\omega = \mathcal{T}\cN.
\end{equation}

\begin{thm}\label{thm:lag}
The following are equivalent:
\begin{itemize}
\item[(a)] $\cN$ is lagrangian;
\item[(b)] $\cN$ is coisotropic and $(C_\cN)_{bas} = \{ f\in C_N^\infty \,|\, f \mbox{ is locally constant }
\}$;
\item[(c)] $\cN$ is coisotropic, and its corresponding quadruple $(N,K, F, \nabla)$ satifies $K=K^\perp$ and $F=TN$;
\item[(d)] $\cN$ is coisotropic and 
$\tdim(\cN) = m_0 + \frac{m_1}{2} =\frac{1}{2}\tdim(\cM)$.\end{itemize}
\end{thm}

We obtain the following known geometric description of lagrangian submanifolds (see  \cite{severa:sometitle}) as a direct consequence of part (c) of the theorem. 

Let $E\to M$ be the pseudo-euclidean vector bundle corresponding to $\cM$.  

\begin{cor}\label{cor:lag}
Lagrangian submanifolds of $\cM$ are equivalent to 
lagrangian subbundles $K\to N$ of $E\to M$. 
Explicitly, the lagrangian submanifold corresponding to $K\to N$ has body $N$ and sheaf of vanishing ideals determined by $\cI_0=\fl_N$, $\cI_1=\Gamma_\st{E,K}$ and $\cI_2 = \Gamma_{\mathbb{A}_E}^\st{N,K}$.
\end{cor}

\begin{proof}[Proof of Theorem~\ref{thm:lag}]

The equivalence between (a) and (b) is a consequence of
Prop.~\ref{prop:basic}(b) and the fact, verified by a direct computation, that, for any  $\N$-manifold of degree 2 $\cN$ and $f\in C_\cN(V)$ (with
$V$ an open subset in the body),  the condition $X(f)=0$ for all $X\in \mathcal{T}\cN(V)$ holds
if and only if $f\in C^\infty(V)$ and is locally constant.

From Prop.~\ref{prop:Cbas} we see that $F=TN$ and $K=K^\perp$ if and
only if $((C_\cN)_{bas})_0\subseteq C_N^\infty$ is the space of
locally constant functions and $((C_\cN)_{bas})_1$ is zero. So (b)
implies (c). If we assume that (c) holds, then Lemma~\ref{lem:NI}
shows that $(\mathfrak{N}_\cI)_2= \Gamma_{\mathbb{A}_E}^\st{N,K} = \cI_2$. It follows that
$((C_\cN)_{bas})_2$ is zero, and (b) holds since $(C_\cN)_{bas}$ is locally generated in degrees 0, 1 and 2, see
Cor.~\ref{cor:basic}.

If $n_0| n_1| n_2 = \dim(\cN)$, then (c) implies that $n_1=
m_1-\rank(K)=m_1/2$ and $n_2=m_2-n_0=m_0-n_0$. So $\tdim(\cN)=m_0+
m_1/2 = \tdim(\cM)/2$,  which is the condition in (d) . Conversely, note that
$$n_0+n_1+n_2 =
n_0+
m_1-\rank(K) + m_0-\rank(F),
$$
and suppose that this agrees with $m_0+ m_1/2$. Since $\rank(K)\leq
m_1/2$ (as $K$ is isotropic), we have
$$
n_0+ m_1/2 + m_0-\rank(F)\leq n_0+ m_1-\rank(K) +
m_0-\rank(F)=m_0+m_1/2,
$$
hence $\rank(F)\geq n_0$. But since $F\subseteq TN$, we must have
$F=TN$. Hence $\tdim(\cN)=m_1-\rank(K)+m_0 = m_1/2 + m_0$. This
implies that $\rank(K)=m_1/2$, so $K=K^\perp$.
 Hence (d) implies (c).
\end{proof}

\begin{ex}\label{ex:conormal}
For a vector bundle $A\to M$, consider the graded cotangent bundle $T^*[2]A[1]$ (Example~\ref{ex:cotangent}), corresponding to the pseudo-euclidean vector bundle $E=A\oplus A^*$.
Given  a vector subbundle $B\to N$ of $A$, we obtain a lagrangian subbundle $K = B\oplus \mathrm{Ann}(B)\subseteq E$. 
The corresponding lagrangian submanifold of $T^*[2]A[1]$ is denoted by $\nu^*[2] B[1]$ and plays the role of the conormal bundle of the submanifold $B[1]\subseteq A[1]$. 
\hfill $\diamond$
\end{ex}

%%%%%%%%%%%%%%%%%%%%%%%%%%%%%%%%%%%%%%%%%%%%%%%%%%%%%%%%%%%%%%
\section{Reduction of coisotropic submanifolds}\label{sec:coisoreduction}
The classical setting of coisotropic reduction is that of  a symplectic manifold $(M, \omega)$ and a coisotropic submanifold $\iota\colon N \to M$  whose null foliation is  simple, i.e., its leaf space is a smooth manifold $\underline{N}$ such that the quotient map $p\colon N \to \underline{N}$ is a surjective submersion. Then $\underline{N}$ acquires a unique symplectic structure $\omega_{red}$ such that $p^* \omega_{red}=\iota^*\omega$. Equivalently, recalling that  $C^\infty(N)_{bas}$ is a Poisson algebra, $\omega_{red}$ is characterized (in terms of its Poisson bracket) by the fact that the identification 
$$
p^*\colon C^\infty(\underline{N})\to C^\infty(N)_{bas}
$$ 
is a Poisson isomorphism.

In this section we present coisotropic reduction for degree 2 symplectic  $\mathbb{N}$-manifolds. We start by recalling general facts about quotients of vector bundles along surjective submersions from \cite[$\S$ 2.1]{mac:lie}.

\subsection{Quotients of vector bundles along submersions}\label{subsec:redvb}
Any surjective submersion  $p:N\to \Mred$ gives rise to a {\em submersion groupoid} $\mathcal{R}_p :=  N\times_{\Mred} N$, given by the equivalence relation that sets $x, y \in N$ as equivalent if $p(x)=p(y)$. The pullback $p^*\underline{A}$ of any vector bundle $\underline{A}\to \Mred$ carries a natural representation of $\mathcal{R}_p$. Conversely, as proven in \cite[Thm.~2.1.2]{mac:lie}, if a vector bundle $A\to N$ is equipped with a representation of $\mathcal{R}_p$, then there exists a vector bundle $\underline{A} \to \Mred$, unique up to isomorphism, such that $A \cong p^*\underline{A}$ through a vector-bundle isomorphism that preserves the representations of $\mathcal{R}_p$.

Let $F\subseteq TN$ be a distribution on a manifold $N$.
We will refer to $F$ as \textit{simple} if $F$ is integrable and its underlying foliation
is simple, in the sense that
 there is a manifold $\Mred$ and a
surjective submersion $p:N\to \Mred$ with connected fibres so that
$\ker(dp)_x = F_x$ for all $x\in N$ (so $\Mred$ is identified with the leaf space of the
foliation). The pull-back map $ C_{\Mred}^\infty \to
p_*C_N^\infty$, defined for each local section $f$ of $
C_{\Mred}^\infty$ by $p^*f=f\circ p$, provides an identification
\begin{equation}\label{eq:basicid}
C_{\Mred}^\infty \cong p_*(C_N^\infty)_{bas},
\end{equation}
where $(C_N^\infty)_{bas}\subseteq C_N^\infty$ is as in \eqref{eq:basicfns}.

Let $F$ be a simple distribution  on $N$ with surjective submersion $p\colon N\to \Mred$, and let $A\to N$ be a vector bundle.
A representation of $\mathcal{R}_p$ on $A$ gives rise, through differentiation, to a representation of $F\to N$ (viewed as a Lie algebroid) on $A$, which is equivalent to a flat $F$-connection $\nabla$ on $A$. It is not always possible to integrate such $\nabla$ to a representation of $\mathcal{R}_p$; when this is the case, we say that $\nabla$ has {\em trivial holonomy}. 
The first part of the following lemma is a consequence of \cite[Thm.~2.1.2]{mac:lie} mentioned above.

\begin{lemma}\label{lem:hol}
Let $F$ be a simple distribution on $N$ with surjective submersion $p\colon N\to \Mred$, and let $A\to N$ be a vector
bundle equipped with a flat $F$-connection $\nabla$.
Then $\nabla$ has trivial holonomy if and only if there is
a vector bundle $\underline{A}\to \Mred$ (unique, up to isomorphism) fitting into the pull-back
diagram
$$
\xymatrix{ A  \ar[r]^{} \ar[d]_{} & \underline{A} \ar[d]^{} \\
N \ar[r]_{p} &  \Mred }
$$
and such that the natural pull-back map $\Gamma_{\underline{A}}\to
p_*\Gamma_A$ is an isomorphism onto $p_*\Gamma^\st{flat}_A$. 
Moreover,
if $A$ is pseudo-euclidean and $\nabla$ is metric, then
$\underline{A}$ inherits a pseudo-euclidean structure for which the
pull-back map is an isometry.
\end{lemma}

Regarding the second assertion in the lemma, note that if $A$ has a pseudo-euclidean structure $\SP{\cdot,\cdot}$ and $\nabla$ is metric, then
\begin{equation}\label{eq:metric}
\SP{\Gamma_A^\st{flat},\Gamma_A^\st{flat}}\subseteq (C_N^\infty)_{bas}.
\end{equation}
By means of the identifications $C^\infty_{\Mred}\cong
p_*(C_N^\infty)_{bas}$ and
\begin{equation}\label{eq:p1}
  \Gamma_{\underline{A}}\cong
p_* \Gamma_A^\st{flat},
\end{equation}
we see that $\SP{\cdot,\cdot}$ induces a pseudo-euclidean structure on $\underline{A}$.

\begin{defi}\label{def:vbquot}
  We refer to $\underline{A}$ in the previous lemma as the {\em quotient} of $A\to N$ with respect to $F$ and $\nabla$. 
\end{defi}

\begin{ex}[Quotients of tangent and cotangent bundles]\label{ex:tanquot}
For a manifold $N$ with an involutive distribution $F$, consider $A=TN/F$ equipped with the Bott connection $\nabla^{Bott}$, see Example~\ref{ex:bott}. In this case, flat sections of $A$ are given by images of ``projectable'' vector fields in $N$ (i.e., vector fields $Y$ satisfying $[Y,\Gamma(F)]\subseteq \Gamma(F)$), see \eqref{eq:flat}.
Whenever $F$ is simple, with quotient map $p: N\to \underline{N}$, the quotient of $A$ with respect to $F$ and $\nabla^{Bott}$ is well defined and given by $\underline{A} =  T \underline{N}$. 
Dually, the quotient of $\mathrm{Ann}(F)=(TN/F)^*$ with respect to $F$ and $(\nabla^{Bott})^*$ is $T^*\underline{N}$. \hfill $\diamond$
\end{ex}

The following special case will be useful in Section~\ref{sec:momap}.
 
\begin{ex}[Quotients of vector bundles by group actions]\label{ex:equivquot}
Let $G$ be a Lie group,  and let $A\to N$ be a $G$-equivariant vector bundle such that the $G$-action on $N$ is free and proper. Then the quotient map $p\colon N \to N/G$ is a surjective submersion, and there is a well-defined vector bundle 
$$
A/G \to N/G
$$  
with the property that $p^*(A/G)=A$ and $\Gamma(A/G)$ is identified with the space of $G$-invariant sections of $A$. In this setting,  the $G$-action on $A$ is equivalent to a representation of the action groupoid $G\ltimes N$ on $A$, and $G\ltimes N$ is naturally identified with the submersion groupoid $\mathcal{R}_p = N \times_{N/G} N$.

Let $F\subseteq TN$ be the distribution tangent to the $G$-orbits.  As recalled in Remark~\ref{rem:actionconnection}, the infinitesimal $\g$-action on $A$ is defined by a representation of $\g\ltimes N = F$ on $A$, which is the same as a flat $F$-connection $\nabla$. When $G$ is connected, $N/G$ is the leaf space of $F$, and $A/G$ is the quotient of $A$ with respect to $F$ and $\nabla$, as in Def. ~\ref{def:vbquot}. \hfill $\diamond$
\end{ex}

Let us recall from $\S$ \ref{subsec:atiyahquot} how the Atiyah algebroids $\mathbb{A}_A$ and $\mathbb{A}_{\underline{A}}$ are related, and that we denote by $\nabla(F)$ the image of the $F$-connection, viewed as a vector-bundle map $\nabla\colon F \to \mathbb{A}_A$. The vector bundle 
$$
\mathbb{A}_{A}^\nabla= \mathbb{A}_A/\nabla(F)
$$
carries a flat $F$-connection $\overline{\nabla}$ induced from $\nabla$, see \eqref{eq:nablainduced}, with flat sections described in Lemma~\ref{lem:flatcharac}. Then
 $\mathbb{A}_{\underline{A}}$ is the quotient of  $\mathbb{A}^\nabla_A$ with respect to $F$ and $\overline{\nabla}$,
 $$
\xymatrix{ {\mathbb{A}^\nabla_A}  \ar[r]^{} \ar[d]_{} & \mathbb{A}_{\underline{A}} \ar[d]^{} \\
N \ar[r]_{p} &  \Mred,}
$$
and the identification 
\begin{equation}\label{eq:p2}
\Gamma_{\mathbb{A}_{\underline{A}}} \cong p_*\Gamma^\st{flat}_{\mathbb{A}^\nabla_A}
\end{equation}
is given as follows (see Prop.~\ref{prop:Ared}): a section  $\overline{(Y,\Delta)}$ of $\Gamma^\st{flat}_{\mathbb{A}^\nabla_A}(V)$ corresponds to a 
section $(\underline{Y}, \underline{\Delta})$ of $\Gamma_{\mathbb{A}_{\underline{A}}}(\underline{V})$, for $\underline{V}\subseteq \Mred$ open and $V=p^{-1}(\underline{V})$, if and only if $\underline{Y}=p_*(Y)$ and the restriction of $\Delta$ to flat sections, $\Delta\colon \Gamma^\st{flat}_A(V)\to \Gamma^\st{flat}_A(V)$, corresponds to $\underline{\Delta}\colon \Gamma_{\underline{A}}(\underline{V})\to \Gamma_{\underline{A}}(\underline{V})$ upon the identification $\Gamma_{\underline{A}}\cong
p_* \Gamma_A^\st{flat}$.

\subsection{Coisotropic reduction}\label{subsec:coisored}
Let $\cM$ be a degree 2 symplectic  $\mathbb{N}$-manifold corresponding to the pseudo-euclidean vector bundle $E\to M$.
Let $\cN$ be a coisotropic submanifold of $\cM$ corresponding to the geometric data $(N,K,F,\nabla)$, as in Theorem~\ref{thm:coisotropic}. 
Consider $(C_\cN)_{bas}$, as in \eqref{eq:cbas} and Prop.~\ref{prop:basic}.

\begin{thm}\label{thm:coisored}
The following are equivalent:
\begin{itemize}
\item[(a)] $F$ is simple and $\nabla$ has trivial holonomy;
\item[(b)] there is a  degree 2 symplectic $\N$-manifold $\cMred$ and a
surjective submersion $(p,p^\sharp):\cN \to \cMred$ such that
$p^\sharp \colon C_{\cMred}\to p_*(C_\cN)_{bas}$ is an isomorphism of
sheaves of Poisson algebras.
\end{itemize}

In this equivalence, $\Mred$ (the body of $\cMred$)
is the leaf space of $F$, and the pseudo-euclidean vector bundle corresponding to $\cMred$ is the quotient of $\Equot = K^\perp/K$ with respect to $F$ and $\nabla$.
\end{thm}

 Note  that the properties of $\cMred$ in (b) define it uniquely, up to symplectomorphisms.

\begin{defi}\label{def:coisored}
We will refer to $\cMred$ as the \textit{reduction} of $\cN$.
\end{defi}

\begin{proof}[Proof of Thm.~\ref{thm:coisored}] 
Assume  that (a) holds.  Let $\Mred$ be the leaf space of
$F$, with quotient map $p\colon N\to \Mred$, and  let $\Ered \to \Mred$ be the
quotient of $\Equot=K^\perp/K\to N$ with respect to $F$ and $\nabla$, as in 
Lemma~\ref{lem:hol}. Since $\nabla$ is metric, $\Ered$ inherits a pseudo-euclidean structure.

Let now $\cMred$ be the degree 2 symplectic  $\N$-manifold corresponding to the pseudo-euclidean vector bundle $\Ered$ (as in Theorem~\ref{thm:pseudoeuc}), determined by
\begin{equation}\label{eq:funs}
(C_{\cMred})_0=C^\infty_{\Mred},\; (C_{\cMred})_1= \Gamma_{\Ered},\;
(C_{\cMred})_2= \Gamma_{\mathbb{A}_{\Ered}}.
\end{equation}
We will define a surjective submersion $(p,p^\sharp):\cN \to \cMred$. For that, it suffices to determine $p^\sharp$ in degrees 1 and 2. We set
\begin{align} \label{eq:map2}
& p^\sharp_1 \colon \Gamma_{\Ered}  \stackrel{\sim}{\to}  p_*
\Gamma_{\Equot}^\st{flat} = p_*((C_\cN)_{bas})_1 \subseteq
p_*(C_\cN)_1,\\ \label{eq:map3}
& p_2^\sharp\colon \Gamma_{\mathbb{A}_{\Ered}}\stackrel{\sim}{\to}  p_*\Gamma^\st{flat}_{\mathbb{A}^\nabla_{E_{quot}}} =  p_*((C_\cN)_{bas})_2 \subseteq
p_*(C_\cN)_2,
\end{align}
to be the identifications in \eqref{eq:p1} and \eqref{eq:p2}.
In order for $p_1^\sharp$ and $p_2^\sharp$ to determine a map $p^\sharp \colon C_{\cMred}\to p_*C_\cN$, we still need to check that they satisfy
\begin{equation}\label{eq:deg1comp}
p^\sharp_1(e) . p^\sharp_1(e') =
p^\sharp_2(e . e').
\end{equation}

Consider the natural inclusion
$$
\wedge^2\Gamma^\st{flat}_{\Equot} \to \Gamma^\st{flat}_{\mathbb{A}^\nabla_{E_{quot}}}, \quad e_1\wedge e_2 \mapsto (0, e_1\wedge e_2),
$$
that corresponds to the map $p_*((C_\cN)_{bas})_1.p_*((C_\cN)_{bas})_1 \hookrightarrow p_*((C_\cN)_{bas})_2$  (note that the image of $\wedge^2\Gamma^\st{flat}_{\Equot}$ coincides with the flat sections of $\wedge^2 E_{quot} \subseteq \mathbb{A}^\nabla_{E_{quot}}$ with respect to the connection $\overline{\nabla}$). To verify \eqref{eq:deg1comp} we recall that  $\overline{(Y,\Delta)}= p_2^\sharp((\underline{Y},\underline{\Delta}))$ satisfies $\Delta \circ p_1^\sharp = p_1^\sharp \circ \underline{\Delta}$. It follows that, for local sections $e_1$, $e_2$ and $e$ of $E_{red}$,
\begin{align*}
p_2^\sharp(e_1\wedge e_2)(p_1^\sharp(e))&=p_1^\sharp((e_1\wedge e_2) (e)) = p_1^\sharp(\SP{e_1,e}e_2 - \SP{e_2,e}e_1) \\
& = p^*\SP{e_1,e} p_1^\sharp(e_1) - p^*\SP{e_2,e}p_1^\sharp(e_1) 
\\& = (p_1^\sharp(e_1)\wedge p_1^\sharp(e_2))(p_1^\sharp(e)),
\end{align*}
where we used that, since $\nabla$ is metric, we have
\begin{equation}\label{eq:brk1}
\SP{p^\sharp_1(e),p^\sharp_1(e')}=p^*\SP{e,e'}.
\end{equation}
Therefore \eqref{eq:deg1comp} holds, and
 \eqref{eq:map2} and \eqref{eq:map3} uniquely determine a
morphism 
$
(p,p^\sharp):\cN \to \cMred.
$ 

Since $p^*$ defines an isomorphism
\begin{equation} \label{eq:map1}
p^* \colon C_{\Mred}^\infty  \stackrel{\sim}{\to} p_*(C_N^\infty)_{bas}
= p_*
((C_\cN)_{bas})_0,
\end{equation}
and $p^\sharp_1$ and
$p^\sharp_2$ are isomorphisms onto  $p_*((C_\cN)_{bas})_1$
and $p_*((C_\cN)_{bas})_2$, respectively, we conclude that 
$p^\sharp \colon C_{\cMred}\to p_*(C_\cN)_{bas}$ is an isomorphism, see Cor.~\ref{cor:basic}. It is
also a consequence of Cor.~\ref{cor:basic} that there are local
coordinates in $\cN$ and $\cMred$ with respect to which
$(p,p^\sharp)$ is in the normal form of Cor.~\ref{cor:submersion},
so it is a surjective submersion.

Comparing with \eqref{eq:basicprel}, we see that the fact that $p^\sharp$ preserves Poisson brackets is equivalent to
\eqref{eq:brk1}, the identities
$$
p^*(\pounds_{\underline{Y}}f) = \pounds_Y(p^*f), \qquad {\Delta}(p_1^\sharp e) = p_1^\sharp(\underline{\Delta}(e)),
$$
when $\overline{(Y,\Delta)}=p_2^\sharp(\underline{Y},\underline{\Delta})$, and the straightforward fact that $p_2^\sharp$ preserves commutators.
This completes the proof that (a) implies (b).

For the converse, suppose that (b) holds, and let $\Ered\to \Mred$ be the
pseudo-euclidean vector bundle corresponding to $\cMred$. Then 
\eqref{eq:map1} and \eqref{eq:map2} hold, and \eqref{eq:map1}
implies that the fibres of $p$ are the leaves of $F$, so $F$ is
simple. On the other hand, \eqref{eq:map2} induces a bundle map
$\Equot\to \Ered$ covering $p:N\to \Mred$  which defines an
identification $\Equot\cong p^*\Ered$. The fact that  $\nabla$ has trivial holonomy then follows from Lemma~\ref{lem:hol}.
\end{proof}

\begin{ex}\label{ex:coisoredaction}
Consider a $G$-equivariant, pseudo-euclidean vector bundle $E\to M$. Let $K\to N$ be a $G$-invariant isotropic subbundle, and
suppose that the $G$-action on $N$ is free and proper. Consider the associated coisotropic geometric data $(N, K, F, \nabla)$, as in Example~ \ref{ex:coisodataaction}. In this case the foliation $F$ is simple and $\nabla$ has trivial holonomy, following Example~\ref{ex:equivquot}; when $G$ is connected, coisotropic reduction yields the pseudo-euclidean vector bundle $E_{red}= E_{quot}/G$, for $E_{quot}=K^\perp/K$. \hfill $\diamond$
\end{ex}

%%%%%%%%%%%%%%%%%%%%%%%%%%%%%%%%%%%%%%%%%%%%%%%%%%%%%%%%%%%%%
\section{Coisotropic reduction of Courant algebroids}\label{sec:redCA}

\subsection{Courant algebroids}\label{subsec:CA}
Let $\cM$ be a degree 2 symplectic  $\N$-manifold, and let
$(E,\SP{\cdot,\cdot})$ be the corresponding pseudo-euclidean vector bundle, as in Thm.~\ref{thm:pseudoeuc}. Suppose that
$\cM$ is equipped with a global section of $C_\cM$ of degree 3, $\Theta
\in C(\cM)_3$, satisfying
\begin{equation}
\{\Theta,\Theta\}=0.
\end{equation}
We will refer to such $\Theta$ as a {\em Courant function}.
The terminology is motivated by the fact that, in terms of the
pseudo-euclidean vector bundle $E$, a Courant function $\Theta$ is equivalent to a \textit{Courant algebroid} structure on $E$ \cite{lwx}, see
\cite[Thm.~4.5]{royt:graded}. In other words, $\Theta$ is equivalent to a
bundle map $\rho \colon E\to TM$, referred to as the {\em anchor}, and a bilinear bracket
$\Cour{\cdot,\cdot} \colon \Gamma(E)\times \Gamma(E)\to \Gamma(E)$
satisfying
\begin{enumerate}
\item[(C1)] $\Cour{e_1,\Cour{e_2,e_3}}=\Cour{\Cour{e_1,e_2},e_3} + \Cour{e_2,
\Cour{e_1,e_3}}$,
\item[(C2)] $\Cour{e_1,fe_2}=f \Cour{e_1,e_2} + (\pounds_{\rho(e_1)}f)e_2$,
\item[(C3)] $\pounds_{\rho(e_1)}\SP{e_2,e_3} = \SP{\Cour{e_1,e_2},e_3} +
\SP{e_2,\Cour{e_1,e_3}}$,
\item[(C4)] $\rho(\Cour{e_1,e_2}) = [\rho(e_1), \rho(e_2)]$,
\item[(C5)] $\SP{e_1,\Cour{e_2,e_2}} = \frac{1}{2}\pounds_{\rho(e_1)} \SP{e_2,e_2}$,
\end{enumerate}
for  all $e_1,e_2, e_3 \in \Gamma(E)$ and $f\in C^\infty(M)$.
 Explicitly, the relationship between $\Theta$ and $\rho$ and
$\Cour{\cdot,\cdot}$ is given by the following derived-bracket
formulas:
\begin{equation}\label{eq:courstr}
\pounds_{\rho(e)} f = \{\{\Theta, e\},f\},\;\;\;\;
\Cour{e_1,e_2}=\{\{\Theta, e_1\},e_2\},
\end{equation}
for $e, e_1, e_2 \in \Gamma(E)=C(\cM)_1$ and $f\in C^\infty(M)$.

\begin{remark}\label{rem:infCAaut}
Conditions (C2) and (C3) say that, for each $e\in \Gamma(E)$, 
$$
\mathrm{ad}_e\colon = \Cour{e,\cdot}\colon \Gamma(E)\to \Gamma(E)
$$
defines a metric-preserving derivation of $E$, i.e., $(\rho(e),\Cour{e,\cdot}) \in \Gamma(\mathbb{A}_E)$ (cf. \eqref{eq:CDO} and \eqref{eq:CDOpairing}). By (C1),  $\mathrm{ad}_e$ is also a derivation for the bracket, so it defines an infinitesimal automorphism of the Courant algebroid. \hfill $\diamond$
\end{remark}

\begin{ex}[Doubles]\label{ex:LAdouble}
Any Lie algebroid gives rise to a Courant algebroid, referred to as its {\em double}, as we now recall.

Following Example~\ref{ex:coisocotangent}, let $A\to M$ be a vector bundle, and let $\cM=T^*[2]A[1]$. Consider the injective morphism  of $C({A[1]})$-modules and graded Lie brackets,
\begin{equation}\label{eq:fiberlin}
\mathfrak{X}({A[1]})\to C({\cM})[2],
\end{equation}
defined by \eqref{eq:fiberlinsh}. In analogy with the classical case, functions in the image of \eqref{eq:fiberlin} are referred to as ``fiberwise linear'' functions on $T^*[2]A[1]$.

Now suppose that $A\to M$ is a Lie algebroid, and denote by $Q=d_A$ its Lie algebroid differential, regarded as a degree $1$ vector field on $A[1]$ (see Remark \ref{rem:Qstructures}). By means of \eqref{eq:fiberlin}, $Q$ defines a fiberwise linear degree 3 function on 
$\cM$, which is a Courant function (since $[Q,Q]=0$). 

Denoting by  $\rho_A$ and $[\cdot,\cdot]_A$ the Lie-algebroid anchor and bracket, the corresponding Courant algebroid structure on the pseudo-euclidean vector bundle $E = A\oplus A^*$ has anchor $\rho(a,\xi) = \rho_A(a)$ and bracket
$$
\Cour{(a_1,\xi_1),(a_2,\xi_2)} = ([a_1,a_1], \pounds_{a_1}\xi_2 - i_{a_2}d_A \xi_1),
$$
where $\pounds_a = d_Ai_a + i_ad_A$. These Courant algebroids are special cases of ``doubles'' of Lie (quasi-)bialgebroids \cite{lwx,royt:quasi}. For instance, given a $d_A$-closed element $\chi \in \Gamma(\wedge^3 A^*)$, one obtains a more general Courant bracket on $A\oplus A^*$ given by 
$$
\Cour{(a_1,\xi_1),(a_2,\xi_2)} = ([a_1,a_1], \pounds_{a_1}\xi_2 - i_{a_2}d_A \xi_1 + i_{a_2}i_{a_1}\chi).
$$
\hfill $\diamond$
\end{ex}

%%%%%%%%%%%%%%%%%%%%%%%%%%%%%%%%%%%%%%%%%%%%%%%%%%%%%%%%%%%%%%%

\begin{ex}[Exact Courant algebroids]\label{ex:standardCA}
For a manifold $M$, by considering the tangent Lie algebroid $A=TM\to M$ in the previous example, one obtains the {\em standard Courant algebroid} structure on $E=TM\oplus T^*M$, with anchor $\rho$ the natural projection onto $TM$ and 
$$
\Cour{(X,\alpha),(Y, \beta)} = ([X,Y], \pounds_X\beta-i_Y d \alpha).
$$
The degree 2 symplectic  $\mathbb{N}$-manifold corresponding to $TM\oplus T^*M$ is denoted by $T^*[2]T[1]M$, and the Courant function $\Theta$ in this case corresponds to the de Rham differential on $M$ via \eqref{eq:fiberlin}. Any set of local coordinates $\{x^i\}$, $i=1,\dots, n$, on $M$ induces local frames on $TM$ and $T^*M$, which give rise to local coordinates 
$$
\{x^1,\ldots,x^n, v^1,\ldots,v^n,\xi_1,\ldots,\xi_n, p_1,\ldots,p_n\}
$$ 
on $T^*[2]T[1]M$, of respective degrees $0,1,1$ and $2$,
satisfying
$$
\{ p_j, x^i\}= \delta_{ij},\;\; \{v^i,\xi_j\}=\delta_{ij} ,\;\; $$
while the other brackets vanish. In these coordinates, $\Theta = v^ip_i$.
More generally, for a closed $\chi \in \Omega^3(M)$, locally written as $\frac{1}{6}\chi_{ijk} dx^idx^jdx^k$, we have that  $\Theta_\chi = v^ip_i + \frac{1}{6}\chi_{ijk} v^iv^jv^k$ is a Courant function that defines the $\chi$-twisted Courant bracket on $TM\oplus T^*M$ \cite{sw}; these types of Courant algebroids are known as {\em exact}. 

We recall that $\chi$-twisted Courant brackets admit interesting symmetries called gauge transformations \cite{sw} (also known as B-field transforms \cite{Gu2}). Any 2-form $B \in \Omega^2(M)$ gives rise to a vector-bundle automorphism 
\begin{equation}\label{eq:Bfields}
TM\oplus T^*M \to TM\oplus T^*M, \qquad (X,\alpha)\mapsto (X, \alpha + i_XB) 
\end{equation}
that preserves the $\chi$-twisted Courant structure if and only if $dB=0$. From a graded-geometric viewpoint, as explained in \cite[\S 4]{royt:quasi}, $B$ is viewed as a degree 2 function on $T^*[2]T[1]M$ (since $\Omega^2(M)\subseteq \Gamma(\wedge^2 E)\subseteq \Gamma(\mathbb{A}_E)$), so it defines a degree 0 hamiltonian vector field $\{B,\cdot\}$ whose time 1 flow corresponds to the pseudo-euclidean automorphism \eqref{eq:Bfields}; the condition 
$\{B,\Theta_\chi\}=0$ (saying that the hamiltonian flow of $B$ preserves the Courant function $\Theta_\chi$) is equivalent to $dB=0$. 
\hfill $\diamond$
\end{ex}

%%%%%%%%%%%%%%%%%%%%%%%%%%%%%%%%%%%%%%%%%%%%%%%%%%%%%%%%%%%%%%%%%%%%%%%%%%%%%
\subsection{Reducible Courant functions}\label{subsec:reducible}

Let $\cN$ be a coisotropic submanifold of $\cM$, and
suppose that $\cMred$ is the reduction of $\cN$ (in the sense of Def.~\ref{def:coisored}). Any function $S$ on $\cM$ that is a section 
of the subsheaf $\mathfrak{N}_{\cI} \subseteq C_\cM$ defines a section of $(C_\cN)_{bas}$ (see \eqref{eq:normalizer} and \eqref{eq:bas}), and hence a
function $S_{red}$ of the reduction $\cMred$ (via the identification in Theorem~\ref{thm:coisored}, part (b)). 
\begin{defi}\label{def:reducfn}
We refer to functions on $\cM$ that are sections of $\mathfrak{N}_\cI$
as \textit{reducible}.
\end{defi}

Let $\Theta$ be a Courant function on $\cM$, with corresponding Courant algebroid defined as in \eqref{eq:courstr}.
If $\Theta$ is reducible, then $\Theta_{red}$ is a Courant function on $\cMred$, which is in turn equivalent to  a Courant algebroid over $\Mred$. Our next goal is to express the reducibility condition for a Courant function in geometric terms, so as to obtain a geometric reduction procedure for Courant algebroids.

Let $\cN$ be a coisotropic submanifold of $\cM$ with sheaf of vanishing
ideals $\cI$ and corresponding geometric data given by $(N,K,F,\nabla)$ (as in Thm.~\ref{thm:coisotropic}).
Recall the notation $\Gamma^\st{flat}_\st{E, K^\perp}$ from \eqref{eq:Snabla}.

\begin{thm}\label{thm:reducible}
A Courant function $\Theta \in C(\cM)_3$ is reducible (i.e., $\{\Theta, \cI\} \subseteq \cI$) if and only if
the following conditions hold:
\begin{enumerate}
\item[(R1)] $\rho(K^\perp)\subseteq TN$,
\item[(R2)] $\rho(K)\subseteq F$,
\item[(R3)] $[\rho(\Gamma^\st{flat}_\st{E,K^\perp}),\Gamma_\st{TM,F}]\subseteq \Gamma_\st{TM,F}$,
\item[(R4)] 
$\Gamma^\st{flat}_\st{E,K^\perp}$ is involutive with respect to the Courant bracket.
\end{enumerate}
\end{thm}
Notice that when $F$ is simple, (R3) is saying that $\rho(\Gamma^\st{flat}_\st{E,K^\perp})|_N$
consists of vector fields that are projectable with respect to the quotient map $N\to \underline{N}$.

For the proof, we need some additional observations about the vanishing ideal $\cI$ of $\cN$.
First, we have the following alternative description of $\cI_2$:
\begin{lemma}\label{lem:I2}
For $U\subseteq M$ open,
\begin{equation}\label{eq:I2equiv}
\cI_2(U) = \{ (X,D) \in \Gamma_{\mathbb{A}_E}(U)\,|\, X\in \Gamma_\st{TM,F}(U),
\, D(\Gamma^\st{flat}_\st{E,K^\perp}(U))\subseteq \Gamma_\st{E,K}(U) \}.
\end{equation}
\end{lemma}
To compare  with \eqref{eq:I2claim}, note that the condition
$D(\Gamma^\st{flat}_\st{E,K^\perp}(U))\subseteq \Gamma_\st{E,K}(U)$ implies that $D(\Gamma_\st{E,K}(U))\subseteq \Gamma_\st{E,K}(U)$, so $(X,D)$
must lie in $\Gamma_{\mathbb{A}_E}^\st{N,K}(U)$. Note also that the same condition says that $[D_N](\Gamma^\st{flat}_{E_{quot}})=0$ (where $[D_N]$ is the image under \eqref{eq:maptoquot} of the restriction of $D$ to $N$). By Remark~\ref{rem:flat},  the descriptions of $\cI_2$ in \eqref{eq:I2claim} and \eqref{eq:I2equiv} are equivalent.

Let $\fl_F$  denote the subsheaf of $C_M^\infty$ consisting  of functions whose
restrictions to $N$ are basic relative to $F$ (i.e., in
$(C^\infty_N)_{bas}$). It now follows from \eqref{eq:I2equiv}  that $\tilde{e}=(X,D) \in \cI_2(U)$ if and only if
\begin{eqnarray}\label{eq:I2a}
\{\tilde{e},\fl_F(U)\} = \pounds_{X}(\fl_F(U)) \subseteq \fl_N(U),\\
\label{eq:I2b} \{\tilde{e},\Gamma^\st{flat}_\st{E,K^\perp}(U)\} =
D(\Gamma^\st{flat}_\st{E,K^\perp})\subseteq \Gamma_\st{E,K}(U).
\end{eqnarray}

We will also need the following characterization of $\cI_3$.

\begin{lemma}\label{lem:I3}
Given an open subset $U\subseteq M$, $T\in \cI_3(U)$ if and only if 
%$\{T,(\mathfrak{N}_\cI)_1(U)\} = 
$\{T,\Gamma^\st{flat}_\st{E,K^\perp}(U)\} \subseteq
\cI_2(U)$. 
\end{lemma}
\begin{proof}
Recall that $\Gamma^\st{flat}_\st{E,K^\perp} = (\mathfrak{N}_\cI)_1$, see Lemma \ref{lem:NI}. So $T\in \cI_3(U)$ implies that $\{T,\Gamma^\st{flat}_\st{E,K^\perp}(U)\} \subseteq
\cI_2(U)$.
Conversely, we now check that $\{T,\Gamma^\st{flat}_\st{E,K^\perp}(U)\}
\subseteq \cI_2(U)$ implies that $T\in \cI_3(U)$. It
suffices to assume that we have local coordinates $\{x^i, e^\mu,
p^I\}$ over $U$, so that $\cI$ is generated by $x^1,\ldots,
x^{r_0}$, $e^1,\ldots, e^{r_1}$, $p^1,\ldots, p^{r_2}$ (where
$r_0=\mathrm{codim}(N)$, $r_1=\mathrm{rank}(K)$, and
$r_2=\mathrm{rank}(F)$). Any $T \in (C_\cM(U))_3$ is written in
coordinates as
\begin{equation}\label{eq:T}
T = s_I p^I + A_{\mu \nu \eta}(x) e^\mu e^\nu e^\eta,
\end{equation}
where $s^I = B^I_\mu(x) e^\mu$. Given $k^\perp \in
\Gamma^\st{flat}_\st{E,K^\perp}(U)$, we have
\begin{equation}\label{eq:bracketT}
\{T,k^\perp\} = p^I\{s_I,k^\perp \} - \{p^I,k^\perp \}s_I + \{A_{\mu
\nu \eta}(x) e^\mu e^\nu e^\eta,k^\perp\} \in \cI_2(U).
\end{equation}
As a consequence, $\{s_I,k^\perp \} \in \cI_0(U)$ for $I>r_2$, i.e.,
$s_I \in \Gamma_\st{E,K}(U)=\cI_1(U)$ for $I>r_2$. Since $p^I \in \cI_2(U)$
for $I\leq r_2$, it follows that $\{p^I,k^\perp \} \in
\Gamma_\st{E,K}(U)=\cI_1(U)$ if $I\leq r_2$. So the first two terms on the
right-hand side of \eqref{eq:bracketT} lie in $\cI_2(U)$, which
forces the term $\{A_{\mu \nu \eta}(x) e^\mu e^\nu e^\eta,k^\perp\}$
to lie in $\cI_2(U)$. But this happens if and only if $A_{\mu \nu
\eta}\in \cI_0(U)$ whenever $\mu, \nu, \eta > r_1$. From the
expression in \eqref{eq:T}, we conclude that $T\in \cI_3(U)$.
\end{proof}

\begin{proof}[Proof of Thm.~\ref{thm:reducible}] 
Recall that $\cI_0 = \fl_N$, $\cI_1 = \Gamma_\st{E,K}$, and $\cI_2$ is given
as in \eqref{eq:I2claim}. The reducibility of $\Theta$ is equivalent
to
\begin{itemize}
\item[$(a)$] $\{\Theta,\cI_0\}\subseteq \cI_1$,
\item[$(b)$] $\{\Theta,\cI_1\}\subseteq \cI_2$,
\item[$(c)$] $\{\Theta,\cI_2\}\subseteq \cI_3$.
\end{itemize}
 For each open $U\subseteq M$, note that $e\in \cI_1(U)=\Gamma_\st{E,K}(U)$ if and
only if $\SP{e,\Gamma_\st{E,K^\perp}(U)}=\{e,\Gamma_\st{E,K^\perp}(U)\}=0$. So $(a)$ is equivalent (on $U$)
to
$$
%\{\{\Theta,\cI_0(U)\},\fj_{K^\perp}(U)\}=\cL_{\rho(e)} \fl_N(U) = 0,
\{\{\Theta,\cI_0(U)\},e'\}=\pounds _{\rho(e')} \fl_N(U) = 0,
$$
for all $e'\in \Gamma_\st{E,K^\perp}(U)$,
i.e., $\rho(K^\perp)\subseteq TN$, which is condition (R1) in the statement of the theorem.

By \eqref{eq:I2a} and \eqref{eq:I2b}, we see that $(b)$ is equivalent (on $U$) to the following two conditions:
\begin{align*}
&\{\{\Theta,\cI_1(U)\}, f \}\subseteq
\fl_N(U)\, \forall\, f\in C^\infty_M(U) \text{ such that } f|_N \in
(C^\infty_{N})_{bas}(U\cap N),\\ 
&\{\{\Theta,\cI_1(U) \},
\Gamma^\st{flat}_\st{E,K^\perp}(U) \}\subseteq \Gamma_\st{E,K}(U).
\end{align*}
By
\eqref{eq:courstr}, the former is equivalent to   condition (R2), while the
latter is equivalent to  
\begin{equation}\label{eq:extracond(3)}
\Cour{\Gamma_\st{E,K}, \Gamma^\st{flat}_\st{E,K^\perp}}\subseteq \Gamma_\st{E,K}.
\end{equation}
 
From Lemma~\ref{lem:I3}, along with \eqref{eq:I2a} and \eqref{eq:I2b}, we see
that $(c)$ is equivalent to
\begin{align}
\{\{\{\Theta,\cI_2(U)\},\Gamma^\st{flat}_\st{E,K^\perp}(U)\},\fl_F(U)\}
\subseteq
\fl_N(U),\label{eq:c1}\\
 \{
\{\{\Theta,\cI_2(U)\},\Gamma^\st{flat}_\st{E,K^\perp}(U)\},\Gamma^\st{flat}_\st{E,K^\perp}(U)\}
\subseteq \Gamma_\st{E,K}(U).\label{eq:c2}
\end{align}
We first consider \eqref{eq:c1}. Let $\tilde{e}=(X,D)\in \cI_2(U)$, $e\in \Gamma^\st{flat}_\st{E,K^\perp}(U)$
and $f\in \fl_F(U)$. By the graded skew-symmetry and Jacobi identity
for $\{\cdot,\cdot\}$, we have that
\begin{equation*}
\{\{\Theta,\tilde{e}\}, e\} =
\{\Theta,\{\tilde{e},e\}\}-\{\tilde{e},\{\Theta, e\}\} =
\{\{\tilde{e},e\},\Theta\}+\{\{\Theta, e\},\tilde{e}\},
\end{equation*}
and
\begin{align*}
\{\{\{\Theta,\tilde{e}\}, e\},f\} & = \pounds  _{\rho(\{\tilde{e},e\})}f +
\{ \{\Theta,e\}, \{ \tilde{e}, f\} \} - \{\tilde{e},\{ \{\Theta,
e\},
f\}\}\\
&= \pounds_{\rho(\{\tilde{e},e\})}f + \pounds_{\rho(e)}(\pounds_X f) -
\pounds_{X}(\pounds_{\rho(e)}f).
\end{align*}
By \eqref{eq:I2b}, $\{\tilde{e},e\}\in \Gamma_\st{E,K}(U)$, so it follows from   condition
(R2) that $\pounds_{\rho(\{\tilde{e},e\})}f\in \fl_N(U)$.  So
$\{\{\{\Theta,\tilde{e}\}, e\},f\}\in \fl_N(U)$ if and only if
$\pounds_{[\rho(e),X]}f\in \fl_N(U)$, i.e., $[\rho(e),X] \in \Gamma_\st{TM,F}(U)$
for arbitrary $e \in \Gamma^\st{flat}_{K^\perp}(U)$ and $X \in \Gamma_\st{TM,F}(U)$.
So \eqref{eq:c1} is equivalent to condition (R3).

 In view of \eqref{eq:c2}, let us now consider $\tilde{e}=(X,D)\in \cI_2(U)$ and $e_1, e_2 \in
\Gamma^\st{flat}_\st{E,K^\perp}(U)$. Then
\begin{align*} 
\{ \{\Theta,\tilde{e}\},e_1\}, e_2  \}&= \{
\{\Theta,\{\tilde{e},e_1\}\},e_2\} - \{ \{\tilde{e},\{\Theta, e_1
\}\}, e_2\}\\ 
&= \Cour{D(e_1), e_2} - (\{\tilde{e},\{ \{\Theta,e_1 \}, e_2\}\} -\{
\{\Theta, e_1\}, \{\tilde{e},e_2\} \})\\ 
&= \Cour{D(e_1), e_2} - \{\tilde{e},\{ \{\Theta,e_1 \}, e_2\}\} +
\Cour{e_1, D(e_2)}.
\end{align*}
By \eqref{eq:I2b} and \eqref{eq:extracond(3)}, we see that both $\Cour{D(e_1),
e_2}$ and $\Cour{e_1, D(e_2)}$ belong to $\Gamma_\st{E,K}(U)$. So the
condition $\{ \{\Theta,\tilde{e}\},e_1\}, e_2  \}\in \Gamma_\st{E,K}(U)$ is
equivalent to
$$
\{\tilde{e},\{ \{\Theta,e_1 \}, e_2\}\} = D(\Cour{e_1,e_2}) \in
\Gamma_\st{E,K}(U).
$$
This holds for arbitrary $\tilde{e}\in \cI_2(U)$ if and only if
$\Cour{e_1,e_2}\in \Gamma^\st{flat}_\st{E,K^\perp}(U)$. Indeed, note first that
$\Cour{e_1,e_2}\in \Gamma_\st{E,K^\perp}(U)$, since, for any $e_3\in
\Gamma_\st{E,K}(U)$,
$$
\SP{\Cour{e_1,e_2},e_3} = \pounds_{\rho(e_1)}\SP{e_2,e_3} -
\SP{e_1,\Cour{e_2,e_3}}\in \fl_N(U),
$$
as a consequence of $\SP{e_2,e_3}\in \fl_N(U)$,   and conditions (R1) and \eqref{eq:extracond(3)}.
Then
$D(\Cour{e_1,e_2}) \in \Gamma_\st{E,K}(U)$ for all $(X,D)\in \cI_2(U)$ if and
only if the projection of $\Cour{e_1,e_2}|_N$  to $\Gamma_{\Equot}$
is $\nabla$-flat (see \eqref{eq:I2claim}), which is equivalent to $\Cour{e_1,e_2}\in
\Gamma^\st{flat}_\st{E,K^\perp}(U)$. So \eqref{eq:c2} is equivalent to   condition (R4).

Summing up, we have shown that $\Theta$ is reducible if and only if conditions (R1)--(R4) (in the statement of the theorem) as well as \eqref{eq:extracond(3)} hold. We finally observe that \eqref{eq:extracond(3)} is redundant, since it is implied by condition (R4). Indeed, let  $e_1\in \Gamma_\st{E,K}(U)$ and $e_2\in \Gamma^\st{flat}_\st{E,K^\perp}(U)$, and notice that
$\Cour{e_1,e_2}=-\Cour{e_2,e_1}$ by the orthogonality of $K$ and $K^{\perp}$. For any $e_3\in \Gamma^\st{flat}_\st{E,K^\perp}(U)$ we have 
$$
\langle \Cour{e_1,e_2}, e_3\rangle = - \langle \Cour{e_2,e_1}, e_3\rangle=0
$$ 
by   property (C3) in the definition of Courant algebroid and condition (R4). This implies that
$$
\langle \Cour{\Gamma_\st{E,K}, \Gamma^\st{flat}_\st{E,K^\perp}}, \Gamma_\st{E,K^\perp}\rangle =0,
$$
since around every point of $N$ there exists a frame of $K^{\perp}$ consisting of sections that project to $\nabla$-flat sections of $K^\perp/K$ (by the flatness of the $F$-connection $\nabla$).
\end{proof}

\begin{remark}\label{rem:rhoKF}
Regarding the reducibility conditions (R1)--(R4), in case the following stronger version of (R2) holds,
$$
\rho(K)=F,
$$ 
then (R3) becomes redundant: it is a consequence of \eqref{eq:extracond(3)} and the fact that $\rho\colon E\to TM$ preserves brackets. \hfill $\diamond$
\end{remark}

Let us now consider a lagrangian submanifold $\cN$ of $\cM$, defined
by $K\to N$, where $K\subset E|_N$ is a lagrangian
subbundle, see Cor.~\ref{cor:lag}.

The following is  a special case of Thm.~\ref{thm:reducible}.

\begin{cor}\label{prop:lagdirac}
Let $\Theta \in C(\cM)_3$ be a Courant function.
The following are equivalent:
\begin{itemize}
\item[(a)] $\Theta$ is reducible      
\item[(b)] $\rho(K)\subseteq TN$ and, for any sections $e_1$, $e_2 \in \Gamma(E)$ with $e_1|_N, e_2|_N \in \Gamma(K)$, we have
$\Cour{e_1,e_2}|_N\in \Gamma(K)$.
\end{itemize}
\end{cor}

\begin{proof}
 Since
$F=TN$, conditions (R1) and (R2) of Thm.~\ref{thm:reducible} reduce to
$\rho(K)\subseteq TN$. Since $K^\perp=K$, we see that
$\Gamma^\st{flat}_\st{E,K^\perp} = \Gamma_\st{E,K}$, so condition (R4) in Thm.~\ref{thm:reducible}  is equivalent to the second condition in
(b). Condition (R3) in Thm.~\ref{thm:reducible} is automatically
satisfied since $F=TN$ and $\rho(K)\subseteq TN$.
\end{proof}

We are led to the following notion,  see, e.g.,
\cite{alekseevxu,bis,severa:sometitle}.

\begin{defi}\label{def:diracsupp}
A lagrangian subbundle $K\to N$ of a Courant algebroid $E\to M$ satisfying condition (b) of Corollary~\ref{prop:lagdirac} is called a \textit{Dirac structure
supported on $N$}.
\end{defi}

 When $N=M$, we have that $K$ is a \textit{Dirac structure} in $E$ in
the sense of \cite{lwx}, extending the original definition of
\cite{courant} when $E=\mathbb{T}M$.

According to the previous corollary, Dirac structures with support are equivalent to lagrangian submanifolds of $\cM$ tangent to the
hamiltonian vector field of the Courant function $\Theta$.

%%%%%%%%%%%%%%%%%%%%%%%%%%%%%%%%%%%%
\subsection{Reduction of Courant algebroids} \label{subsec:cred}

Using the equivalence between Courant functions and Courant algebroids, we can derive from Thm. \ref{thm:reducible} a reduction procedure for Courant algebroids. We will express this procedure in classical geometrical terms.

Let $E\to M$ be a Courant algebroid, with pseudo-euclidean structure $\SP{\cdot,\cdot}$, anchor $\rho$ and bracket $\Cour{\cdot,\cdot}$. Let $(N, K, F, \nabla)$ be geometric coisotropic data, i.e., 
\begin{itemize}
\item $N$ is a submanifold of $M$, 
\item $K\subseteq E|_N$ is an isotropic subbundle, 
\item $F\subseteq TN$ is an integrable subbundle, and 
\item $\nabla$ is a metric, flat  $F$-connection on the vector bundle $E_{quot}= K^\perp/K$ over $N$. 
\end{itemize}

When $F$ is simple and $\nabla$ has trivial holonomy, we have a surjective submersion $p\colon N \to \Mred$ onto the leaf space of $F$ and a quotient of $E_{quot}$ with respect to $F$ and $\nabla$ (see Lemma~\ref{lem:hol}), which is a pseudo-euclidean vector bundle $E_{red} \to \Mred$. Then $E_{quot} = p^* E_{red}$, and we denote by $p^\sharp\colon \Gamma_{E_{red}}\to p_*\Gamma_{E_{quot}}$ the pullback map of sections.

\begin{thm}[Coisotropic reduction of Courant algebroids]\label{cor:courred} 
Let $E\to M$ be a Courant algebroid and $(N, K, F, \nabla)$ as above. Suppose that $F$ is simple and
$\nabla$ has trivial holonomy.
If conditions (R1)--(R4) in
Thm.~\ref{thm:reducible}  hold, then the pseudo-euclidean vector bundle $\Ered$ inherits a Courant algebroid structure,  with anchor $\rho_{red}$ and bracket $\Cour{\;,\;}_{red}$, given as follows: for $\underline{e}$, $\underline{e}_1$, $\underline{e}_2$ sections of $E_{red}$ and $f$ of $C^\infty_{\underline{N}}$, 
\begin{equation}\label{eq:courantred}
%p^*(\pounds_{\rho_{red}(\underline{e})}f)= \pounds_{\rho(e)} (p^*f),
\rho_{red}(\underline{e}) =p_*(\rho(e)),
\qquad\quad
p^\sharp\Cour{\underline{e}_1,\underline{e}_2,}_{red}= [\Cour{e_1,e_2}|_N ],
\end{equation}
where $e$ is a section of $K^\perp$ such that $[e]=p^\sharp \underline{e}$, and $e_i$ is a section of $\Gamma^\st{flat}_\st{E,K^\perp}$ such that $[e_i|_N]=p^\sharp \underline{e}_i$, $i=1,2$.
\end{thm}

This theorem can be proven by directly showing that conditions (R1)--(R4) in Thm.~\ref{thm:reducible} guarantee  that the bracket and anchor can be reduced. But using graded geometry, we will see that the proof is immediate.

\begin{proof}
We know that the Courant algebroid $E\to M$ is equivalent to
a Courant function $\Theta$ on a symplectic degree 2 $\mathbb N$-manifold $\cM$, and that the quadruple $(N, K, F, \nabla)$ defines a coisotropic submanifold. Further $\Theta$ is reducible by Thm.~\ref{thm:reducible}, and $\Theta_{red}$ is a Courant function on the reduction $\cMred$,
which is in turn equivalent to a Courant algebroid structure on $E_{red}\to \Mred$. The expressions for the reduced anchor and bracket in \eqref{eq:courantred} follow from the derived-bracket formulas  \eqref{eq:courstr} and the definition of $\{\cdot,\cdot\}_{bas}$, see \eqref{eq:poisbasic} and \eqref{eq:basicprel}. 
\end{proof}

The previous theorem motivates the following notion.

\begin{defi}\label{def:reduc}
We say that a Courant algebroid $(E, \SP{\cdot,\cdot}, \rho, \Cour{\cdot,\cdot})$ is {\em reducible} with respect to geometric coisotropic data $(N, K, F, \nabla)$ (see~Definition \ref{def:quadruples}) if conditions (R1)--(R4) in Thm. \ref{thm:reducible} hold.
\end{defi}

We illustrate Thm. \ref{cor:courred} with some special cases. 
%The first case is elementary, and the remaining two extend it. 
We keep the setup and notation of the theorem.
\begin{ex}
\label{ex:casescoisored} \
\begin{itemize}
    \item[(i)] Let $N\subseteq M$ be a submanifold, and set $K$, $F$ and $\nabla$ to be trivial (i.e., the corresponding coisotropic submanifold in Thm. \ref{thm:coisotropic} has only constraints in degree 0). Then the reducibility conditions (R1)-(R4) reduce to the single condition that $\rho(E)\subseteq TN$. In this case, the theorem boils down to the simple fact that the restriction $E|_N$ inherits a Courant algebroid structure. 
\item [(ii)] Consider a submanifold $N\subset M$ and an isotropic subbundle $K\subset E|_N$, setting $F$ and $\nabla$ to be trivial. The reducibility conditions in the theorem become
 $\rho(K^{\perp})\subseteq TN$, $\rho(K)=0$, and  that $\Gamma_\st{E,K^{\perp}}$ is closed under the Courant bracket.  
 Thm. \ref{cor:courred} then says that there is an induced Courant algebroid structure on $K^{\perp}/K\to N$. This recovers the result in \cite[Prop. 2.1]{LBMeinrenken}, which has as a special case the
 pullback of Courant algebroids to  submanifolds transverse to the anchor \cite[Prop 2.4]{LBMeinrenken}.
\item[(iii)] If we set $K$ to be zero in Thm. \ref{cor:courred},
%(i.e. the  associated coisotropic submanifold   has no constraints in degree $1$),
we obtain the following statement: let $N\subseteq M$ be a submanifold with a regular involutive distribution $F$ and a flat  metric $F$-connection on $E|_N$. Assume that $\rho(E)\subseteq TN$, and that the restricted Courant algebroid  $E|_N$ is such that
 $[\rho(\Gamma^\st{flat}_\st{E|_N}),\Gamma_\st{F}]\subseteq \Gamma_\st{F}$ and 
$\Gamma^\st{flat}_\st{E|_N}$ is involutive.
% $[\rho(\Gamma^\st{flat}_\st{E}),\Gamma_\st{TM,F}]\subseteq \Gamma_\st{TM,F}$, and that$\Gamma^\st{flat}_\st{E}$ is involutive.
Then, whenever $F$ is simple and $\nabla$ has trivial holonomy, there is an induced Courant algebroid structure on $E_{red}\to \underline{N}$, the quotient of $E|_N$ with respect to $F$ and $\nabla$. 
\item[(iv)] As a simple special case of (iii), suppose that a (connected) Lie group $G$ acts on the Courant algebroid $E\to M$ by automorphisms. Let $N$ be a $G$-invariant submanifold of $M$  where the action is free and proper, and such that $\rho(E)\subseteq TN$; 
the $G$-action gives rise to a flat  metric partial connection on $E|_N$, see Example~\ref{ex:coisodataaction}, satisfying the conditions in (iii). It follows that $E_{red}= (E|_N)/G \to N/G$ has an induced Courant algebroid structure. 
\end{itemize}
\hfill $\diamond$
\end{ex}

The following example compares quotients of Lie algebroids with reduction of their Courant doubles.

\begin{ex}\label{ex:doublered}
Given a vector bundle $A\to M$, we saw in Example~\ref{ex:coisocotangent} that an involutive distribution $\mathcal{D}$ on $A[1]$ defines a coisotropic submanifold $\mathcal{N}_\cD$ in $\cM=T^*[2]A[1]$. Let $Q$ be a degree 1 homological vector field $Q$ on $A[1]$, codifying a Lie algebroid structure on $A$, with anchor $\rho_A$ and bracket $[\cdot,\cdot]_A$ (Remark~\ref{rem:Qstructures}). Following Example~\ref{ex:LAdouble}, $Q$ defines a Courant function $\Theta$ on $\cM$ via \eqref{eq:fiberlin}, and  this Courant function is reducible with respect to $\cN_\cD$ if and only if $\cD$ is $Q$-invariant,
\begin{equation}\label{eq:projectable}
[Q,\cD]\subseteq \cD.
\end{equation}
As in classical geometry, we interpret this condition as saying that $Q$ is  ``projectable'' with respect to $\cD$; indeed it is proven in \cite[Prop.~3]{ZZL2} that, under suitable regularity conditions, one can obtain a new $\mathbb{N}$-manifold of degree 1 $A_{red}[1]$ as the quotient of $A[1]$ by $\cD$ in such a way that $Q$ projects onto a degree 1 homological vector field $Q_{red}$, thereby defining a quotient Lie algebroid $A_{red}$. In this case, the coisotropic reduction of $\cN_\cD$ agrees with $T^*[2]A_{red}[1]$, and the reduction of $\Theta$ coincides with the Courant function defined by $Q_{red}$.

We now describe this construction in classical geometrical terms. Let $(B, F, \overline{\nabla})$ be the geometric data corresponding to the involutive distribution $\cD$ on $A[1]$, as recalled in Example~\ref{ex:coisocotangent}: $B\to M$ is a subbundle of $A$, $F$ is an involutive distribution on $M$, and $\overline{\nabla}$ is a flat $F$-connection on $A/B\to M$. Then, when $A$ is a Lie algebroid, \cite[Prop.~4]{ZZL2} shows that \eqref{eq:projectable} is equivalent to the following conditions\footnote{Triples $(B, F, \overline{\nabla})$ satisfying conditions (1)--(5) are referred to as {\em infinitesimal ideal systems} in \cite{JotzOrtizFol}, cf. \cite[Def.~4.4.2]{mac:lie}; they arise as infinitesimal counterparts of multiplicative foliations on Lie groupoids \cite{Hawk08,JotzOrtizFol}.}:
\begin{enumerate}
 \item $B$ is a Lie subalgebroid of $A$;
 \item if  $a,a' \in \Gamma^{\st{flat}}(A)$,  then $[a,a']_A\in \Gamma^{\st{flat}}(A)$;
 \item if  $b\in \Gamma(B)$ and $a\in \Gamma^{\st{flat}}(A)$, then $[b,a]_A\in \Gamma(B)$;
 \item the anchor $\rho_A$ maps $B$ into $F$;
 \item if $a \in \Gamma^{\st{flat}}(A)$ then  $[\rho_A(a),\Gamma(F)]\subseteq \Gamma(F)$,
 \end{enumerate}
where $\Gamma^{\st{flat}}(A):=\{a\in \Gamma(A)\,|\, \overline{\nabla}[a]=0 \}$. 
Denoting by $(K, F, \nabla)$ the geometric coisotropic data defined by $(B, F, \overline{\nabla})$ as in Example~\ref{ex:coisocotangent} (recall that $N=M$), one can directly check that conditions (1)--(5) above translate into the Courant reducibility conditions (R1)--(R4) for the double Courant algebroid $E=A\oplus A^*$ of Example~\ref{ex:LAdouble}.

When $F$ is simple and $\overline{\nabla}$ has trivial holonomy, it follows from \cite[Thm.~4.4.3]{mac:lie} that the quotient of $A/B \to M$ with respect to $F$ and $\overline{\nabla}$ (as in Def.~\ref{def:vbquot}), denoted by $A_{red}\to \underline{M}$, inherits a natural Lie algebroid structure. In this case, the reduced Courant algebroid $E_{red}\to \underline{M}$ of Theorem~\ref{cor:courred} coincides with the double $E_{red}=A_{red}\oplus A^*_{red}$ of the quotient Lie algebroid.

\end{ex}

%%%%%%%%%%%%%%%%%%%%%%%%%%%
\subsubsection{\underline{The exact case}}

Given a Courant algebroid
$(E,\SP{\cdot,\cdot},\rho,\Cour{\cdot,\cdot})$ over $M$, there is an
associated chain complex
$$
0\to T^*M \stackrel{\rho^*}{\to} E \stackrel{\rho}{\to} TM \to 0,
$$
where we have used the identification $E\cong E^*$ for the map
$\rho^*$. A Courant algebroid is called \textit{exact} if this
sequence is exact. A choice of lagrangian splitting of this sequence identifies $E$ with a twisted Courant algebroid structure on $TM\oplus T^*M$, with respect to a suitable closed 3-form \cite{sw}; this leads to a classification of 
exact Courant algebroids by elements in $H^3(M)$, the so-called {\em {\v S}evera classes} \cite{severaletters} 
 
 One can refine the
coisotropic reduction of Courant algebroid  to see when the reduction of an exact
Courant algebroid is again exact.

\begin{prop}\label{prop:exactCA}
Consider the same setup and notation as  in Thm.~\ref{cor:courred}, but assuming that $E$ is exact.
The
reduced Courant algebroid  $\Ered \to \Mred$ is exact if and only
if:
\begin{itemize}
\item[(a)] $\dim(N)-\rank(\rho(K^\perp)) = \rank(F)-\rank(\rho(K))$,
\item[(b)] $\rho(K^\perp)\cap F = \rho(K)$.
\end{itemize}
In particular, this holds when $\rho(K^\perp)=TN$ and $\rho(K)=F$.
\end{prop}

\begin{proof}

We will use the following fact: a Courant algebroid is exact if and only if its rank is twice the
dimension of its base manifold and the kernel of its anchor is
isotropic.

Let $m_0=\dim(M)=\frac{1}{2}\rank(E)$ and
$r_1=\rank(K)$. The rank of $E_{red}$ agrees with the rank of
$K^\perp/K$, which is $2(m_0 - r_1)$. By properties (R1) and (R2) of
Thm.~\ref{thm:reducible}, we know that $\rank(TN)\geq
\rank(\rho(K^\perp))$ and $\rank(F)\geq \rank(\rho(K))$. 
Since $E$
is exact, we have an exact sequence
$$
0\to \frac{T^*M\cap K^\perp + K}{K} \to \frac{K^\perp}{K} \to
\frac{\rho(K^\perp)}{\rho(K)} \to 0.
$$
Noticing that $(T^*M\cap K^\perp + K)/K$ is lagrangian in $K^\perp/K$,
we conclude that $\rank(K^\perp/K)= 2 \rank(\rho(K^\perp)/\rho(K)) =
2 (\rank(\rho(K^\perp))-\rank(\rho(K)))$. It follows that $\rank(K^\perp/K)$ equals $2 \dim(\Mred) = 2 (\rank(TN)-\rank(F))$ if and only if
$$
\rank(\rho(K^\perp))-\rank(\rho(K)) = \rank(TN)-\rank(F),
$$
which is condition (a).

If (a) holds, then $E_{red}$ is exact if and only if 
the kernel of the map $K^\perp/K \to TN/F$, given by $C_1=(K^\perp\cap
\rho^{-1}(F))/K$, is isotropic. We will check that this is equivalent to condition (b).

Note that $C_1$ contains the lagrangian subbundle $C_2= (T^*M\cap K^\perp)/K$, so $C_1$ is isotropic if and only if $C_1=C_2$. Since $C_2$ is the kernel of the map $K^\perp/K \to
\rho(K^\perp)/\rho(K)$, the condition $C_1=C_2$ is equivalent to the map $\rho(K^\perp)/\rho(K)\to TN/F$ induced by the inclusion being injective (in fact, an isomorphism by (a)), which is the same as  condition (b).
\end{proof}

\begin{ex}[Reduction of the standard Courant algebroid]\label{red:stdca}
Consider the standard Courant algebroid $E= TM \oplus T^*M$, let $N\subseteq M$ be a submanifold and $F\subseteq TN$ be an involutive subbundle. Assume that $F$ defines a simple foliation with quotient map $N \to \underline{N}$. Then $E$ is reducible with respect to the geometric coisotropic data defined by $(N,F)$ (see Example~\ref{ex:bott}), with $K= F\oplus \mathrm{Ann}(TN)$ and $\nabla = \nabla^{Bott} \oplus (\nabla^{Bott})^*$. Further,
$$
E_{red} =  T\underline{N}\oplus T^*\underline{N},
$$
see Example~\ref{ex:tanquot}. 
This reduction agrees with the construction in \cite[Example~3.9]{Za} and is a particular instance of Example~\ref{ex:JSGK} below. (When $N=M$, it is also a special case of Example~\ref{ex:doublered}.)
\end{ex}

We now see how the reduction of exact Courant algebroids in \cite{Za} is a special case of Theorem \ref{cor:courred}.

\begin{ex} [Reduction by isotropic, involutive subbundles]\label{ex:JSGK}
Consider an exact Cou\-rant algebroid $E$  
 over $M$. Let $K\to N$ be an isotropic subbundle of $E$, and suppose that 
 \begin{itemize}
 \item[(a)] $\rho(K^{\perp}) = TN$ 
 \item[(b)] $K$ is involutive, in the sense that
$$
 \Cour{\Gamma_\st{E,K},\Gamma_\st{E,K}}\subseteq \Gamma_\st{E,K}.
$$
\end{itemize}
We can use $K\to N$ to produce geometric coisotropic data $(N, K, F, \nabla)$  as follows. 
 By the exactness of $E$, we have that
\begin{equation}\label{eq:annTN}
K\cap T^*M= \mathrm{Ann}(\rho(K^\perp)) = \mathrm{Ann}(TN),
\end{equation}
which implies that $\rho(K) \subseteq TN$ has constant rank (since $\mathrm{ker}(\rho|_K)=K\cap T^*M$). It follows from the involutivity condition on $K$ and the bracket-preserving property of $\rho$ that $\rho(K)$ is involutive. We set $F:=\rho(K)$. One can now verify that the expression
$$
\nabla_X[e|_N]=[\Cour{k,e}]
$$ 
defines a 
metric, flat $F$-connection $\nabla$ on $E_{quot}=K^{\perp}/K$,
where $X\in \Gamma(F)$, $k$ is any section of $\Gamma_\st{E,K}$  satisfying $\rho(k)|_N=X$, and $e$ is a section of $\Gamma_\st{E,K^{\perp}}$. (The fact that $\nabla$ is well defined follows as in \cite[Lemma~3.1]{Za}, flatness is a consequence of the Jacobi identity (C1) of Courant brackets, while the metric condition follows from (C3).)

One can also check that $E$ is automatically reducible with respect to $(N, K, F, \nabla)$ (as in Def.~\ref{def:reduc}). Indeed, it remains to verify the reducibility condition (R4) (since (R3) holds as observed in Remark~\ref{rem:rhoKF}). Noticing that, for each open $U\subseteq M$,
$$
\Gamma^\st{flat}_\st{E,K^\perp}(U)=\{e\in \Gamma_\st{E,K^\perp}(U)\,|\, 
\Cour{k,e}\in \Gamma_\st{E,K}(U) \text{ for all $k\in\Gamma_\st{E,K}(U)$}\},
$$
condition (R4) follows from properties (C1), (C3) and (C5) of Courant brackets.

So, assuming that $F$ is simple and $\nabla$ has trivial holonomy, we obtain a reduced exact Courant algebroid $E_{red}$ over $\underline{N}$, the leaf space of $N$ by $F$, as a consequence of  Theorem~\ref{cor:courred} and Proposition~ \ref{prop:exactCA}. This construction recovers \cite[Thm.~3.7]{Za} from a slightly different perspective (see \cite[Prop. 7.1]{Gu2} for a special case).
\hfill $\diamond$
\end{ex}

The following is a particular case of the previous example.

\begin{ex}[Reduction by isotropic, involutive subbundles with an action]\label{ex:exactredact}\

\noindent Let $E\to M$ be an exact Courant algebroid, and $K\to N$ be an isotropic, involutive subbundle with $\rho(K^\perp)=TN$; let $(N, K, F, \nabla)$ be the corresponding geometric coisotropic data, as in the previous example.
Assume further that 
\begin{itemize}
\item a (connected) Lie group $G$ acts on $E$ by Courant-algebroid automorphisms,
\item the $G$-action on $E$ is infinitesimally generated by a bracket-preserving map $\psi\colon \g \to \Gamma(E)$, in the sense that the $\g$-action on $E$ by infinitesimal Courant automorphisms is given by
$$
u \mapsto (\mathrm{ad}_{\psi(u)}= \Cour{\psi(u), \cdot}\colon \Gamma(E)\to \Gamma(E)), 
$$
for $u\in \g$ (see Remark~\ref{rem:infCAaut}), 
\item for each $x\in N$, the map $\psi_x\colon \g \to E|_x$, $u\mapsto \psi(u)(x)$ takes values in $K|_x$ and $\rho(\psi_x(\g))=\rho(K|_x)=F|_x$.
%$u\in \g$, $\psi(u)|_N \in \Gamma(K)$, and $\rho(K|_x)= F|_x = \{ \rho(\psi(u)(x)) \,|\, u\in \g \}$ for all $x\in N$.
\end{itemize}
This last condition and the involutivity of $K$ imply that $K\to N$ is a $G$-invariant subbundle of $E\to M$. Assuming that the $G$-action on $N$ is free and proper, it is a simple verification that the geometric coisotropic data defined by the $G$-action on $K$ (as in Example~\ref{ex:coisodataaction}) coincides with 
$(N, K, F, \nabla)$. Therefore, with this setup, $E$ is reducible with respect to the coisotropic data, $F$ is simple, $\nabla$ has trivial holonomy, and $E_{red} = (K^\perp/K)/G$ is an exact Courant algebroid over $N/G$.

A special case of this construction is described in \cite[$\S$ 3.1]{HKquot} (based on \cite{bcg}), when $K\to N$ and the $G$-action on $E$ are defined by means of an ``extended action'' (see $\S$ \ref{subsec:extactions}). 
\hfill $\diamond$
\end{ex}

%%%%%%%%%%%%%%%%%%%%%%%%%%%%%%%%%%%%%%%%%%%%%%%%%%%%%%%%%%%%
\section{Coisotropic reduction of GC and Dirac structures}
\label{sec:redGCDirac}

%%%%%%%%%%%%%%%%%%%%%%%%%%%%%%%%%%%%%%%%%%%%%%%%%%%%%%%%%%%%%%%%%%%
\subsection{Reduction of GC structures}

Let $\cM$ be a degree 2 symplectic  $\mathbb N$-manifold and $(E, \SP{\cdot,\cdot})$ be the corresponding pseudo-euclidean vector bundle. By a {\em quadratic function} on $\cM$ we mean a section of $(C_\cM)_1\cdot (C_\cM)_1 \subseteq (C_\cM)_2$. With the identification $(C_\cM)_1 =  \Gamma_{E^*}$, quadratic functions are seen as sections of $\Gamma_{\wedge^2 E^*}$,
which are equivalent to skew-symmetric endomorphisms $E\to E$ (i.e., sections of $\mathbb{A}_E$ with vanishing symbol). For simplicity, we will keep the same notation for the quadratic function and the corresponding skew-symmetric endomorphism.

Let $\cN$ be a coisotropic submanifold of $\cM$ defined by $(N,K,F,\nabla)$.

\begin{lemma}\label{lem:Jreducible}
A quadratic function $\J$ on $\cM$ is reducible (i.e., it is a section of $\mathfrak{N}_\cI$) if and only if either of the following equivalent conditions hold:
\begin{itemize}
\item[(a)] $\J(\Gamma^\st{flat}_\st{E,K^\perp})\subseteq \Gamma^\st{flat}_\st{E,K^\perp}$.
\item[(b)] $\J(K)\subseteq K$ and $\nabla \J_{quot} =0$, where
$\J_{quot}$ is the skew-symmetric endomorphism of $E_{quot}= K^\perp/K$
defined by $\J_{quot}([e])=[\J(e)]$, for $e\in K^\perp$.
\end{itemize}
\end{lemma}

Note that, since $\J$ is skew-symmetric, the condition that $\J(K)\subseteq K$ is equivalent to $\J(K^\perp)\subseteq K^\perp$, so $\J_{quot}$ in (b) is well defined. 

\begin{proof}
The fact that the reducibility of $\J$ is equivalent to (a) is an immediate consequence of the description of $(\mathfrak{N}_\cI)_2$ in Lemma~\ref{lem:NI} once one notices that the condition in (a) implies that $\J(K^\perp)\subseteq K^\perp$, and hence $\J(K)\subseteq K$  (so that $\J$ is automatically  a section of $\Gamma^\st{N,K}_{\mathbb{A}_E}$). On the other hand, (a) is equivalent to saying that $\J_{quot}$ preserves flat sections of $K^\perp/K$, which is in turn equivalent to the equality $\J_{quot}\circ \nabla = \nabla \circ \J_{quot}$ (by the flatness of $\nabla$, which ensures the existence of local frames of flat sections), i.e., $\nabla \J_{quot}=0$.
\end{proof}

Let us recall that a \textit{generalized complex (GC) structure} \cite{Gu2}   on a
Courant algebroid $E\to M$ is an endomorphism $\J \colon E \to E$
%preserving the pairing $\SP{\cdot,\cdot}$, 
such that $\J^2=-\mathrm{Id}$, $\J = -\J^*$, and
whose Nijenhuis torsion vanishes,
$$
\Cour{\J(e_1),\J(e_2)}-\J(\Cour{\J(e_1),e_2}) -
\J(\Cour{e_1,\J(e_2)}) + \Cour{e_1,e_2} = 0.
$$

\begin{remark}[Graded-geometric characterization]
Since generalized complex structures $\J: E\to E$ are skew-symmetric, they are particular types of quadratic functions on the corresponding degree 2 symplectic  $\N$-manifold $\cM$; their additional properties can be all characterized in terms of $\{\cdot,\cdot\}$ and the Courant function $\Theta$ by noticing that $\J(e) = \{\J,e\}$ and making use of \eqref{eq:courstr}; see also Remark~ \ref{rem:grab} below.
\hfill $\diamond$
\end{remark}

\begin{thm}\label{thm:redGCS}
Suppose that the Courant algebroid $E$ is reducible with respect to the geometric coisotropic data $(N, K, F, \nabla)$, and let $E_{red}$ be the reduced Courant algebroid.
If a generalized complex structure $\J \colon E \to E$ satisfies 
%$\J(\Gamma^\st{flat}_\st{E,K^\perp})\subseteq \Gamma^\st{flat}_\st{E,K^\perp}$ 
either of the equivalent conditions in Lemma~\ref{lem:Jreducible}, then it gives rise to a generalized complex structure $\J_{red}$ on $E_{red}$ by
\begin{equation}\label{eq:Jred}
p^\sharp \J_{red} (\underline{e}) = [\J(e)],
\end{equation}
where $\underline{e}$ is a section of $E_{red}$ and $e$ is a section of $K^\perp$ such that $[e]=p^\sharp \underline{e}$.
\end{thm}

\begin{proof}
Consider the coisotropic submanifold $\cN$ corresponding to $(N, K, F, \nabla)$, and let $\underline{\cN}$ be its reduction.
By   Lemma \ref{lem:Jreducible}, the condition in the statement says that  the quadratic function $\J$ is reducible, so it defines a function $\J_{red}$ on $\underline{\cN}$, which is easily seen to be quadratic; moreover, the skew-symmetric endomorphisms corresponding to $\J$ and $\J_{red}$ are related as in the statement of this  theorem. To see that $\J_{red}$ is generalized complex, one can make use of  \eqref{eq:courantred} and \eqref{eq:Jred} to show that the almost-complex and Nijenhuis conditions on $\J$ are transferred to $\J_{red}$.
\end{proof}

\begin{remark} One can prove versions of the previous theorem under weaker conditions, e.g., assuming that $\J(K)\cap K^\perp$ has constant rank and is contained in $K$ 
and that $\J$ preserves the flat sections of this bundle, see, e.g., \cite[Prop.~6.1]{Za}.
Graded-geometric interpretations of these results require considering graded submanifolds beyond the coisotropic ones. \hfill $\diamond$
\end{remark}

\begin{remark} \label{rem:grab}
 When $E= TM \oplus T^*M$ is the standard Courant algebroid, it is proven in \cite{grabowski}
that generalized complex structures $\J\colon E\to E$ are equivalent to quadratic functions $\J$ on $\cM$ satisfying
\begin{equation}\label{eq:GCSsuper}
\{\{\Theta,\J\},\J\}=-\Theta.
\end{equation}
In this case, we can give an alternative argument  for the fact that $\J_{red}$ is generalized complex in  Theorem \ref{thm:redGCS}:   condition \eqref{eq:GCSsuper} also  holds for $\Theta_{red}$ and $\J_{red}$, since both the passage from the Lie normalizer $\mathfrak{N}_\cI$ to basic functions on $\cN$ and the identification of the latter with functions on $\underline{\cN}$ preserves Poisson brackets. \hfill $\diamond$
\end{remark}

\begin{ex}[Coisotropic reduction of symplectic structures]\label{ex:coisosympred}
Let $\J$ be the generalized complex structure on the standard Courant algebroid $TM\oplus T^*M$ defined by a symplectic 2-form $\omega$ on $M$, 
\[ \J = \left( \begin{matrix} 
0 & -\omega^{-1} \\ 
\omega & 0 
\end{matrix}\right). \] 
Let $\iota: N\hookrightarrow M$ be a submanifold, and let $F\subseteq TN$ be an involutive subbundle.
As in Example \ref{ex:bott} we obtain geometric coisotropic data by setting  $K= F\oplus \mathrm{Ann}(TN)$ and $\nabla = \nabla^{Bott} \oplus (\nabla^{Bott})^*$. To verify the reducibility of $\J$ with respect to these data, first note
that the isotropic subbundle $K$ is $\J$-invariant if and only if $N$ is a coisotropic submanifold of $(M,\omega)$ and $F=TN^{\omega}$. In this case, the condition  $\nabla \J_{quot}=0$ is automatically satisfied since $F=\mathrm{ker}(\iota^*\omega)$ and the closed 2-form $\iota^*\omega$ on $N$ is invariant under the Lie derivative by vector fields in its kernel. 
Whenever $F$ defines a simple foliation, with quotient map $N\to \underline{N}$, by Theorem \ref{thm:redGCS} we obtain a generalized complex structure $\J_{red}$ on $T\underline{N}\oplus T^*\underline{N}$. As expected, it corresponds to the symplectic form on $\underline{N}$ obtained from the coisotropic reduction of  $N\subseteq (M,\omega)$. 
\hfill $\diamond$
\end{ex}

\begin{ex}[Holomorphic quotients] \label{ex:coisoholred}
Let $J$ be a complex structure on $M$, and let 
\[ \J = \left( \begin{matrix} 
J & 0 \\ 
0 & -J^* 
\end{matrix}\right) \] 
be the corresponding generalized complex structure on $TM\oplus T^*M$. Let $F\subseteq TM$ be an involutive subbundle, and consider the associated geometric coisotropic data with $N=M$, $K=F$ and $\nabla=\nabla^{Bott} \oplus (\nabla^{Bott})^*$. The reducibility of $\J$ in this case amounts to the facts that $JF=F$ and that the induced operator $TM/F \to TM/F$ is parallel with respect to $\nabla^{Bott}$; these conditions are equivalent to saying that $F$ defines a holomorphic foliation in $(M,J)$ (i.e., the leaves of $F$ can be locally identified with fibers of a holomorphic submersion).  Whenever $F$ is simple, the leaf space $\underline{M}$ inherits a complex structure (uniquely determined by the property that the quotient map $M\to \underline{M}$ is holomorphic) whose corresponding generalized complex structure agrees with the reduction of $\J$
given in Theorem \ref{thm:redGCS}. 
\hfill $\diamond$
\end{ex}
 
\begin{ex}\label{ex:redJact}
Consider the setup of Example~\ref{ex:exactredact}, with reduced Courant algebroid $E_{red}=(K^\perp/K)/G$. In this case, any $G$-invariant GC structure $\J$ satisfying $\J(K^\perp)\subseteq K^\perp$ (or, equivalently, $\J(K)\subseteq K$) over $N$ satisfies the condition in Theorem~\ref{thm:redGCS}, and hence can be reduced to a GC structure $\J_{red}$ on $E_{red}$. This extends \cite[Thm.~4.1]{HKquot} (based on \cite[Thm.~5.2]{bcg}). \hfill $\diamond$
\end{ex}

%%%%%%%%%%%%%%%%%%%%%%%%%%%%%%%%%%%%%%%%%%%%%%%%%%
\subsection{Reduction of lagrangian submanifolds and Dirac structures}

A well-known construction in classical symplectic geometry is the reduction of lagrangian submanifolds through coisotropic reduction, see, e.g.,  \cite[Lect.~3]{wein-lecturesonsymplectic}.
Consider a symplectic manifold $M$ along with
a coisotropic submanifold $N$ and a lagrangian submanifold $L$. Assume that $N$ and $L$ intersect cleanly, i.e., $N\cap L$ is a submanifold and $T(N\cap L)= TN \cap TL$. The latter condition can be equivalently expressed in terms of vanishing ideals as $I_{N\cap L} = I_N + I_L$ (see, e.g., \cite[Lemma 5.1]{LiLi}). Suppose also that the null foliation of $N$ is simple, so that its leaf space $\underline{N}$ is a smooth symplectic manifold. Then the projection of $N\cap L$ to $\underline{N}$ is, at least locally, a lagrangian submanifold in $\underline{N}$ (sometimes called a lagrangian ``sub-immersion'').

We will now discuss an analog of this construction for degree $2$ $\mathbb N$-manifolds. As a consequence,  using the description of Dirac structures in terms of lagrangian submanifolds (explained in the end of $\S$ \ref{subsec:reducible}), we will obtain a reduction procedure for Dirac structures.

Let $\cM$ be a degree 2 symplectic  $\N$-manifold,  $\cN$ a  coisotropic submanifold, and $\cL$ a  
lagrangian submanifold, with bodies $N$ and $S$, respectively. 
The corresponding vanishing ideals are denoted by $\cI_{\cN}$ and $\cI_{\cL}$. We assume, for simplicity, that 
$$
N\subseteq S
$$ 
(in particular, one could take $S=M$, which would be enough to treat reduction of ordinary Dirac structures).
Let $E$ be the pseudo-euclidean vector bundle corresponding to $\cM$, so that $\cL$ is described by a  lagrangian subbundle $L\to S$ (as in Corollary \ref{cor:lag}), and $\cN$ corresponds to geometric coisotropic data $(N,K,F, \nabla)$ (as in Theorem \ref{thm:coisotropic}).

\subsubsection{\underline{Clean intersection}}\label{subsubsec:clean}
 
We start by considering the intersection of $\cN$ and $\cL$ in $\cM$. Following the classical case, we say that $\cN$ and $\cL$ {\em intersect cleanly} if the sheaf of ideals 
$\cI_{\cN} + \cI_{\cL}$ is regular  (in the sense of \S \ref{subsec:Nsub}).

We have the following geometric characterization of the clean-intersection condition.

\begin{prop}\label{lem:clean}
The sheaf of ideals $\cI_{\cN} + \cI_{\cL}$ is regular
 if and only if
 \begin{itemize}
     \item[i)] $K\cap L|_N$ has constant rank, and
\item[ii)] $\nabla_Y(\Gamma_{L_{quot}})\subseteq \Gamma_{L_{quot}}$, for any section $Y$ of $\Gamma_F$,
 \end{itemize}
where
$L_{quot}=\frac{(K^\perp\cap L|_N) + K}{ K}$.
\end{prop}

Note that condition i) ensures that 
$$
L_{quot}\subseteq E_{quot} = K^\perp/K
$$ 
is a smooth lagrangian subbundle.

\begin{proof} 

The sheaf of ideals $\cI_{\cN} + \cI_{\cL}$ is locally generated in
degrees 0, 1 and 2, and it satisfies 
$$
(\cI_{\cN} + \cI_{\cL})_0 =
\fl_N, \;\;\ (\cI_{\cN} + \cI_{\cL})_1 = \Gamma_\st{E,K + L|_N}, \;\; (\cI_{\cN} +
\cI_{\cL})_2= \Gamma_\st{\mathbb{A}_E,\tilde{K}+\tilde{L}|_N},
$$ 
recalling that 
$(\cI_{\cL})_1 = \Gamma_\st{E,L}$,  $(\cI_{\cN})_1=\Gamma_\st{E,K}$,
and the vector bundles
$\tilde{K}\to N$ and $\tilde{L}\to S$ are such that $(\cI_{\cN})_2=\Gamma_\st{\mathbb{A}_E,\tilde{K}}$ and
$(\cI_{\cL})_2=\Gamma_\st{\mathbb{A}_E,\tilde{L}}$. It follows that
$\cI_{\cN} + \cI_{\cL}$ is regular if and only if the following conditions are satisfied: $K + L|_N$ and
$\tilde{K} + \tilde{L}|_N$ have constant rank, and
\begin{equation}\label{eq:regcond}
(\tilde{K}+\tilde{L}|_N)\cap \wedge^2 E|_N = (K + L|_N)\wedge E|_N,\end{equation}
see Lemma \ref{lem:geomideal} (upon using the isomorphism $E\cong E^*$). 
Notice that the condition that $\tilde{K} + \tilde{L}|_N$ has constant rank
is implied by the other two. Indeed the sequence
$$
0\to (\tilde{K}+\tilde{L}|_N)\cap \wedge^2 E|_N \to \tilde{K} + \tilde{L}|_N \to TS|_N\to 0,
$$ 
obtained by restricting \eqref{eq:atiyahseq}, is exact, because  $\tilde{L}$ maps surjectively onto  ${TS}$ under the symbol map (to verify this last claim, recall that $\Gamma_{\tilde{L}}=\Gamma^\st{L}_{\st{\mathbb{A}_{E|_N}}}$ and the surjectivity of \eqref{eq:maptoquot}, shown in Prop. \ref{prop:atiyahonto}). 
If  $K + L|_N$ has constant rank and \eqref{eq:regcond} holds, 
 then the first (nontrivial) term in the previous sequence has constant rank, and the exactness implies that the middle term has constant rank as well.
Hence, to prove the proposition, we may assume that item i) holds and show that, in this case, item ii) is equivalent to  \eqref{eq:regcond}. This will be verified in the following two claims.

\smallskip

\noindent{{\bf Claim:} \emph{Condition \eqref{eq:regcond} holds if and only if the symbol map $\sigma \colon \tilde{K} \cap \tilde{L}|_N\to F$ is onto.}}

\smallskip

To prove the claim, recall (from Prop.~\ref{prop:1-1} and Lemma ~\ref{lem:geomideal})  that 
\begin{equation}\label{eq:tKtL}
    \tilde{K}\cap \wedge^2 E|_N = K\wedge E|_N \;\text{   and   }\;\;
\tilde{L}|_N\cap \wedge^2 E|_N = L|_N\wedge E|_N,
\end{equation}
so we have that
$$
(K+L|_N)\wedge E|_N \subseteq \tilde{K}\cap \wedge^2 E + \tilde{L}|_N\cap
\wedge^2 E \subseteq (\tilde{K} + \tilde{L}|_N)\cap \wedge^2 E.
$$ 
It follows that \eqref{eq:regcond} is equivalent to   the opposite inclusion, namely   
$$
(\tilde{K} + \tilde{L}|_N)\cap
\wedge^2 E\subseteq (K+L|_N)\wedge E|_N.
$$

Suppose that this last condition holds, and let $X \in F$. One can take
$\tilde{k} \in \tilde{K}$ and $\tilde{l}\in \tilde{L}|_N$ with
$\sigma(\tilde{k})=X$ and $\sigma(\tilde{l})=-X$, so that
$\tilde{k}+\tilde{l}\in \wedge^2 E$. Then $\tilde{k}+\tilde{l} =
k_i\wedge e^i + l_i\wedge e^i$, for $k_i\in K$, $l_i\in L$, and
$e^i\in E$. Hence 
$\tilde{k} - k_i\wedge e^i = -\tilde{l} +
l_i\wedge e^i$ lies in $\tilde{K}\cap \tilde{L}|_N$ and has symbol $X$, showing that
the symbol map $\tilde{K}\cap \tilde{L}|_N \to F$ is onto.

Conversely,  suppose now that the symbol map $\tilde{K}\cap \tilde{L}|_N \to F$ is
onto, and let $\tilde{k}+\tilde{l}\in  \wedge^2 E|_N$, with
$\tilde{k}\in \tilde{K}$ and $\tilde{l}\in \tilde{L}|_N$. Then 
$\sigma(\tilde{k})=-\sigma(\tilde{l})=:X$. Let $u \in \tilde{K}\cap
\tilde{L}|_N$ be such that $\sigma(u)=X$. 
Then
$\tilde{k}+\tilde{l}$ is the sum of $\tilde{k}-u
\in \tilde{K}\cap \wedge^2 E|_N$ 
%\subseteq K\wedge E|_N
and
$\tilde{l}+u \in \tilde{L}|_N\cap \wedge^2 E|_N$, %\subseteq L|_N\wedge E|_N
and   \eqref{eq:tKtL} shows implies that 
$\tilde{k}+\tilde{l}\in (K+L|_N)\wedge E|_N$, proving the claim.

\smallskip

\noindent{{\bf Claim:} \emph{The symbol map $\tilde{K} \cap \tilde{L}|_N\to F$ is onto if and only if \begin{equation}\label{nablaLquot}
\nabla_Y(\Gamma_{L_{quot}})\subseteq \Gamma_{L_{quot}}, \;\; \mbox{for any section $Y$ of $\Gamma_F$}.
\end{equation}
}}

\smallskip
 
From  \eqref{eq:I2claim} it follows that $\Gamma_{\widetilde{L}} = \Gamma_{\mathbb{A}_{E|_S}}^{\st{L}}$ and, for each open subset $V\subseteq N$, 
$$
\Gamma_{\widetilde{K}}(V)  = \{  (Y,D)\in
\Gamma_{\mathbb{A}_{E|_N}}^{\st{K}}(V)\,|\, Y\in \Gamma_F(V),\, [D]=\nabla_{Y} \}.
$$
Therefore 
\begin{equation*}\label{eq:gammatktl}
\Gamma_{\tilde{K} \cap \tilde{L}|_N}(V) = \{ (Y,D)\in
\Gamma_{\mathbb{A}_{E|_N}}^{\st{K,L|_N}}(V)\,| \, Y\in \Gamma_F(V),  \,[D]=\nabla_{Y} \}.
\end{equation*}

By Prop.~\ref{prop:Lonto},  the map 
$$
\Gamma_{\mathbb{A}_{E|_N}}^{\st{K,L|_N}}  \rightarrow \Gamma_{\mathbb{A}_{E_{quot}}}^{L_{quot}},\;\; (Y,D)\mapsto (Y,[D]),
$$
is onto, and it is clear that it restricts to a surjective map
$$
\Gamma_{\tilde{K} \cap \tilde{L}|_N} \twoheadrightarrow
\Gamma_{\mathbb{A}_{E_{quot}}}^{L_{quot}}\cap \nabla(\Gamma_F), 
$$
where the sheaf on the right-hand side has sections (over an open $V\subseteq N$) of the form
$(Y, \nabla_{Y}) \in \Gamma_{\mathbb{A}_{E_{quot}}}(V)$, with
$Y\in \Gamma_F(V)$ and $\nabla_{Y}(\Gamma_{L_{quot}})\subseteq \Gamma_{L_{quot}}$. Considering symbol maps we obtain the commutative diagram
$$
\xymatrix{
\Gamma_{\tilde{K} \cap \tilde{L}|_N} \ar[dr] \ar@{->>}[rr]_{}
  && \Gamma_{\mathbb{A}_{E_{quot}}}^{L_{quot}}\cap \nabla(\Gamma_F) \ar[dl] \\
   & \Gamma_F 
}
$$
Now the surjectivity of the horizontal map implies that each symbol map is onto if and only if the other one is, which is the statement in the claim.
\end{proof}

\subsubsection{\underline{Reduction of lagrangian submanifolds}}

Assuming that the coisotropic submanifold $\cN$ and the lagrangian submanifold $\cL$ (whose body contains the body of $\cN$) intersect cleanly, and that $\cN$ is reduced to $\underline{\cN}$, we now describe the reduction of $\cL$ to $\underline{\cN}$.

Denoting by $\iota\colon N \hookrightarrow M$ the inclusion of bodies, we have that $C_{\cN} = \iota^{-1}(C_\cM / \cI_N)$. Hence
$\cI:=\iota^{-1}((\cI_{\cN} + \cI_{\cL}) / \cI_{\cN})$ is a sheaf of ideals in $C_{\cN}$, to be understood as the vanishing ideal of the submanifold $\cN\cap \cL$ in $\cN$. This sheaf is locally generated in
degrees 0, 1 and 2, and we have
$$
\cI_0 = 0,\;\quad \cI_1 = \Gamma_{(K+L|_N) /K}, \; \quad \cI_2 =
\Gamma_{(\tilde{K}+\tilde{L}|_N) /\tilde{K}}.
$$
Recall from $\S$ \ref{subsubsec:clean} that, by the clean-intersection condition, the quotients
$$
\frac{(K+L|_N)}{K} =\frac{L|_N}{K\cap L|_N} \quad \mbox{ and } \quad 
\frac{(\tilde{K}+\tilde{L}|_N)}{\tilde{K}}=\frac{\tilde{L}|_N}{\tilde{K}\cap \tilde{L}|_N}
$$ 
are vector bundles over $N$.

Now consider the subsheaf (of algebras) of the sheaf of basic functions on $\cN$ (see $\S$ \ref{subsec:basic}) given by
$$
\cI_{quot}:= \cI\cap
(C_{\cN})_{bas} \subseteq (C_{\cN})_{bas},
$$
which is locally generated in
degrees 0, 1 and 2, and satisfies
\begin{equation}\label{eq:Ired01}
(\cI_{quot})_0 = 0, \;\; \quad (\cI_{quot})_1 = \Gamma_{(K+L|_N) /K}\cap
\Gamma_{E_{quot}}^\st{flat} = \Gamma^\st{flat}_{L_{quot}},
\end{equation}
and, for each open $V\subseteq N$, $(\cI_{quot})_2(V)$ is given by 
operators $(Y,D) \in \Gamma_{\mathbb{A}_{E|_N}}^{\st{K,L|_N}}(V)$ satisfying
$$
 [Y,\Gamma_F(V)]\subseteq
\Gamma_F(V),\;\;\; [D](\Gamma^\st{flat}_{E_{quot}}(V))
\subseteq \Gamma^\st{flat}_{E_{quot}}(V)
$$ 
modulo    
those such that $Y\in \Gamma_F(V)$, and $[D]=\nabla_{Y}$.
We will give another characterization of $(\cI_{quot})_2$.

Consider the vector bundle $\mathbb{A}^\nabla_{E_{quot}}$ (see \eqref{eq:Anabla}) with its natural flat partial connection.  Let $\Gamma^\st{L_{quot},flat}_{\mathbb{A}^\nabla_{E_{quot}}}$ be the subsheaf of $\Gamma^\st{flat}_{\mathbb{A}^\nabla_{E_{quot}}}$
whose sections $\overline{(Y,\Delta)}$ satisfy the additional property that 
$$
\Delta(\Gamma_{L_{quot}})\subseteq \Gamma_{L_{quot}}.
$$ 
(Note that this last condition is well defined for the class $\overline{(Y,\Delta)}$ by the clean-intersection assumption and  Prop.~\ref{lem:clean} part (ii).)
Then the map \eqref{eq:basic2} giving the identification $((C_\cN)_{bas})_2\cong \Gamma^\st{flat}_{\mathbb{A}^\nabla_{E_{quot}}}$ restricts to an injective map
$$
(\cI_{quot})_2\to \Gamma^\st{L_{quot},flat}_{\mathbb{A}^\nabla_{E_{quot}}}
$$
which is also surjective by \eqref{prop:Lonto}, so 
\begin{equation}\label{eq:Iquot2}
(\cI_{quot})_2\cong \Gamma^\st{L_{quot},flat}_{\mathbb{A}^\nabla_{E_{quot}}}.
\end{equation}

Suppose that $\cN$ is reducible to $\cMred$, with projection $(p,p^\sharp)\colon \cN \to \cMred$, as in $\S$ \ref{subsec:coisored}.
Recall that we have an isomorphism
$p^\sharp \colon C_{\cMred}\to p_* (C_\cN)_{bas}$, and let $\cI_{red}\subseteq C_{\cMred}$ be defined by 
$$
p^\sharp(\cI_{red})=p_*(\cI_{quot}).
$$

Consider the pseudo-euclidean vector bundle $E_{red}\to \Mred$ corresponding to $\cMred$, given by the quotient of $E_{quot}=K^\perp/K$ with respect to $F$ and $\nabla$ (see Theorem~\ref{thm:coisored}). 
Due to the geometric description of the clean-intersection condition in Prop.~\ref{lem:clean}, one may also consider the quotient of the lagrangian vector subbundle $L_{quot}=\frac{(K^\perp\cap L|_N) + K}{K}$ of $E_{quot}$ with respect $F$ and $\nabla$, which is a lagrangian subbundle 
$$
L_{red}\subseteq E_{red}.
$$

\begin{thm}\label{thm:lagred} Suppose that a coisotropic submanifold $\cN$ and a lagrangian submanifold $\cL$ (with body containing the one of $\cN$) intersect cleanly, and that $\cN$ reduces to $\cMred$. Then: 
\begin{itemize}
\item[(a)] $\cI_{red}$ is the vanishing ideal of a lagrangian submanifold $\underline{\cL}$ of $\cMred$, corresponding to the
lagrangian subbundle $L_{red}$ (in the sense of Cor. \ref{cor:lag}).  

\item[(b)] If $\Theta$ is a Courant function on $\cM$ that is reducible for $\cN$ and $\cL$, then $\Theta_{red}$ satisfies $\{\Theta_{red}, \cI_{red}\} \subseteq \cI_{red}$ (i.e., it is reducible for $\underline{\cL}$).\end{itemize}
\end{thm}

\begin{proof}
We know that $L_{red}\to \Mred$ is a lagrangian subbundle of $E_{red}$, so it gives rise to a lagrangian submanifold $\underline{\cL}$ of $\cMred$ by Cor. \ref{cor:lag}. To prove (a), we must check that its vanishing ideal equals $\cI_{red}$, and it is enough to verify this fact in degrees 0, 1 and 2. From the expressions for $(\cI_{quot})_0$ and $(\cI_{quot})_1$ in \eqref{eq:Ired01}, it is clear that 
$$
(\cI_{red})_0 = 0 = (\cI_{\underline{\cL}})_0, \qquad (\cI_{red})_1 = \Gamma_{L_{red}} = (\cI_{\underline{\cL}})_1.
$$
Verifying the remaining case in degree 2 amounts to checking that
$(\cI_{red})_2$ coincides with $\Gamma^{\st{L_{red}}}_{\mathbb{A}_{E_{red}}}$, i.e., the sheaf of sections of $\mathbb{A}_{E_{red}}$ that preserve $\Gamma_{L_{red}}$.

Recall the map \eqref{eq:map3} that identifies 
$
\Gamma_{\mathbb{A}_{E_{red}}}$ with 
$$
p_*((C_\cN)_{bas})_2 = p_* \Gamma^\st{flat}_{\mathbb{A}^\nabla_{E_{quot}}}.
$$ 
A section $(\underline{Y}, \underline{\Delta})$ of $\Gamma_{\mathbb{A}_{E_{red}}}$ corresponds to $\overline{(Y,\Delta)}$ in $p_*\Gamma^\st{flat}_{\mathbb{A}^\nabla_{E_{quot}}}$ if and only if
$$
\overline{Y}= p_*(Y), \qquad p_1^\sharp \circ \underline{\Delta}\circ (p_1^\sharp)^{-1} = \Delta|_{\Gamma^\st{flat}_{E_{quot}}},
$$
with $p_1^\sharp$ as in \eqref{eq:map2}  (see Prop.~\ref{prop:Ared}).
Recall that $(\cI_{red})_2 \subseteq \Gamma_{\mathbb{A}_{E_{red}}}$ is defined by the condition that it agrees with  $p_*(\cI_{quot})_2= p_* \Gamma^\st{L_{quot},flat}_{\mathbb{A}^\nabla_{E_{quot}}}$ under this identification. We will check that $\Gamma^{\st{L_{red}}}_{\mathbb{A}_{E_{red}}}$ has this property and hence coincides with $(\cI_{red})_2 $.

Note that a section $(\underline{Y}, \underline{\Delta})$ of $\Gamma_{\mathbb{A}_{E_{red}}}$ satisfying $\underline{\Delta}(\Gamma_{L_{red}})\subseteq \Gamma_{L_{red}}$ corres\-ponds to a section $\overline{(Y,\Delta)}$ such that  ${\Delta}(\Gamma^\st{flat}_{L_{quot}})\subseteq \Gamma^\st{flat}_{L_{quot}}$. The fact that flat sections locally generate $\Gamma_{L_{quot}}$
%Now the fact that $\Gamma_{L_{quot}}$ is invariant by $\nabla$ (by the clean-intersection assumption, see Prop.~\ref{lem:clean}) 
implies that 
$$
{\Delta}(\Gamma^\st{flat}_{L_{quot}})\subseteq \Gamma^\st{flat}_{L_{quot}} \; \iff \; {\Delta}(\Gamma_{L_{quot}})\subseteq \Gamma_{L_{quot}}.
$$
Hence  $\overline{(Y,\Delta)}$ is a section of $\Gamma^\st{L_{quot},flat}_{\mathbb{A}^\nabla_{E_{quot}}}$, proving (a).

To prove (b), recall that $\Theta$ being reducible for $\cN$
says that it is a section of $\mathfrak{N}_{\cI_{\cN}}$, and hence defines a section $\Theta_\cN$ of $(C_\cN)_{bas} = \iota^{-1}(\mathfrak{N}_{\cI_{\cN}}/\cI_{\cN})$. Assuming that $\Theta$ is also reducible for $\cL$, i.e., 
$
\{ \Theta, \cI_{\cL}  \} \subseteq \cI_{\cL},
$
directly implies that 
$$
\{ \Theta_{\cN}, \cI_{quot}  \}\subseteq \cI_{quot},
$$
which proves (b) upon the identification $p^\sharp: C_{\underline{\cN}}\stackrel{\sim}{\to} p_*(C_{\cN})_{bas}$.
\end{proof}

\subsubsection{\underline{Reduction of Dirac structures}}

We now use Theorem \ref{thm:lagred} to obtain a reduction procedure for Dirac structures, phrased in classical geometric terms. Let $E$ be a Courant algebroid over $M$, with anchor $\rho\colon E \to TM$, pseudo-euclidean structure $\SP{\cdot, \cdot}$ and bracket $\Cour{\cdot,\cdot}$.

Let us consider the setup for Courant algebroid reduction of $\S$ \ref{subsec:cred}: geometric coisotropic data $(N, K, F, \nabla)$ with respect to which $E$ is reducible, as in Definition~\ref{def:reduc}.
We also assume that $F$ is simple and $\nabla$ has trivial holonomy, so that we have a reduced Courant algebroid $E_{red}\to \underline{N}$, as in Thm.~\ref{cor:courred}.  Let $L$ be a Dirac structure in $E$ with support on a submanifold $S$ containing $N$ (see Def.~\ref{def:diracsupp}). 

\begin{thm}\label{thm:diracred}
In the setup above, suppose that
\begin{itemize}
\item[(a)] $L|_N\cap K$ has constant rank and 
\item[(b)] the lagrangian subbundle $L_{quot}= \frac{(K^\perp\cap L|_N) + K}{K}$ in $E_{quot}=K^\perp/K$ is $\nabla$-invariant, i.e., $\nabla_Y(\Gamma_{L_{quot}})\subseteq \Gamma_{L_{quot}}$ for any section $Y$ of $F$. 
\end{itemize}
Then the  quotient of $L_{quot}$ with respect to $F$ and $\nabla$ is a Dirac structure  $L_{red}\subset E_{red}$.
\end{thm}

The proof is just a translation of Thm.~\ref{thm:lagred}. The Courant algebroid $E$ corresponds to a  degree 2 symplectic $\mathbb{N}$-manifold $\cM$ equipped with a Courant function $\Theta$, the coisotropic data $(N, K, F, \nabla)$ gives  a  coisotropic submanifold $\cN$ with respect to which $\Theta$ is reducible, and the Dirac structure $L$ corresponds to a lagrangian submanifold $\cL$ (with body $S\supseteq N$) with respect to which $\Theta$ is reducible. The assumptions in the theorem amount to the clean intersection condition, and the fact that $L_{red}$ is a Dirac  structure is equivalent to the property in part (b) of Thm.~\ref{thm:lagred} (by Cor.~\ref{prop:lagdirac}).

\begin{ex}[Reduction of Dirac structures in the standard Courant algebroid]

Consider the standard Courant algebroid $E= TM \oplus T^*M$, let $\iota: N\hookrightarrow M$ be a submanifold and $F\subseteq TN$ be an involutive subbundle defining a simple foliation with quotient map 
$$
p: N \to \underline{N}.
$$
We saw in Example \ref{red:stdca} that, with respect to the geometric coisotropic data given by $N$, $F$, the isotropic subbundle  $K= F\oplus \mathrm{Ann}(TN)$ and $\nabla = \nabla^{Bott} \oplus (\nabla^{Bott})^*$, the Courant algebroid $E$  reduces to $ T\underline{N}\oplus T^*\underline{N}$. For a Dirac structure $L \subset TM\oplus T^*M$ satisfying conditions (a) and (b) of Theorem \ref{thm:diracred}, one can check that the reduced Dirac structure 
$$
L_{red} \subset T\underline{N}\oplus T^*\underline{N} 
$$
coincides with the Dirac structure on $\underline{N}$ resulting from the pullback of $L$ to $N$ followed by its pushforward to $\underline{N}$ (see \cite[\S 5.1]{BuDir} for the definitions); conditions (a) and (b) in Theorem~\ref{thm:diracred} ensure that the result of the composition of these operations is a smooth Dirac structure on $\underline N$.

 The following are two special cases.
\begin{itemize}
\item {\em (Reduction of presymplectic structures)}. Let $\omega$ be a closed 2-form on $M$ and $L=\mathrm{graph}(\omega)$, the graph of the map $TM\to T^*M$ induced by $\omega$ via contraction. Suppose that the subbundle $F\subseteq TN$ satisfies 
\begin{equation}\label{eq:Fker}
F\subseteq \ker(\iota^*\omega).
\end{equation}
This condition turns out to guarantee that (a) and (b) in Theorem~\ref{thm:diracred} hold. Indeed, 
$L|_N\cap K \cong F,$ and hence (a) is satisfied. To verify (b), note that $\iota^*\omega$ induces a map $TN/F\to (TN/F)^*$ defining an element $\omega_{quot}\in \Gamma(\wedge^2 (TN/F)^*)$, in such a way that $L_{quot}=\mathrm{graph}(\omega_{quot})$. Then condition (b), which means that $\nabla^{Bott} \omega_{quot} =0$, amounts to the fact that $\pounds_Y \iota^*\omega =0$ for $Y\in \Gamma(F)$, which holds by \eqref{eq:Fker} and by the closedness of $\omega$.

To describe the reduced Dirac structure on $\underline{N}$, note that $\iota^*\omega$ is a basic form, and 
$$
L_{red}=\mathrm{graph}(\omega_{red}),
$$ 
where $\omega_{red}$ is the unique 2-form on $\underline{N}$ such that $p^*\omega_{red}=\iota^*\omega$. In particular, when $\omega$ is symplectic and $N$ is such that  $TN\cap TN^\omega$ has constant rank, by taking $F=TN\cap TN^\omega=\ker(\iota^*\omega)$ one recovers usual symplectic reduction. 

\smallskip

\item {\em (Reduction of Poisson structures).} Let $\pi \in \Gamma(\wedge^2 TM)$ be a Poisson structure and $L=\mathrm{graph}(\pi)$ be the graph of the map $\pi^\sharp: T^*M\to TM$ given by contraction. With the notation
$$
TN^\pi:= \pi^\sharp(\mathrm{Ann}(TN)) = \mathrm{Ann}((\pi^\sharp)^{-1}(TN)),
$$
assume that $N$ is such that
\begin{itemize}
\item[(i)] $TN\cap TN^\pi$ has constant rank,
\item[(ii)] $TN+TN^\pi$ has constant rank.
\end{itemize}
(These conditions imply that (iii) $TN^\pi$ has constant rank, and 
any two of the conditions (i), (ii) and (iii) imply the third.) Submanifolds $N$ satisfying (ii) above are called  {\em pre-Poisson} submanifolds \cite{CaZa:coiso}, and generalize the notion of coisotropic submanifolds. 

With these assumptions on $N$ and setting 
$$
F= TN\cap TN^\pi,
$$
we will see that (a) and (b) in  Theorem~\ref{thm:diracred} are satisfied, so one can carry out reduction of $L$.

Notice that $L|_N\cap K \cong \mathrm{Ann}(TN)\cap (\pi^{\sharp})^{-1}(TN)=\mathrm{Ann}(TN+TN^\pi)$, so the fact that it has constant rank is equivalent to (ii), showing that (a) holds.
To verify (b), first note that $\pi^\sharp$ restricts to a map $\mathrm{Ann}(TN^\pi)= (\pi^\sharp)^{-1}(TN)\to TN$ and that the natural projection $T^*M|_N\to T^*N$ restricts to a projection $\mathrm{Ann}(TN^\pi)\twoheadrightarrow (TN/F)^*$, $\alpha\mapsto \overline{\alpha}$ (surjectivity can be verified by noticing that the dual map $TN/F\to TM|_N/TN^\pi$ is injective by the very definition of $F$).
We then obtain a well-defined vector-bundle map 
$$
(TN/F)^*\to (TN/F), \qquad \overline{\alpha}\mapsto \overline{(\pi^\sharp(\alpha))},
$$
where $\alpha\in \mathrm{Ann}(TN^\pi)$ and $X\mapsto \overline{X}$ is the projection of $TN$ on $TN/F$, defining an element $\pi_{quot}\in \Gamma(\wedge^2 (TN/F))$ such that $L_{quot} =  \mathrm{graph}(\pi_{quot})$. Condition (b) in  Theorem~\ref{thm:diracred}, which says that  $\nabla^{Bott} \pi_{quot}=0$, amounts to verifying that, for  any 
$Z\in \Gamma(F)$, there is a vector field
$Y$ on $M$ with $Y|_N = Z$ such that  
\begin{equation}\label{eq:(b)}
(\pounds_Y \pi)^\sharp (\alpha) \in F
\end{equation}
for all $\alpha \in \mathrm{Ann}(TN^\pi)$. To prove that this last property holds,
we first note that there exists a local frame for $F$ consisting of the restriction  to $N$ of hamiltonian vector fields on $M$. Indeed, we have a surjective map 
$$
\pi^\sharp: \mathrm{Ann}(TN)\cap \mathrm{Ann}(TN^\pi)=\mathrm{Ann}(TN+TN^\pi) \to F,
$$
and the fact that $TN+TN^\pi$ and $F$ have constant rank ensure that $\mathrm{Ann}(TN+TN^\pi)$ has a local frame given by restrictions of differentials of functions on $M$
(cf. \cite[Lem.~1]{CaFa}). Now supposing that $F$ is spanned by $X_{f_i}|_N$ and writing $Y=a^i X_{f_i}$, we have $\pounds_Y\pi = -X_{a^i}\wedge X_{f_i}$. For $\alpha \in \mathrm{Ann}(TN^\pi)$, it follows that $(\pounds_Y \pi)^\sharp (\alpha)=-\alpha(X_{a^i}) X_{f_i} \in F$, showing that \eqref{eq:(b)} holds.

Since $\pi_{quot}$ is parallel, it gives rise to a Poisson structure $\pi_{red}$ on $\underline{N}$ in such a way that $L_{red}=\mathrm{graph}(\pi_{red})$. This construction matches the reduction of Poisson structures in \cite[Thm.~3]{CaFa}; in particular, when $N$ is coisotropic, $L_{red}$ corresponds to the Poisson structure obtained by coisotropic reduction
(cf. \cite[Ex. 4.2]{Za}). 
\end{itemize}
\hfill $\diamond$
\end{ex}

Building on Example \ref{ex:JSGK}, we see that Thm. \ref{thm:diracred} recovers the following reduction construction  from  \cite{Za}.

\begin{ex} [Reduction of  Dirac structures with support]\label{ex:JSGL}
 Let $E$ be an exact 
Cou\-rant algebroid over $M$, and let $L$ be a Dirac structure supported on a submanifold $N \subseteq M$ such that $\rho(L) = TN$. Suppose that $N$ is equipped with an involutive distribution $F$ that is simple, with leaf space $\underline{N}$. With this setup, we now recall how to canonically obtain an exact Courant algebroid $E_{red}$ over $\underline{N}$ together with a Dirac structure $L_{red}$ therein.

We set $K:=L\cap \rho^{-1}(F) \to N$, which is an isotropic subbundle of $E$ with $\rho(K)=F$, and  moreover satisfies the condition 
$$
\rho(K^{\perp})=TN,
$$
since $K^{\perp}=L+\rho^*(\mathrm{Ann}(F))$.
Then, as explained in Example \ref{ex:JSGK}, one canonically obtains coisotropic data $(N, K, F, \nabla)$ with respect to which $E$ is reducible.
Moreover, it is proven in \cite[Lemma 5.4, Prop. 5.5]{Za} that the $F$-connection $\nabla$ obtained in this case automatically has trivial holonomy. (This relies on the fact that any splitting of $\rho|_L \colon L\to TN$, applied to projectable vector fields on $N$, yields flat sections of $E_{quot}$.) Hence by Thm.~\ref{cor:courred}  we obtain a reduced Courant algebroid $E_{red}\to \underline{N}$, which is exact by Prop.~\ref{prop:exactCA}.

Using that $K\subseteq L \subseteq K^\perp$ and the involutivity of $K$ and $L$, 
one can directly verify  that the two conditions in Thm. \ref{thm:diracred}
are satisfied.
Therefore $L$ reduces to a Dirac structure $L_{red}$ in $E_{red} \to \underline{N}$.
\hfill $\diamond$
\end{ex}

By combining the above example with  Thm. \ref{thm:redGCS} one recovers the ``reduction of branes'' in generalized complex geometry described in \cite[Thm. 7.4]{Za} (it is shown in  part b) of the proof of \cite[Thm. 7.4]{Za} that the hypothesis of Thm. \ref{thm:redGCS} is satisfied).

%%%%%%%%%%%%%%%%%%%%%%%%%%%%%%%%%%%%%%%%%%

\section{Momentum maps and hamiltonian reduction}\label{sec:momap}

In this section we discuss symplectic reduction of hamiltonian actions in the context of  degree 2 symplectic $\mathbb{N}$-manifolds.
We recall the main ingredients of the classical procedure \cite{MWpaper}.

Let $M$ be a symplectic manifold, with Poisson bracket $\{\cdot,\cdot\}$, and let $\g$ be a Lie algebra. A $\g$-action on $M$,
$$
\g \to \mathfrak{X}(M), \quad u\mapsto u_M,
$$
is called {\em hamiltonian} if there is a smooth map $\mu\colon M \to \g^*$ satisfying the following properties: its dual map 
$$
\mu^*\colon \g \to C^\infty(M), \quad (\mu^*u) (x) = \SP{\mu(x),u},
$$
is a Lie algebra homomorphism such that
$$
u_M =  X_{\mu^*u} = \{\mu^*u, \cdot\}.
$$
The map $\mu\colon M \to \g^*$ is called a {\em momentum map}. 
%Letting $G$ be any {\em connected} Lie group with Lie algebra $\g$, 
The bracket-preserving property of $\mu^*$ is equivalent to the $\g$-equivariance of $\mu$ with respect to the co-adjoint action of $\g$ on $\g^*$.

We will restrict ourselves to the simplest formulation of symplectic reduction, by assuming that $0$ is a regular value of $\mu$, so that $N:=\mu^{-1}(0)$ is a $\g$-invariant submanifold of $M$; letting $G$ be any connected Lie group with Lie algebra $\g$, we also assume that the $\g$-action on $N$ integrates to a $G$-action that is free and proper.  Then $M_{red}:= N/G$ carries a natural symplectic structure. In this setting,  $N$ is a coisotropic submanifold of $M$, and $M_{red}$ agrees with its coisotropic reduction. In particular, $C^\infty(N)^G = C^\infty(N)_{bas}$ is a Poisson algebra, and the reduced symplectic form on $M_{red}$ is characterized by the fact that the identification $q^*\colon C^\infty(M_{red})\stackrel{\sim}{\to} C^\infty(N)^G$ induced by the quotient map $q\colon N \to M_{red}$ is an isomorphism of Poisson algebras. 

We will present an analogue of this construction for $\mathbb{N}$-manifolds of  degree 2 .

\subsection{Differential graded Lie algebras of degree 2}

%\begin{defi}
A \emph{graded Lie algebra} is a (real) graded vector space $\tilde{\mathfrak{g}} = \oplus_{k\in \mathbb{Z}} \mathfrak{g}_k$, equipped with a degree-preserving graded Lie bracket $[\cdot,\cdot]  \colon \tilde{\mathfrak{g}} \otimes \tilde{\mathfrak{g}} \to \tilde{\mathfrak{g}}$.
This means that, for homogeneous elements $x \in \g_k$ and $y \in \g_l$,  
$$
[x,y]=-(-1)^{kl}[y,x]
$$ 
and $[x,\cdot]$ is a degree $k$ derivation of the bracket (graded Jacobi identity).
%\end{defi}

For a $\N$-manifold $\cM$ of degree $2$, we consider the space of global sections of the sheaf of vector fields $\mathcal{T} \cM$, denoted by $\vect (\cM)$, equipped with its graded Lie bracket. An \emph{action} of a graded Lie algebra $\tilde{\mathfrak{g}}$ on $\cM$ is a graded Lie algebra morphism $ \tilde{\mathfrak{g}} \to \vect (\cM)$.

A \emph{differential graded Lie algebra (DGLA)} is a graded Lie algebra equipped with a coboundary operator $\delta  \colon \tilde{\mathfrak{g}}\to \tilde{\mathfrak{g}}$   that is a degree 1 derivation of the Lie bracket, 
$$
\delta [x,y] = [\delta x, y] + (-1)^k [x, \delta y],
$$
for $x$ of degree $k$.
Such a $\delta$ is referred to as a {\em differential}.

For the purpose of studying hamiltonian actions on symplectic $\N$-manifolds of degree $2$ (see $\S$ \ref{sec:hamred}), we will focus on (differential) graded Lie algebras that are concentrated in degrees $0$, $-1$, and $-2$; we will say that such a (differential) graded Lie algebra is \emph{of degree $2$}. 

For ordinary vector spaces $\mathfrak{h}$, $\mathfrak{a}$, and $\mathfrak{g}$, consider the graded vector space  
\begin{equation}\label{eq:hag}
\tilde{\mathfrak{g}} = \mathfrak{h}[2] \oplus \mathfrak{a}[1] \oplus \mathfrak{g}.
\end{equation}
Recall that the notation means that elements in $\h$ have degree $-2$, and elements in $\mathfrak{a}$ have degree $-1$.

\begin{prop}\label{prop:gla}
	A graded Lie algebra structure on $\tilde{\mathfrak{g}}$ is equivalent to  the following data:
\begin{itemize}
	\item a Lie algebra structure on $\mathfrak{g}$,
	\item a representation $\tau \colon \mathfrak{g} \to \End (\mathfrak{a})$, 
	\item a representation $\lambda \colon \mathfrak{g} \to \End (\mathfrak{h})$, and
	\item a symmetric bilinear map $\varpi \colon \mathfrak{a} \otimes \mathfrak{a} \to \mathfrak{h}$ that is $\mathfrak{g}$-equivariant, i.e.,  $u \cdot\varpi(a_1,a_2)=\varpi(u \cdot a_1,a_2)+\varpi(a_1,u\cdot a_2)$ for all $v\in \g$, $a_1,a_2\in \mathfrak{a}$.
\end{itemize}
\end{prop}
\begin{proof}
     The correspondence is given by 
     $$
     \tau(u)(a) = [u,a], \;\; \lambda(u)(h) = [u,h], \;\; \varpi(a_1,a_2) = [a_1,a_2],
     $$ 
     for $u \in \mathfrak{g}$, $a,a_1,a_2 \in \mathfrak{a}$, and $h \in \mathfrak{h}$. The fact that $\tau$ and $\lambda$ are representations, as well as the $\mathfrak{g}$-equivariance property of $\varpi$, follow from the graded Jacobi identity for the bracket on $\tilde{\mathfrak{g}}$.
\end{proof}

Now suppose that $\tilde{\mathfrak{g}}$ as in \eqref{eq:hag} is equipped with a graded Lie algebra structure as well as a coboundary operator $\delta$, so that we have a $3$-term chain complex 
\[
\mathfrak{h} \tolabel{\delta} \mathfrak{a} \tolabel{\delta} \mathfrak{g}. 
\]
\begin{prop}\label{prop:dgla}
	$(\tilde{\mathfrak{g}}, [\cdot,\cdot], \delta)$ is a DGLA if and only if the following conditions hold:
\begin{itemize}
	\item[(a)] $\delta$ is $\mathfrak{g}$-equivariant with respect to $\lambda$, $\tau$, and the adjoint action of $\mathfrak{g}$ on itself.
	\item[(b)] $\delta \varpi (a_1,a_2) = \tau(\delta a_1)(a_2) + \tau(\delta a_2)(a_1)$ for all $a_1,a_2\in \mathfrak{a}$.
	\item[(c)] $\lambda(\delta a)(h) = \varpi(a,\delta h)$ for all $a\in \mathfrak{a}, h\in \mathfrak{h}$.
\end{itemize}
\end{prop}

The proof is a direct verification.

%%%%%%%%%%%%%%%%%%%%%%%%%%%%

\subsection{Hamiltonian actions and reduction in degree 2}\label{sec:hamred}

We start with some general considerations.
Given a graded vector space
$$
\tilde{\mathfrak{g}} = \mathfrak{h}[2] \oplus \mathfrak{a}[1] \oplus \mathfrak{g},
$$ 
we will regard the graded vector space $\tilde{\g}^*[2]=\mathfrak{g}^*[2] \oplus \mathfrak{a}^*[1] \oplus \mathfrak{h}^*$
as a split $\mathbb{N}$-manifold of  degree 2  determined  by the vector bundles $E_1= \mathfrak{a}^* \times \mathfrak{h}^* \to \mathfrak{h}^* $ and $E_2= \mathfrak{g}^* \times \mathfrak{h}^* \to \mathfrak{h}^*$,  as in Example~\ref{ex:gradedVS}. In particular, given a degree 2 $\mathbb{N}$-manifold $\cM$, any morphism $\tilde{\mu}=(\mu, \mu^\sharp)\colon \cM \to \tilde{\g}^*[2]$ is determined by three maps, 
\begin{equation}\label{eq:mucomponents}
\mu\colon M\to \mathfrak{h}^*, \qquad \varrho\colon \mathfrak{a}\to C(\cM)_1, \qquad \varphi\colon \mathfrak{g}\to C(\cM)_2,
\end{equation}
where $\mu$ is the map between bodies, $\varrho$ and $\varphi$ are defined by the components of $\mu^\sharp$ in degrees 1 and 2, and we denote by $C(\cM)$ the graded algebra of global sections of the sheaf $C_\cM$. 
Let $\mu^*\colon \mathfrak{h}\to C^\infty(M)$ be given by 
$$\mu^*(h) (x) = \SP{\mu(x),h}.
$$
With a slight abuse of notation, we will denote by 
\begin{equation}\label{eq:musharp}
\tilde{\mu}^\sharp\colon \tilde{\g} \to C(\cM)[2]
\end{equation}
the morphism of graded vector spaces defined by $\mu^*$,  $\varrho$ and $\varphi$ in \eqref{eq:mucomponents}.

Suppose now that $\tilde{\mathfrak{g}}$ is a graded Lie algebra of degree 2 and 
$\cM$ is a  degree 2  symplectic $\mathbb{N}$-manifold, so that $C(\cM)[2]$ is a graded Lie algebra with respect to the Poisson bracket.
A $\tilde{\g}$-action on $\cM$, 
$$
\tilde{\g}\to \mathfrak{X}(\cM), \quad \xi \mapsto \xi_\cM,
$$ 
is called {\em hamiltonian} if there is a morphism of  $\mathbb{N}$-manifolds  of degree 2,
$$
\tilde{\mu}\colon \cM \to \tilde{\g}^*[2],
$$
so that the  induced map $\tilde{\mu}^\sharp\colon \tilde{\g} \to C(\cM)[2]$ is  a morphism of graded Lie algebras and determines the action via
\begin{equation}\label{eq:mmapcond}
\xi_\cM = X_{\tilde{\mu}^\sharp \xi} = \{\tilde{\mu}^\sharp \xi, \cdot\}, 
\end{equation}
for all $\xi \in \tilde{\g}$. As in the classical case we refer to $\tilde{\mu}$ as a {\em momentum map}.

Similarly to the classical case, one can perform reduction with respect to a hamiltonian action with suitable regularity assumptions. Suppose that $0$ is a regular value  for the momentum map (in the sense of $\S$ \ref{subsec:tangent}), so that 
$$
\cN= \tilde{\mu}^{-1}(0)
$$ 
is a submanifold of $\cM$ with body $N=\mu^{-1}(0)$ and sheaf of vanishing ideals $\cI$ generated by the image of the maps ${\mu}^*$, $\varrho$ and $\varphi$ in \eqref{eq:mucomponents}. 

\begin{remark}\label{rem:012}
Since the sheaf of vanishing ideals of a submanifold is locally generated in degrees 0, 1 and 2, our assumption that $0$ is a regular value for $\tilde{\mu}$ justifies why $\tilde{\g}$ was taken to be of degree 2.  \hfill $\diamond$
\end{remark}

A direct consequence of $\tilde{\mu}^\sharp$ being bracket preserving is that $\cN$ is coisotropic (see \eqref{eq:coiso}), and condition \eqref{eq:mmapcond} implies that
the vector fields $\xi_{\cM}$, for $\xi \in \tilde{\g}$, preserve $\cI$. 
Hence each $\xi_\cM$ induces a vector field $\xi_{\cN}$ on $\cN$
(see \eqref{eq:rest1}), in such a way that the map $\xi \mapsto \xi_{\cN}$ defines a $\tilde{\g}$-action on $\cN$. Moreover,
the momentum map condition \eqref{eq:mmapcond} implies that the vector fields $\xi_{\cN}$, for $\xi \in \tilde{\g}$, span the null distribution of $\cN$ (see Lemma~\ref{lem:tauIwlocal}). Hence the sheaf of invariant functions on $\cN$, 
defined on each open subset $V\subseteq N$ by
$$
C_\cN^{\tilde{\g}}(V) = \{ f\in C_\cN(V)\;|\; \xi_\cN (f) =0 \;\, \forall\, \xi \in \tilde{\g} \}, 
$$
agrees with the sheaf of basic functions, $C_\cN^{\tilde{\g}} = (C_\cN)_{bas}$, and therefore is a sheaf of Poisson algebras. 

Assuming that the regularity conditions for coisotropic reduction are satisfied (as in Theorem~\ref{thm:coisored}), one obtains a degree 2 symplectic $\N$-manifold $\cM_{red}$, along with a surjective submersion $\cN \to \cM_{red}$ that identifies $C_{\cM_{red}}$ with $C_\cN^{\tilde{\g}}$ as sheaves of Poisson algebras (and this identification uniquely characterizes the symplectic structure on $\cM_{red}$). As in the classical case, one refers to $\cM_{red}$ as the {\em symplectic reduction} of $\cM$ with respect to the momentum map $\tilde{\mu}$.

Suppose now, additionally,  that $\cM$ is equipped with a Courant function $\Theta$. Then $C(\cM)[2]$ is not only a graded Lie algebra but a DGLA, with differential given by $\{\Theta,\cdot\}$.
A natural way to ensure that $\Theta$ is reducible with respect to $\cN= \tilde{\mu}^{-1}(0)$ is assuming that
$\tilde{\g}$ is also a DGLA, with differential $\delta$, and that 
$\tilde{\mu}^\sharp\colon \tilde{\g}\to C(\cM)[2]$  is a morphism of DGLAs, since in that case 
the momentum map satisfies 
\begin{equation}\label{eq:dglacond}
\{\Theta, \tilde{\mu}^\sharp(\xi)\} = \tilde{\mu}^\sharp(\delta \xi), \;\;\;\; \forall \xi \in \tilde{\g}.
\end{equation}
Since this condition implies that $\Theta$ is reducible, it follows, as explained in $\S$~\ref{subsec:reducible}, that $\Theta$ gives rise to a Courant function $\Theta_{red}$ on $\cM_{red}$.

We now give the classical geometric descriptions of the momentum map $\widetilde{\mu}$, of the coisotropic submanifold $\cN = \widetilde{\mu}^{-1}(0)$, and of the symplectic reduced manifold $\cM_{red}$ and Courant function $\Theta_{red}$.

\subsection{Degree 2 hamiltonian actions in classical terms}\label{subsec:deg2ham}

Let $\tilde{\g}=\mathfrak{h}[2]\oplus \mathfrak{a}[1] \oplus \g$ be a graded Lie algebra of degree 2, as in \eqref{eq:hag}, and let $\cM$ be a degree 2 symplectic  $\N$-manifold equipped with a hamiltonian $\tilde{\g}$-action with momentum map $\tilde{\mu}\colon \cM \to \tilde{\g}^*[2]$. 

Suppose that $\cM$ corresponds to the pseudo-euclidean vector bundle $(E, \SP{\cdot,\cdot})$. Then by \eqref{eq:mucomponents} (and \eqref{eq:Mdeg12}) the momentum map $\tilde{\mu}$ is described by maps
\begin{equation}\label{eq:mucomp2}
\mu\colon M\to \mathfrak{h}^*, \qquad \varrho\colon \mathfrak{a}\to \Gamma(E), \qquad \varphi\colon \mathfrak{g}\to \Gamma(\mathbb{A}_E).
\end{equation}
Recall the characterization of the graded Lie algebra structure on $\tilde{\g}$ in Prop.~\ref{prop:gla} as a Lie algebra $\g$, together with representations $\tau$ and $\lambda$ on $\mathfrak{a}$ and $\mathfrak{h}$, respectively, and an equivariant symmetric map $\varpi\colon \mathfrak{a}\otimes \mathfrak{a}\to \mathfrak{h}$.

One can directly verify the next result.

\begin{lemma}\label{lem:muequiv}
The property that $\tilde{\mu}^\sharp$ in \eqref{eq:musharp} preserves graded Lie brackets is equivalent to
\begin{itemize}
\item[(a)] $[\varphi(u),\varphi(v)]=\varphi([u,v])$,
\item[(b)] $\varrho(\tau(u)a)=\varphi(u)(\varrho(a))$,
\item[(c)] $\mu^*(\lambda(u)h)=\pounds_{u_M}(\mu^*h)$,
\item[(d)] $\mu^*(\varpi(a_1,a_2))=\langle \varrho (a_1), \varrho (a_2) \rangle$,
\end{itemize}
for all $u,v\in \mathfrak{g}$, $a, a_1, a_2 \in \mathfrak{a}$, $h\in \mathfrak{h}$, and where
$u_M\in \vect(M)$ is the symbol of the differential operator $\varphi(u)$.
\end{lemma}

By (a) in the previous lemma, $\varphi$ defines an action of $\g$ on $E\to M$ by derivations preserving the pseudo-euclidean metric, and the composition of $\varphi$ with the symbol map,
$$
\g \stackrel{\varphi}{\to} \Gamma(\mathbb{A}_E) \stackrel{\sigma}{\to} \mathfrak{X}(M), \quad u\mapsto u_M,
$$
defines a $\g$-action on $M$. Conditions (b) and (c) say that the maps $\mu\colon M\to \h^*$ and $\varrho\colon \mathfrak{a}\to \Gamma(E)$ in \eqref{eq:mucomp2} are $\g$-equivariant, where $\h^*$ is equipped with the linear $\g$-action given by the dual of the representation of $\g$ on $\h$. Condition (d), which can be written as
\begin{equation}\label{eq:betaeq}
\SP{\varrho(a_1),\varrho(a_2)}(x) = \SP{\mu(x),\varpi(a_1,a_2)}, \qquad x\in M,
\end{equation}
shows how  $\varrho$ is related to the pseudo-euclidean metric on $E$.

The first assumption for reduction is that $0$ is a regular value for the momentum map $\tilde{\mu}$. 
%We now describe this condition in terms of its components \eqref{eq:mucomp2}.
Denote by $\mathfrak{a}_M$ and $\g_M$ the trivial bundles $\mathfrak{a}\times M \to M$ and $\g\times M \to M$, respectively, and keep the same notation $\varrho$ and $\varphi$ for the maps
$$
\mathfrak{a}_M \to E, \, (a,x)\mapsto \varrho(a)|_x, \qquad 
\g_M \to \mathbb{A}_E, \, (u,x) \mapsto \varphi(u)|_x.
$$

We have the following direct consequence of Lemma~\ref{lem:regular}.

\begin{lemma}\label{lem:regval}
The origin $0\in \h^*$ is a regular value of the momentum map $\tilde{\mu}\colon \cM \to \tilde{\g}^*[2]$ if and only if
\begin{itemize}
\item[(a)] $0$ is a regular value of $\mu \colon  M \rightarrow \h^*$,
\item[(b)]   $\varrho \colon   \A_M\to  E$ is fiber-wise injective at every point of $\mu^{-1}(0)$,
\item[(c)] $\sigma\circ \varphi\colon \g_M \to TM$ is fiber-wise injective at every point of $\mu^{-1}(0)$. 
\end{itemize}
\end{lemma}

Notice that condition (c) states that the   action of $\g$ on $\mu^{-1}(0)$ is locally free.

Assuming that $0 \in \h^*$ is a regular value of $\widetilde{\mu}$, we saw in \S \ref{sec:hamred} that $\cN = \widetilde{\mu}^{-1}(0)$ is a coisotropic submanifold with body $N=\mu^{-1}(0)$. We will now describe the corresponding geometric coisotropic data $(N, K, F, \nabla)$ (as in Thm.~\ref{thm:coisotropic}).

By the equivariance of $\mu\colon M \to \h^*$, the submanifold $N$ is $\g$-invariant, and we denote the restricted action of $\g$ on $N$  by $u \mapsto u_N := u_M |_N$.  As a consequence, for each $u\in \g$, the derivation $\varphi(u) \in \Gamma(\mathbb{A}_E)$ restricts to a derivation 
$$
\varphi(u)|_N \in \Gamma(\mathbb{A}_{E|_N}).
$$ 
Note also that the vector subbundle 
$$
K:=\varrho(\mathfrak{a}_M)|_N \subseteq E|_N
$$ 
is isotropic (by \eqref{eq:betaeq}) and invariant by the derivations $\varphi(u)|_N$
(by part (b) of Lemma~\ref{lem:muequiv}), i.e., $\varphi(u)|_N(K)\subseteq K$. Recall that the induced derivation of $E_{quot}= K^\perp/K$ (see \eqref{eq:maptoquot}) is denoted by $[\varphi(u)|_N]$.

\begin{lemma}\label{lem:momentcoiso}
The geometric coisotropic data $(N, K, F, \nabla)$ corresponding to the coisotropic submanifold $\cN = \widetilde{\mu}^{-1}(0)$ are given as follows:
\begin{itemize}
    \item $N=\mu^{-1}(0)$,
    \item $K= \varrho(\mathfrak{a}_M)|_N$,
    \item $F=\sigma(\varphi(\g_M))|_N = \{ u_N, : u\in \g\}$,
    \item $\nabla_{u_N} = [\varphi(u)|_N]$, for $u\in \g$.
\end{itemize}
 
\end{lemma}

\begin{proof}
The geometric characterization of submanifolds defined by regular values of maps is given in Prop.~\ref{prop:invcgeom}, from where the descriptions of $N$ and $K$ immediately follow. By \eqref{eq:regvalueKt}, the vector bundle $\widetilde{K}$ is given by $\varphi(\g_M)|_N + K\wedge E|_N$, which corresponds to $F$ and $\nabla$ as in the statement using  \eqref{eq:Fnabla}. 
\end{proof}

Let us now assume that $\widetilde{\g}$ is a DGLA, with differential $\delta$, 
see Prop.~\ref{prop:dgla}, and that $\cM$ is equipped with a Courant function $\Theta$, so that $E$ acquires a Courant algebroid structure with anchor $\rho$ and bracket $\Cour{\cdot,\cdot}$ (see $\S$ \ref{subsec:CA}).

\begin{lemma}\label{lem:dglamap}
The map $\tilde{\mu}^{\sharp}\colon \widetilde{\g}\to C(\cM)[2]$ preserves differentials (as in \eqref{eq:dglacond}, implying that $\Theta$ is reducible) if and only if 
\begin{itemize}
\item[(a)]  $\varrho(\delta h) = \rho^*(d(\mu^*h))$,
\item[(b)]  $\varphi(\delta a) = \llbracket \varrho(a),\cdot\rrbracket$,
\item[(c)] $\varphi(u)(\Cour{e_1,e_2})=\Cour{\varphi(u)(e_1), e_2}+
\Cour{e_1, \varphi(u)(e_2)}$,
\end{itemize}
for $h\in \h$, $a\in \mathfrak{a}$, $u \in \g$ and $e_1, e_2 \in \Gamma(E)$.
\end{lemma}

\begin{proof}
The conditions for $\widetilde{\mu}^\sharp$ to preserve differentials are
$$
\{\Theta,\mu^*h\}=\varrho(\delta h), \qquad \{\Theta, \varrho(a)\}=\varphi(\delta a), \qquad \{\Theta, \varphi(u)\}=0,
$$
for all $h\in \h$, $a\in \mathfrak{a}$, $u\in \g$. The first two conditions coincide with $(a)$ and $(b)$ since
$$
\{\Theta, f\} =\rho^* df,\qquad \{\Theta, e\}=\Cour{e, \cdot},
$$
for $f\in C^\infty(M)$ and $e\in \Gamma(E)$. The last condition is the vanishing of the degree 3 function $\{\Theta, \varphi(u)\}$, for any $u\in \g$. 
This condition is equivalent to 
$$
\{\{\{\Theta, \varphi(u)\},e_1\}, e_2\}=0 
$$
for all $e_1, e_2\in \Gamma(E)$. Indeed, for any degree $3$ function $S$,  the coordinate expression \eqref{eq:T} shows that the derived brackets $\{\{S,e_1\}, e_2\}$ and $\{\{S,e\}, f\}$  determine   $S$, where   $e_1,e_2$ range over $\Gamma(E)$ and $f$ over $C^{\infty}(M)$; further, the vanishing of the first expression for all $e_1, e_2$ implies the vanishing of the second, by the Leibniz rule.  Now, by the Jacobi identity, note that
\begin{align*} 
\{  \{ \{\Theta,\varphi(u)\},e_1\}, e_2  \}&= \{
\{\Theta,\{\varphi(u),e_1\}\},e_2\} - \{ \{\varphi(u),\{\Theta, e_1
\}\}, e_2\}\\ 
&= \Cour{\varphi(u)(e_1), e_2} - \left(\{\varphi(u),\{ \{\Theta,e_1 \}, e_2\}\} -\{
\{\Theta, e_1\}, \{\varphi(u),e_2\} \}\right)\\ 
&= \Cour{\varphi(u)(e_1), e_2} - \varphi(u)(\Cour{e_1, e_2}) +
\Cour{e_1, \varphi(u)(e_2)}.
\end{align*}
\end{proof}

%%%%%%%%%%%%%%%%%%%%%%%%%%%%%%%%%%%%%%%%
\subsection{Hamiltonian reduction of Courant, Dirac and GC structures}\label{subsec:hamCAred}

Using hamiltonian symplectic reduction in degree 2 and the results in \S \ref{subsec:deg2ham}, we now formulate \emph{ in classical terms} reduction procedures for Courant, Dirac and GC structures.

We start by recalling the general setup.

\medskip

\noindent{\em {\bf (A)} Objects to be reduced}: We consider a Courant algebroid $E\to M$ with pairing $\SP{\cdot,\cdot}$, anchor map $\rho\colon E\to TM$ and bracket $\Cour{\cdot,\cdot}$ on $\Gamma(E)$. We may equip $E$ with a Dirac structure $L\subset E$, or with a generalized complex structure $\J\colon E\to E$.

\medskip

\noindent{\em {\bf (B)} Objects that act (DGLAs of degree 2)}:
Following Propositions \ref{prop:gla} and \ref{prop:dgla},
we consider a Lie algebra $\g$, along with $\g$-modules $\mathfrak{a}$ and $\h$, and a $\g$-equivariant symmetric bilinear map $\varpi\colon \mathfrak{a} \otimes \mathfrak{a} \to \h$.
Suppose that there are operators
\begin{equation}\label{eq:deltaseg}
\h \stackrel{\delta}{\longrightarrow} \mathfrak{a} \stackrel{\delta}{\longrightarrow} \g, 
\end{equation}
with $\delta^2=0$, and that preserve the $\g$-module structures (where $\g$ is viewed as a $\g$-module with respect to the adjoint representation). Assume moreover that 
\begin{equation}\label{eq:condsbeta}
\delta \varpi(a_1,a_2) = (\delta a_1) \cdot a_2 + (\delta a_2) \cdot a_1, \quad (\delta a) \cdot h = \varpi(a, \delta h),
\end{equation}
for $a$, $a_1$, $a_2 \in \mathfrak{a}$ and $h\in \h$.

\medskip

\noindent{\em {\bf (C)} Infinitesimal hamiltonian actions}:
Following Lemmas~\ref{lem:muequiv} and \ref{lem:dglamap}, 
we suppose that the vector bundle $E$ is equipped with a Lie algebra morphism $\varphi\colon \g \to \Gamma(\mathbb{A}_E)$, that makes $E$ into a $\g$-equivariant vector bundle and induces a $\g$-action on $M$. We further assume that there are $\g$-equivariant maps $\mu\colon M\to \h^*$  and $\varrho\colon \mathfrak{a}\to \Gamma(E)$ such that 
\begin{equation}\label{eq:Kisot}
\mu^*(\varpi(a_1,a_2))=\SP{\varrho(a_1),\varrho(a_2)},
\end{equation}
for $a_1, a_1 \in \mathfrak{a}$,
and that 
\begin{equation}\label{eq:courantcomp}
\begin{split}
& \varrho(\delta h) = \rho^*(d(\mu^*h)),\\
%&    \SP{\varrho(\delta h), e} = \pounds_{\rho(e)} (\mu^*h), \\
&    (\delta a)\cdot e = \Cour{\varrho(a), e},\\
&    u\cdot \Cour{e_1,e_2} =\Cour{u\cdot e_1, e_2}+
\Cour{e_1, u\cdot e_2},
\end{split}
\end{equation}
for $u\in \g$, $a \in \mathfrak{a}$, $h \in \h$ and $e, e_1, e_2\in \Gamma(E)$. 
(Here we used the notation $\varphi(u)(e)= u \cdot e$.)

\begin{defi}\label{def:infHact}
We say that the maps $\varphi$, $\varrho$ and $\mu$ define
a {\em hamiltonian action} of the degree 2 DGLA $(\g, \A, \h, \delta, \varpi)$ (as in {\bf (B)}) on the Courant algebroid $E\to M$. 
\end{defi}

\begin{remark}\label{rem:extracond}
\begin{itemize}

\item[$(i)$] 
 The first  and second identities in \eqref{eq:courantcomp} 
amount to the commutativity of the following diagram:
$$
 \xymatrix{
 \h \ar[d]_{\mu^*}\ar[r]^{\delta} & \A \ar[d]^{{\varrho}}\ar[r]^{\delta}
 & \g \ar[d]_{\varphi}
 \\
 C^\infty(M) \ar[r]^{\rho^* d} & \Gamma(E)\ar[r]^{\mathrm{ad}}
&\Gamma(\mathbb{A}_E).}
 $$
The third identity states that $\varphi$ takes values in infinitesimal Courant algebroid automorphisms.

\item[$(ii)$] It follows from the first two conditions in \eqref{eq:courantcomp}, the first identity in \eqref{eq:condsbeta}, the equivariance of $\varrho$ and axiom (C5) of Courant algebroids that 
$$
\rho^* d (\mu^*(\varpi(a_1,a_2)))= \rho^* d (\SP{\varrho(a_1),\varrho(a_2)}),
$$
for $a_1, a_1 \in \mathfrak{a}$, which is a weaker version of \eqref{eq:Kisot}.
\end{itemize}
\hfill $\diamond$
\end{remark}

\medskip

\noindent{\em {\bf (D)} Regularity assumptions}: Following  Lemma~\ref{lem:regval}, we assume that $0\in \h^*$ is a regular value for $\mu\colon M\to \h^*$ and that, for each $x\in \mu^{-1}(0)$, the map 
\begin{equation}\label{eq:aK}
\mathfrak{a}\to E|_x, \qquad a \mapsto \varrho(a)|_x, 
\end{equation}
 is injective. 
Note that, by the $\g$-equivariance of $\mu$, the $\g$-action on $E$ restricts to a $\g$-action on
$E|_{\mu^{-1}(0)}$. For a connected Lie group $G$ integrating $\g$, we will assume that this infinitesimal action integrates to a $G$-action on $E|_{\mu^{-1}(0)}$ for which the restricted $G$-action on $\mu^{-1}(0)$ is free and proper.

\medskip
\noindent{\em Some consequences of the setup}:
The images of \eqref{eq:aK} for each $x\in \mu^{-1}(0)$ define a vector subbundle 
    \begin{equation}\label{eq:K}
    K\subseteq E|_{\mu^{-1}(0)}
    \end{equation} 
    that is isotropic (by \eqref{eq:Kisot}) and $G$-invariant (as a consequence of the $\g$-equivariance of $\varrho$ and the fact that $G$ is connected). Note that the geometric coisotropic data $(N, K, F, \nabla)$ defined in Lemma~\ref{lem:momentcoiso} coincides with the one of Example~\ref{ex:coisodataaction}. As seen in Example~\ref{ex:coisoredaction}, in the present setup $F$ is simple and 
    $\nabla$ has trivial holonomy, and the reduction of the pseudo-euclidean vector bundle $E$ is given by $E_{red}=(K^\perp/K)/G$ over $M_{red}  = \mu^{-1}(0)/G$,
$$
\xymatrix{K^{\perp}/K   \ar[r]^{} \ar[d]_{} & E_{red} \ar[d]^{} \\
\mu^{-1}(0) \ar[r]_{} &  M_{red}.}
$$

By Thm.~\ref{thm:coisored}, $E_{red}$ is the pseudo-euclidean vector bundle corresponding to $\cM_{red}$, the degree 2 symplectic  $\N$-manifold obtained by hamiltonian reduction described in $\S$ \ref{sec:hamred}.

\begin{thm}\label{thm:geomhamred} 
 Consider a hamiltonian action of a DGLA of degree 2 (as in {\bf (B)}) on a Courant algebroid $E\to M$ satisfying the regularity conditions in {\bf (D)}.
%In the setup given by {\bf (A), (B), (C), (D)} above, 
Then the following conclusions hold:
\begin{itemize}
\item[(a)] The pseudo-euclidean vector bundle $E_{red}\to \mu^{-1}(0)/G$ inherits a Courant algebroid structure (with anchor and bracket defined as in \eqref{eq:courantred});
\item[(b)] If $L\subseteq E$ is a $G$-invariant Dirac structure such that $L|_{\mu^{-1}(0)}\cap K$ has constant rank, then it gives rise to a Dirac structure $L_{red}$ in $E_{red}$ (as in Thm.~\ref{thm:diracred}); 
\item[(c)] If $\J$ is a $G$-invariant GC structure on $E$ such that $\J(K)\subseteq K$, then $E_{red}$ inherits a GC structure $\J_{red}$ (as in Thm.~\ref{thm:redGCS}).
\end{itemize}
\end{thm}

In part (a), the reducibility of the Courant algebroid structure is ensured by \eqref{eq:courantcomp} (following Lemma~\ref{lem:dglamap} and  \S\ref{sec:hamred}). Parts (b) and (c) are special cases of Theorems~\ref{thm:diracred} and \ref{thm:redGCS}, respectively.

\begin{remark}\label{rem:exactCA} Prop.~\ref{prop:exactCA} gives conditions for $E_{red}$ in part (a)  to be an exact Courant algebroid; e.g., it is enough that $\rho(K^\perp) = T(\mu^{-1}(0))$ and that $\rho(K)$ agrees with the distribution tangent to the $G$-orbits on $\mu^{-1}(0)$. \hfill $\diamond$
\end{remark}

\begin{remark}
One can slightly weaken the invariance hypotheses in the previous theorem: 
in (b) it is enough to require the $G$-invariance of $L_{quot}\subset K^\perp/K$, and in (c) the $G$-invariance of the map $\J_{quot}\colon K^\perp/K \to K^\perp/K$, induced by $\J$. \hfill $\diamond$
\end{remark}

We now illustrate the specific roles of $\mathfrak{h},\mathfrak{a}$ and $\mathfrak{g}$ in ${\bf (B)}$ (i.e., the components of the DGLA) in the reduction procedure of a Courant algebroid $E$, described in part (a) of the previous theorem.
We consider situations where precisely one of  $\mathfrak{h},\mathfrak{a}$ or $\mathfrak{g}$ is non-zero.

\begin{ex}\

\begin{itemize}
\item When only $\mathfrak{h}$ is non-zero, it is just a vector space, and  the previous setup consists of a smooth map $\mu\colon M\to \h^*$ such that $\rho(E)\subseteq \mathrm{ker}(d\mu)$. If $0$ is a regular value of $\mu$, then $\mu^{-1}(0)$ is a submanifold with the property that $\rho(E) \subseteq T(\mu^{-1}(0))$. The reduced Courant algebroid in this case is the restriction 
$E|_{\mu^{-1}(0)}$ (see Example~\ref{ex:casescoisored} (i)).
\item If only $\mathfrak{a}$ is nonzero, then it is a vector space, and the previous setup consists of a linear map $\varrho\colon \mathfrak{a}\to \Gamma(E)$  such that $\SP{\varrho(a_1),\varrho(a_2)}=0$ and $\Cour{\varrho(a), \cdot}=0$ for all $a_1$, $a_2$ and $a\in \mathfrak{a}$ (by \eqref{eq:Kisot} and the second equation in \eqref{eq:courantcomp}). Assuming that $\varrho$ is fiberwise injective, its image defines an isotropic subbundle $K\subset E$, and the reduction scheme in this case says that $K^{\perp}/{K}\rightarrow M$ has an induced Courant algebroid structure. 
  \item When only $\mathfrak{g}$ is non-zero, the setup consists of
a Lie algebra morphism $\g \to \Gamma(\mathbb{A}_E)$ acting by infinitesimal Courant algebroid automorphisms (the third equation in \eqref{eq:courantcomp}). Assume that $G$ is a connected Lie group integrating $\g$ and that the infinitesimal action on $M$ integrates to a $G$-action that is free and proper. In this case, the reduced Courant algebroid is $E/G\to M/G$, which is simply the quotient of $E$ by an action of $G$ by Courant algebroid automorphisms.  
\end{itemize}
\hfill $\diamond$
\end{ex}

We will see more concrete examples in $\S$ \ref{subsubsec:reddata} below.

In the full reduction scheme for the Courant algebroid $E\to M$ in Theorem~\ref{thm:geomhamred}, we see that $\h$ is used to cut out the submanifold $\mu^{-1}(0)$, $\mathfrak{a}$ is used to define the vector bundle $K^\perp/K \to \mu^{-1}(0)$, and $\g$ is used to quotient this bundle.

%%%%%%%%%%%%%%%%%%%%%%%%%%%%

\subsection{The case of exact DGLAs}
\label{subsec:extactions}
We will now consider the reduction setup of the previous subsection
in the special case where the DGLA of degree 2 is {\em exact}, i.e., the complex \eqref{eq:deltaseg} defines a short exact sequence
$$
0 \longrightarrow \h \stackrel{\delta}{\longrightarrow} \mathfrak{a} \stackrel{\delta}{\longrightarrow} \g \longrightarrow 0.
$$
In the context of exact Courant algebroids, we will see how this special setup recovers the reduction schemes of \cite{bcg,HKquot}.

%%%%%%
\subsubsection{\underline{Courant algebras}}
We now show that exact DGLAs of degree 2 admit an alternative description in terms of the following objects, introduced in \cite[Definitions~ 2.6 and 2.7]{bcg}.

 \begin{defi}\label{def:Calgebra} A \emph{Courant algebra} over a Lie algebra $\mathfrak{g}$
is a vector space $\mathfrak{a}$ with a bilinear bracket
$\llbracket \cdot, \cdot \rrbracket$ and a map $p: \mathfrak{a}
\to \mathfrak{g}$ such that, for all $a_1, a_2, a_3 \in
\mathfrak{a}$,
\begin{itemize}
\item[(a)] $\llbracket a_1, \llbracket a_2, a_3 \rrbracket \rrbracket =
\llbracket \llbracket a_1,  a_2\rrbracket, a_3  \rrbracket +
\llbracket a_2, \llbracket a_1, a_3 \rrbracket \rrbracket$, that is, $(\A, \Cour{\cdot,\cdot})$ is a Leibniz algebra;
\item[(b)]
$p( \llbracket a_1, a_2 \rrbracket) = [p(a_1), p(a_2)]$.
\end{itemize}
We define a  Courant algebra to be \emph{exact} if $p$ is surjective and
%$\mathfrak{h} \defequal \ker p$ is abelian.  We will make the stronger assumption that
$\mathfrak{h}:=\mathrm{ker}(p)$ is left-central, i.e.\ $\llbracket h, a \rrbracket = 0$ for all $h
\in \mathfrak{h}$, $a \in \mathfrak{a}$.
\end{defi}

Just as in \cite{bcg,HKquot}, we will only be concerned with exact Courant algebras. We note that the previous definition is a mild modification of the original one in \cite[Def,~2.7]{bcg}, where $\h$ was just required to be  abelian (i.e., $\llbracket \h, \h \rrbracket = 0$) rather than left-central. 
The next result indicates  that our slightly stronger condition is actually more natural, and as we will see in $\S$ \ref{subsubsec:reddata} below, not restrictive.

\begin{prop}\label{lem:exact}
There is a canonical one-to-one correspondence between
exact DGLAs of degree 2 and exact Courant algebras.
\end{prop}

\begin{proof} 
Let  $\widetilde{\g}=\h[2]\oplus \mathfrak{a}[1] \oplus \g$ be  a DGLA of degree 2, given in terms of the data in ${\bf (B)}$ of $\S \ref{subsec:hamCAred}$, that we assume to be exact. 
We immediately obtain a surjective map $p:=\delta \colon \mathfrak{a}
\to \mathfrak{g}$. Define a bracket on $\mathfrak{a}$ by
 the derived bracket formula
\begin{equation}\label{eq:calg}
\llbracket a_1, a_2 \rrbracket = [\delta a_1, a_2] = \delta a_1\cdot a_2.
\end{equation}
The fact that $\delta \colon \mathfrak{a}
\to \mathfrak{g}$ is a morphism of $\g$-modules implies that 
$$
p(\llbracket a_1, a_2 \rrbracket)= \delta (\delta a_1\cdot a_2)=\delta a_1\cdot \delta a_2 = [p(a_1),p(a_2)].
$$
Using this last condition and that $\mathfrak{a}$ is a $\g$-module,
we see that
\begin{align*}
\Cour{a_1, \Cour{a_2,a_3}} & =  \delta a_1 \cdot (\delta a_2\cdot a_3)= [\delta a_1, \delta a_2]\cdot a_3 + \delta a_2 \cdot (\delta a_1\cdot a_3)\\
& = (\delta (\delta a_1\cdot a_2))\cdot a_3 + \delta a_2 \cdot (\delta a_1\cdot a_3) \\
& = \Cour{\Cour{a_1, a_3},a_3} + \Cour{a_2, \Cour{a_1,a_3}}.
\end{align*}
So $\llbracket \cdot, \cdot \rrbracket$ satisfies both conditions displayed
in Def. \ref{def:Calgebra}. Note  that $\Cour{a_1,\cdot}=0$ whenever $\delta a_1= p(a_1) = 0$, hence $\mathrm{ker}(p)\cong \h$ is left-central. So $p\colon \mathfrak{a}\to \g$ and the bracket \eqref{eq:calg} make $\mathfrak{a}$ into an exact Courant algebra over $\g$.

Conversely, consider an exact Courant algebra $p \colon \mathfrak{a}
\to \mathfrak{g}$, with bracket $\Cour{\cdot,\cdot}$ on $\mathfrak{a}$ and $\h=\mathrm{ker}(p)$. We can now produce the data described in {\bf (B)} of $\S$ \ref{subsec:hamCAred} as follows. Take $\g$ with the given Lie-algebra structure. Then $\mathfrak{a}$ acquires a $\g$-module structure via
$$
u\cdot a := \llbracket \hat{u}, a \rrbracket,
$$
where $\hat{u}\in \mathfrak{a}$ is such that $p(\hat{u})=u$. (This is well-defined since $\h$ is left-central; the two conditions in Def. \ref{def:Calgebra} ensure that  it is a $\g$-module.) Note that $\h$
is a $\g$-submodule of $\mathfrak{a}$. We finally define the symmetric bilinear map $\varpi \colon \mathfrak{a} \otimes \mathfrak{a} \to \mathfrak{h}$ by 
\begin{equation}\label{eq:beta}
\varpi(a_1,a_2):=\llbracket a_1, a_2 \rrbracket+\llbracket a_2, a_1 \rrbracket = p(a_1)\cdot a_2 + p(a_2)\cdot a_1.
\end{equation}
(The fact that $\varpi$ takes values in $\h$ follows from $p\colon \mathfrak{a}\to \h$ being bracket preserving.) Using the first condition in Def.~\ref{def:Calgebra}, one can check that $\varpi$ is $\g$-equivariant. The operators $\delta$ in \eqref{eq:deltaseg} are defined as the inclusion $\h \hookrightarrow \mathfrak{a}$ and projection $p\colon \mathfrak{a}\to \g$. The conditions in \eqref{eq:condsbeta} are both satisfied: the first is just by the definition of $\varpi$ above, and the second follows from $\h$ being left central. By Propositions \ref{prop:gla} and \ref{prop:dgla},  $\widetilde{\g}=\h[2]\oplus \mathfrak{a}[1] \oplus \g$ becomes a DGLA of degree 2, with bracket given by
\begin{equation*}
\begin{split}
&[h_1 + a_1 + u_1, h_2 + a_2 + u_2] = \\
&= 
 \left( \llbracket \hat{u}_1, h_2
\rrbracket + \llbracket a_1, a_2 \rrbracket + \llbracket a_2, a_1
\rrbracket - \llbracket \hat{u}_2, h_1 \rrbracket \right)
+ \left(\llbracket \hat{u}_1, a_2 \rrbracket - \llbracket
\hat{u}_2, a_1 \rrbracket \right) 
+[u_1,u_2],
\end{split}
\end{equation*}
where $\hat{u}_i \in \mathfrak{a}$ are such that $p(\hat{u}_i) = u_i$.
\end{proof}

%%%%%%%%%%%%%%%%%%%%%%%%%%%%%%%%%%%%%%%%%%%%%%%%%%%%
\subsubsection{\underline{Infinitesimal hamiltonian actions and reduction data}}\label{subsubsec:reddata}
We now see how the data for an infinitesimal hamiltonian action on a Courant algebroid, described in ${\bf (C)}$ of $\S$ \ref{subsec:hamCAred}, can be simplified if we assume the DGLA to be exact. 
 
\begin{prop}\label{prop:exactDGLAact} 
A hamiltonian action of an exact DGLA of degree 2, with corresponding Courant algebra $p\colon \A \to \g$, on a Courant algebroid $E\to M$ is equivalent to:
\begin{itemize}
\item a bracket preserving map $\varrho\colon \A \to \Gamma(E)$, such that $\Cour{\varrho(\h),\cdot}=0$,
\item a $\g$-equivariant map $\mu\colon M \to \h^*$ such that
\begin{itemize}
\item[$(i)$] $\varrho(h)=\rho^*(d (\mu^*h))$, for all $h\in \h$,
\item[$(ii)$] $\mu^*(\Cour{a,a})=\frac{1}{2}\SP{\varrho(a),\varrho(a)}$, for all $a\in \A$.
\end{itemize}
\end{itemize}
\end{prop}

Before proving the proposition, we need to explain the equivariance condition on $\mu$. Note that the condition $\Cour{\varrho(\h),\cdot}=0$ implies that $\rho (\varrho(\h))=0$ (by the Leibniz identity (C2) of Courant brackets), so $\rho\circ \varrho$ induces a Lie-algebra map $\g = \A/\h \to \Gamma(TM)$, i.e., a $\g$-action on $M$. The equivariance of $\mu$ is with respect to this action.

\begin{proof}
Assume that \eqref{eq:deltaseg} defines a short exact sequence. We must check that the data in ${\bf (C)}$ of $\S$ \ref{subsec:hamCAred}, given by 
\begin{itemize}
\item a Lie algebra morphism $\varphi\colon \g \to \Gamma(\mathbb{A}_E)$, 
%(defining $\g$-actions on $E$ and $M$),
\item $\g$-equivariant maps $\varrho\colon \A \to \Gamma(E)$ and $\mu\colon M \to \h^*$
\end{itemize}
and satisfying conditions \eqref{eq:Kisot} and \eqref{eq:courantcomp}, reduce to the data in this proposition. 

By exactness, the second condition in \eqref{eq:courantcomp} says that $\Cour{\varrho(\h),\cdot}=0$ and that $\varphi$ is completely determined by $\varrho$ via 
\begin{equation}\label{eq:phibrk}
\varphi(u) = \Cour{\varrho(\hat{u}),\cdot},    
\end{equation}
where $\hat{u}\in \A$ satisfies $\delta \hat{u} = p(\hat{u}) = u$. In particular, the third condition in 
\eqref{eq:courantcomp} is automatically satisfied (by axiom (C1) of Courant brackets). 
By \eqref{eq:calg} and \eqref{eq:phibrk}, we see that in terms of the Courant algebra $\A\to \g$, the $\g$-equivariance of $\varrho\colon \A \to \Gamma(E)$ is equivalent to $\varrho$ being bracket preserving. Further, \eqref{eq:Kisot} can be writtten as the second item  of the proposition, while the first item agrees with the first  condition in 
\eqref{eq:courantcomp}. This shows that the maps $\varrho$ and $\mu$  in ${\bf (C)}$ satisfy the conditions stated in the proposition. On the other hand, we also see that starting with the data in this proposition, by setting  $\varphi\colon \g \to \Gamma(\mathbb{A}_E)$ as in \eqref{eq:phibrk} we obtain the data as in ${\bf (C)}$, yielding the claimed equivalence.
\end{proof}

To obtain a further simplification of the setup, we will consider the following concrete Courant algebra, which is the main example in \cite{bcg,HKquot}.

\begin{ex}[Hemisemidirect product]\label{ex:hemisemi}
Given a Lie algebra $\g$ and a $\g$-module $\h$, there is a natural exact Courant algebra structure on $\A= \g\oplus \h$ with $p\colon \A \to \g$ the canonical projection and  bracket defined by
$$
\Cour{(u_1,h_1),(u_2,h_2)}= ([u_1,u_2],u_1\cdot h_2).
$$
The previous bracket is the hemisemidirect product of $\g$ and $\h$, first considered in 
\cite[Example 2.2]{MR1833152}.
The corresponding exact DGLA $\tilde{\g}={\h}[2]\oplus \A[1]\oplus \g$ is determined by the $\g$-module structure on $\A= \g\oplus \h$ given by $u \cdot (u_1,h_1) = ([u,u_1], u\cdot h_1)$ and 
$$
\varpi((u_1,h_1),(u_2,h_2)) =  u_1\cdot h_2 + u_2\cdot h_1.
$$
\hfill{$\diamond$}
\end{ex}

Suppose that the Courant algebroid  $E\to M$ has the additional property that $\rho^* d( C^\infty(M))\subseteq \Gamma(E)$ is left central, i.e.,
\begin{equation}\label{eq:leftcentralCA}
\Cour{\rho^*d (C^\infty(M)),\cdot}=0.
\end{equation}
Let $\g$ be a Lie algebra and $\h$ be a $\g$-module.

\begin{defi}\label{def:reddata}
 A pair $(\tilde{\psi},\mu)$ 
consisting of
\begin{itemize}
\item a bracket preserving map $\tilde{\psi}\colon \g \to \Gamma(E)$  with isotropic image, i.e., $\SP{\tilde{\psi}(u), \tilde{\psi}(v)}=0$ for all $u, v \in \g$,
\item a $\g$-equivariant map $\mu\colon M \to \h^*$, where $M$ is equipped with the action 
$$
\psi= \rho \circ \widetilde{\psi}\colon \g \to \Gamma(TM),
$$
\end{itemize}
is called {\em reduction data} for $E$.
\end{defi}

The previous definition extends the notion of reduction data in \cite[$\S$ 3.2]{HKquot}.

The importance of reduction data is that they give rise to hamiltonian actions on Courant algebroids. Let $\A= \g\oplus \h$ be the Courant algebra of Example~\ref{ex:hemisemi}. 

\begin{prop}\label{prop:reddata}
Let $E$ be a Courant algebroid satisfying \eqref{eq:leftcentralCA} and equipped with
reduction data $(\tilde{\psi},\mu)$. Then  $\varrho\colon \A\to \Gamma(E)$,
$$
\varrho(u,h) := \widetilde{\psi}(u) + \rho^*d (\mu^*h),
$$
and $\mu$ define a hamiltonian action of $\A$ on $E$. 
\end{prop}

\begin{proof}
We must verify that the conditions in Prop. ~\ref{prop:exactDGLAact} hold. It is clear that $\varrho(0,h)=\rho^*d(\mu^*h)$ for all $h\in \h$. By \eqref{eq:leftcentralCA}, it follows that $\Cour{\varrho(\h),\cdot}=0$.
To check that $\varrho$ is bracket preserving, notice that
\begin{align*}
\Cour{\varrho(u_1,h_1),\varrho(u_2,h_2)} &=  \Cour{\tilde{\psi}(u_1)+ \rho^* d(\mu^* h_1), \tilde{\psi}(u_2)+ \rho^* d(\mu^* h_2)}\\
&=\tilde{\psi}([u_1,u_2]) + \rho^* d (\pounds_{{\psi}(u_1)}\mu^*h_2) \\
&= \tilde{\psi}([u_1,u_2]) + \rho^*d(\mu^*(u_1\cdot h_2))\\
&= \varrho(\Cour{(u_1,h_1),(u_2,h_2)}),
\end{align*}
where we used that $\tilde{\psi}$  preserves brackets as well as axiom (C5) of Courant algebroids in the second equality, and the equivariance of $\mu$ in the third equality. It remains to verify that 
$$
\mu^*(\Cour{(u,h),(u,h)})=\frac{1}{2}\SP{\varrho(u,h),\varrho(u,h)}
$$
for all $(u,h)\in \g\oplus \h$. By equivariance, the left-hand side equals
$$
\mu^*(u\cdot h)= \pounds_{\psi(u)}\mu^*h.
$$
On the other hand,  we have
$$
\SP{\varrho(u,h),\varrho(u,h)} = 2\SP{\widetilde{\psi}(u), \rho^*d(\mu^*h)}= 2 \pounds_{\psi(u)}\mu^*h,
$$
where we have used that $\tilde{\psi}$ has isotropic image and that $\rho\circ \rho^*=0$ on Courant algebroids. This completes the proof.
\end{proof}

We will now use the simplified description of hamiltonian actions via reduction data in Prop.~\ref{prop:reddata} to illustrate Theorem~\ref{thm:geomhamred}.

We will fix the following set-up for Examples~\ref{ex:hamcour}, \ref{ex:hamdirac} and \ref{ex:hamGCS} below:
\begin{itemize}
\item $G$ is a (connected) Lie group acting on a manifold $M$, and $\psi: \g\to \Gamma(TM)$ is the corresponding infinitesimal action;
\item $\h$ is a $G$-module and $\mu: M\to \h^*$ is a $G$-equivariant map;
\item $E=TM\oplus T^*M$ is the standard Courant algebroid.
\end{itemize}
Note that condition \eqref{eq:leftcentralCA} holds for $E$, and that
$(\psi,\mu)$ is reduction data (as in Def.~\ref{def:reddata}). By the previous proposition, we have a hamiltonian action of the (hemisemidirect product) Courant algebra $\g\oplus \h$ on $E$ with 
$$
\varrho(u,h)=\psi(u) + d(\mu^*h).
$$
Suppose that $0$ is a regular value of $\mu$ and the $G$-action on $\mu^{-1}(0)$ is free and proper, and let $M_{red}=\mu^{-1}(0)/G$.

\begin{ex}[Hamiltonian reduction of the standard Courant algebroid]\label{ex:hamcour}
To carry out Courant reduction as in Theorem~\ref{thm:geomhamred} (a), we note that the induced $G$-action on $E=TM\oplus T^*M$ (defined by \eqref{eq:phibrk}) coincides with the natural $G$-action by tangent and cotangent lifts, and $K$ (defined by the image of $\varrho$) is $F\oplus \mathrm{Ann}(TN)$, where $N=\mu^{-1}(0)$ and $F$ is the tangent distribution to the $G$-orbits on $N$. Similarly to Example~\ref{red:stdca}, in this case we have
$$
E_{red} = TM_{red}\oplus T^*M_{red},
$$
the standard Courant algebroid on $M_{red}$.
\hfill $\diamond$
\end{ex}

\begin{ex}[Hamiltonian reduction of Dirac structures]\label{ex:hamdirac}
Assume that $\h=\g$, and let $L\subset E$ be a Dirac structure satisfying the following ``momentum map condition'' (see \cite[Def.~2.15]{BrahicFernandesDiracRed}):
$$
\psi(u) + d(\mu^*u) \in \Gamma(L), \qquad \forall u\in \g.
$$
When $L$ is defined by a Poisson structure $\pi$ (resp. closed 2-form $\omega$), one recovers the familiar condition $\psi(u)=\pi^\sharp(d(\mu^*u))$ (resp. $i_{\psi(u)}\omega= d(\mu^*u)$).

With this condition, the fact that $L$ is involutive with respect to the Courant bracket ensures that it is $\g$-invariant (and hence $G$-invariant since $G$ is connected). To carry out the reduction in Theorem~\ref{thm:geomhamred} (b) and obtain a Dirac structure $L_{red}$ on $M_{red}$, we need the additional assumption that $L|_{\mu^{-1}(0)}\cap K$ has constant rank; here $K$ is as in Example \ref{ex:hamcour}. (Notice that, when $L$ is given by a Poisson structure or closed 2-form, this condition follows from the momentum map condition, so it is automatically satisfied.)

We can re-formulate this constant-rank condition as follows.

\smallskip

\noindent {\bf Claim:} $L|_{\mu^{-1}(0)}\cap K$ has constant  rank if and only if 
$(d_x\mu)(\mathrm{pr}_{TM}(L))\subseteq T_0\g^*$ has constant dimension for all $x\in \mu^{-1}(0)$.

\smallskip

To verify the claim, recalling that $N=\mu^{-1}(0)$, note that for $u\in \g$ and $\xi\in\mathrm{Ann}(TN)$ we have: $(\psi(u),\xi)\in L|_{N}\cap K$ if and only if $\xi-d(\mu^*u)\in \mathrm{Ann}(TN+\mathrm{pr}_{TM}L|_N)$ (as one sees using the momentum map condition and the equality $L\cap T^*M=\mathrm{Ann}(\mathrm{pr}_{TM}(L))$). This shows that $L|_{N}\cap K$ has constant rank if and only if $TN+\mathrm{pr}_{TM}(L)|_N$ has constant rank, which is is equivalent to $(d_x\mu)(\mathrm{pr}_{TM}(L)|_N)\subseteq T_0\g^*$ having constant dimension for all $x\in N$.

By the claim, the constant-rank condition holds, for instance, when the restriction of $\mu$ to any presymplectic leaf of $L$ intersecting $N$  has $0$ as a regular value, i.e. when
the momentum map $\mu$ is \emph{regular at $0\in \g^*$} in the sense of \cite[Def. 2.18]{BrahicFernandesDiracRed}. In this case, the reduced Dirac structure $L_{red}$ of Theorem~\ref{thm:geomhamred} (b) coincides with the Dirac structure on $\mu^{-1}(0)/G$ constructed in \cite[Thm. 2.21]{BrahicFernandesDiracRed}.
\hfill $\diamond$
\end{ex}

\begin{ex}[Hamiltonian reduction of GCS]\label{ex:hamGCS}
Suppose that $M$ is equipped with a $G$-invariant generalized complex structure $\J$.
Theorem~\ref{thm:geomhamred} (c) asserts that $\J$ reduces to $\J_{red}$ on $M_{red}=\mu^{-1}(0)/G$ provided $\J(K)=K$. We illustrate this result in two cases, parallel to Examples~\ref{ex:coisosympred} and \ref{ex:coisoholred}.
\begin{itemize}
\item Let $\h=\g$, and suppose that $\J$ corresponds to a symplectic form $\omega$ on $M$,
\[ 
\J = \left( \begin{matrix} 
0 & -\omega^{-1} \\ 
\omega & 0 
\end{matrix}\right). 
\] 
In case $\mu:M\to \g^*$ is a momentum map for the $G$-action on $(M,\omega)$, i.e.,
$$
i_{\psi(u)}\omega = d (\mu^*u), \qquad \forall u\in \g,
$$
we have that $\J \varrho(u,v)  = \varrho(-v,u)$, and hence $\J(K)=K$. The reduced GCS $\J_{red}$ is the one corresponding to the reduced symplectic form on $M_{red}$ obtained by Marsden-Weinstein reduction. 
\item Let $\h=\{0\}$, and suppose that $\J$ corresponds to a complex structure $J$ on $M$,
\[ 
\J = \left( \begin{matrix} 
J & 0 \\ 
0 & -J^*
\end{matrix}\right). 
\] 
Assuming that $\g$  is a complex Lie algebra, with complex structure denoted by $I$, and that the action is holomorphic, we have that 
$$
\J\varrho(u) = J\psi(u) = \psi(I u) = \varrho(I u),
$$
so $\J(K)=K$. The reduced GCS $\J_{red}$ on $M_{red}=M/G$ corresponds to the complex structure obtained by holomorphic quotient.
\end{itemize}
\hfill $\diamond$
\end{ex}

A more general class of Courant algebroids satisfying  \eqref{eq:leftcentralCA}  can be obtained from Lie algebroids, as explained in Example~\ref{ex:LAdouble}. Recall that for a Lie algebroid $A\to M$, with anchor $\rho_A$, bracket $[\cdot,\cdot]_A$, and differential $d_A$, and any $d_A$-closed element $\chi\in \Gamma(\wedge^3 A^*)$, we have a Courant algebroid structure on the pseudo-euclidean vector bundle $E = A\oplus A^*$ with anchor $\rho(a,\xi) = \rho_A(a)$ and bracket
$$
\Cour{(a_1,\xi_1),(a_2,\xi_2)} = ([a_1,a_1], \pounds_{a_1}\xi_2 - i_{a_2}d_A \xi_1 + i_{a_2}i_{a_1}\chi).
$$
It is a direct verification that 
\eqref{eq:leftcentralCA} holds for such Courant algebroids.
Particular examples are given by doubles of Lie algebroids (when $\chi=0$) and by
exact Courant algebroids (when $A=TM)$, see Example~\ref{ex:standardCA}; in this last setting, the previous proposition recovers \cite[Prop.~2.3]{HKquot}.

\begin{ex}\label{ex:reddatasuper}
When $E=A\oplus A^*$ is the double Courant algebroid of a Lie algebroid $A\to M$, 
reduction data for $E$ amount to
\begin{itemize}
\item[(1)] a Lie algebra map $\eta\colon \g\to \Gamma(A)$,
\item[(2)] a linear map $\nu\colon \g \to \Gamma(A^*)$ such that, for $u, v \in \g$,
\begin{itemize}
%\item[a)] $d_A \nu (u) = 0$,
\item[(a)] $\nu([u,v])=\pounds_{\eta(u)}\nu(v) - i_{\eta(v)}d_A \nu (u)$,  
\item[(b)] $i_{\eta(u)}\nu(v) = - i_{\eta(v)}\nu(u)$,
\end{itemize}
\item[(3)] a $\g$-equivariant map $\mu\colon M \to \h^*$, where $M$ carries the action
$\rho_A\circ \eta\colon \g \to \Gamma(TM)$.
\end{itemize}
Here $\eta$ and $\nu$ are components of a map $\tilde{\psi}\colon \g \to \Gamma(E)$, $\tilde{\psi}(u)= \eta(u) + \nu(u)$, which is bracket preserving with isotropic image, by (a) and (b) in (2) above. According to Prop.~\ref{prop:reddata}, $\eta$, $\nu$ and $\mu$ yield a hamiltonian action of the hemisemidirect product Courant algebra $\A= \g\oplus \h$ on $E$.

A very special case is when $\nu=0$ and $\h=0$, so that
the only non-trivial map is the  Lie algebra map $\eta\colon \g\to \Gamma(A)$. This defines a (inner) $\g$-action on $A$ by $u \mapsto [u, \cdot]$ lifting the $\g$-action on $M$. Suppose that the $\g$-action on $A$ integrates to a global action of a connected Lie group $G$, and that the $G$-action on $M$ is free and proper. The map $\g\times M \to A$, $(u,x)\mapsto \eta(u)|_x$ is injective (by freeness of the action) and its image, denoted by $\eta(\g)\subseteq A$, is a $G$-invariant subbundle. Then  the quotient of $A/
\eta(\g)$ by the $G$-action defines a reduced Lie algebroid $A_{red}$ over $M/G$, as considered, e.g., in \cite[Thm.~3.6]{MOP} (see also \cite[$\S$ 2.2]{BIL} and references therein). 
In this case, the hamiltonian reduction of the Courant algebroid $E=A\oplus A^*$  (as described in \S \ref{subsec:hamCAred}) is
$E_{red}= A_{red}\oplus A_{red}^*$, the double of the reduced Lie algebroid $A_{red}$.
\hfill $\diamond$
\end{ex}

 The next remark describes the graded hamiltonian action (in the sense of $\S$ \ref{sec:hamred}) corresponding to the previous example.  It is a refinement of a  cotangent lifted action on $T^*[2]A[1]$.

 \begin{remark}[Graded cotangent bundle reduction]\label{rem:short8.6}
We will be concerned with the graded  counterpart of the following construction in ordinary symplectic geometry.
Let $N$ be a  manifold and $\kappa\colon T^*N \to N$ its cotangent bundle, endowed with its canonical symplectic structure. We regard $C^\infty(N)$ as an abelian Lie subalgebra of $C^\infty(T^*N)$ via $\kappa^*$. By identifying vector fields on $N$ with linear functions on $T^*N$, we view 
$\mathfrak{X}(N) \subseteq C^\infty(T^*N)$
as a Lie subalgebra. Hence any Lie algebra map $ \g \to \mathfrak{X}(N)$ may be seen as a Lie algebra map into $C^\infty(T^*N)$, we see that any $\g$-action on $N$ automatically defines a hamiltonian $\g$-action on $T^*N$, known as its {\em cotangent lift}. A natural way
to obtain more general hamiltonian actions on $T^*N$ is by combining cotangent lifts with ``fiber translations''. Consider
\begin{itemize}
\item[(i)] a Lie algebra map  $\eta\colon \g \to \mathfrak{X}(N) \subseteq C^\infty(T^*N)$;
\item[(ii)] 
a linear map $\nu\colon \g \to C^\infty(N)$ satisfying the cocycle condition 
\begin{equation*}
%\label{eq:notequiv}
\nu([u_1,u_2]) =  \pounds_{\eta (u_1)}(\nu(u_2))-  \pounds_{\eta (u_2)}(\nu(u_1)), \;\;\;\; \forall u_1,u_2 \in \g
\end{equation*}
(this condition is saying that $\eta + \kappa^*\nu \colon \g \to C^\infty(T^*N)$ is a Lie algebra map);
\item[(iii)]  a $\g$-module $\h$ and a $\g$-equivariant map $\mu\colon N \to \h^*$, with dual map $\mu^*:\h \to C^\infty(N)$.
\end{itemize}
 Let $\g\ltimes \h$ denote the semidirect product Lie algebra.
Then 
 \begin{equation}\label{eq:musemidirect}
 \tilde{\mu}\colon \g\ltimes \h\to C^{\infty}(T^*N), \;\;(u,w)\to  \big(\eta (u)+ \kappa^*(\nu(u))\big)+\kappa^*(\mu^*w)
 \end{equation}
is a Lie algebra map, defining a hamiltonian $\g \ltimes \h$-action on $T^*N$.
We now consider the analogous construction in the graded context.

Let $A[1]$ be an $\mathbb{N}$-manifold of degree 1 and $\cM=T^*[2]A[1]$. Assume that $A[1]$ is equipped with a $Q$-structure, see Remark \ref{rem:Qstructures}, so that $A\to M$ is a Lie algebroid. As seen in Example~\ref{ex:LAdouble}, the  Courant function on $\cM$ defined by $Q$ via  \eqref{eq:fiberlin} corresponds to the Courant algebroid $A\oplus A^*$, the double of $A$.

 Let $\g$ be a Lie algebra 
and consider the DGLA $\g[1]\oplus \g$ with differential given by the identity, and graded Lie bracket given by the one of $\g$ and its adjoint representation. 
For an element $u\in \g$, we write $u[1]$ for the same element in $\g[1]$ (with this notation, the differential is given by $u[1]\mapsto u$).
Any DGLA morphism $\g[1]\oplus \g\to \mathfrak{X}(A[1])$ (where the differential in $\mathfrak{X}(A[1])$ is $[Q,\cdot]$) is determined by 
its restriction $\eta\colon \g[1]\to \mathfrak{X}(A[1])_{-1}=\Gamma(A)$ via
$$
(w,u)\mapsto \eta(w)+[Q,\eta(u[1])].
$$
Since the natural embedding $\mathfrak{X}(A[1])\hookrightarrow C(\cM)[2]$ in  \eqref{eq:fiberlin} is a DGLA morphism, we conclude that  a Lie algebra map $\eta\colon \g\to \Gamma(A)$ yields a DGLA morphism  
$$
\overline{\eta}\colon \g[1]\oplus \g \to C(\cM)[2],
$$ 
thereby defining a hamiltonian $\g[1]\oplus \g$-action on $\cM$ (the ``cotangent lift'' of the given $\g[1]\oplus \g$-action on $A[1]$). 

For the analogue of (ii) above, we consider a map of complexes $\g[1]\oplus \g \to C(A[1])[2]$ (the latter is endowed with the differential $\pounds_Q$); as before, such a map is determined by its restriction $\nu\colon \g[1] \to C(A[1])_1=\Gamma(A^*)$ via
$$
(w,u) \mapsto \nu(w) + \pounds_Q \nu(u[1]).
$$
Since the embedding 
$\kappa^\sharp\colon C(A[1])\to C(\cM)$ is a  chain map   (recall that $\kappa^\sharp$ was defined in Example~\ref{ex:cotangent}), we see that any linear map $\nu\colon \g \to \Gamma(A^*)$ yields a map of complexes 
$$
\overline{\nu}\colon \g[1]\oplus \g \to C(\cM)[2].
$$
The analogue of the cocycle condition in (ii) (making $\overline{\eta} +\overline{\nu}\colon \g[1]\oplus \g \to C(\cM)[2]$ into a DGLA morphism) is
\begin{align}
 \label{eq:nueq1}
\nu([u,w]) &= \pounds_{[Q,\eta(u[1])]} \nu(w) - \pounds_{\eta(w)}(\pounds_Q\nu(u[1])),\\
  \label{eq:nueq2}
 \pounds_{\eta (w_1)} \nu (w_2)&=- \pounds_{\eta(w_2)} \nu(w_1),
 \end{align}
for all $w, w_1, w_2\in \g[1]$, and $u\in \g$.

For the analogue of (iii) above, let $\h$ be a $\g$-module. The  complex  $\h[2]\oplus \h[1]$, with the identity map as differential, then carries a natural $\g[1]\oplus \g$-module structure. A chain map $\h[2]\oplus \h[1]\to C(A[1])[2]$ is determined by its restriction to $\h[2]$, which is a linear map $\mu^*\colon \h \to C^\infty(M)$. We assume  that the chain map is $\g[1]\oplus \g$-equivariant, which amounts to the
$\g$-equivariance of  $\mu\colon M\to \h^*$, the map dual to $\mu^*$ (here $M$ carries the action given by $\rho_A\circ \eta$). The semidirect product  $\tilde{\g}=(\g[1]\oplus \g)\ltimes (\h[2]\oplus \h[1])$  can be checked to agree with the (exact) DGLA of Example~\ref{ex:hemisemi}, corresponding to the hemisemidirect product of $\g$ and $\h$. 
In analogy with \eqref{eq:musemidirect},  the following map can be shown to be  a  morphism of DGLAs: 
 \begin{align*}
\tilde{\mu}\colon \h[2]\oplus (\h[1]\oplus \g[1])\oplus \g
 &\to C(\cM)[2]\\ 
 (h_a,h_b,w,u) &\mapsto  (\overline{\eta}+\overline{\nu})(w,u) +  \kappa^\sharp(\mu^*h_a)+ \kappa^\sharp(\pounds_Q({\mu^*(h_b[1])})).
 \end{align*} 
 Hence the map $\tilde{\mu}$ satisfies \eqref{eq:dglacond} and
 yields a hamiltonian $\tilde{\g}$-action on $\cM=T^*[2]A[1]$ for which the Courant function $Q$ is reducible with respect to $\tilde{\mu}^{-1}(0)$ (see $\S$ \ref{sec:hamred}). It is a direct verification that \eqref{eq:nueq1} and \eqref{eq:nueq2} translate into equations (a) and (b) in (2) of Example~\ref{ex:reddatasuper}, and this hamiltonian $\tilde{\g}$-action on $T^*[2]A[1]$ is precisely the one corresponding to the hamiltonian action on the Courant algebroid $A\oplus A^*$ described in that example.
\hfill $\diamond$
 \end{remark}

%%%%%%%%%%%%%%%%%%%%%%%%%%%%%%%%%%%%%%%%%%%%%%%%%%%%%%
\subsubsection{\underline{Extended actions on exact Courant algebroids}}
We now explain how our general reduction setup relates to the framework of Courant reduction in \cite{bcg}, which is formulated in terms of 
 a suitable notion of action of an exact Courant algebra on an {\em exact} Courant algebroid. 
 
 Let $E$ be an exact Courant algebroid.
 Recall that, in this case, the anchor $\rho\colon E\to TM$ defines an inclusion 
$$
\rho^*\colon \Omega^1(M) \hookrightarrow \Gamma(E).
$$ 
The next definition is found in \cite[$\S$ 2.2 and 2.3]{bcg} and \cite[$\S$ 2.3 and 2.4]{HKquot}.

\begin{defi}\label{Def:extaction}
 An {\em extended action} of  an exact Courant algebra $ \mathfrak{a}\to
 \mathfrak{g}$ on an exact Courant algebroid $E\to M$ is given by bracket-preserving maps $\Psi$ and $\psi$  making the following diagram commute, 
 $$
 \xymatrix{
 \A \ar[d]^{\Psi}\ar[r] & \g \ar[d]^{{\psi}}\\
 \Gamma(E) \ar[r] &  \Gamma(TM), }
 $$
and such that 
$\Psi(\h)\subseteq \Omega^1_{cl}(M)$, where $\h=\mathrm{ker}(\A\to \g)$. A  \emph{momentum map} for the extended action is a  $\mathfrak{g}$-equivariant map $\mu \colon  M \rightarrow \h^*$  such that 
 $\Psi(h)=\rho^*(d (\mu^*h))$, for all $h\in \h$.
 \end{defi}

Note that, in the previous definition, $\psi$ is completely determined by $\Psi$ via $\psi(u) = \rho (\Psi(\hat{u}))$, where $\hat{u} \in \A$ is any lift of $u\in \g$ (this is well defined since $\rho(\Psi(\h))=0$). So an extended action with momentum map is simply given by 
\begin{itemize}
\item a bracket-preserving map $\Psi\colon \A\to \Gamma(E)$, and
\item a $\g$-equivariant map $\mu\colon M\to \h^*$ such that
$\Psi(h)=\rho^*(d (\mu^*h))$, for all $h\in \h$.
\end{itemize}

By comparison with Prop.~\ref{prop:exactDGLAact}, we immediately obtain

\begin{prop}\label{prop:bcg} Hamiltonian actions of an exact DGLA of degree 2, with corresponding Courant algebra $\A\to \g$, on an exact Courant algebroid $E$ are equivalent to extended actions $\Psi\colon \A\to \Gamma(E)$ with momentum map 
$\mu\colon M\to \h^*$ satisfying the additional property that
\begin{equation}\label{eq:extracond}
\mu^*(\Cour{a,a})=\frac{1}{2}\SP{\Psi(a),\Psi(a)}, \qquad \forall a\in \A.
\end{equation}
\end{prop}

With the notation of Prop.~\ref{prop:exactDGLAact}, we just set
$\Psi=\varrho$, noticing that the condition $\Cour{\Psi(\h),\cdot}=0$ is automatic since $\Psi(\h)\subseteq \rho^* d (C^\infty(M))$ and the Courant algebroid $E$ is assumed to be exact. 

\begin{remark}\label{rem:bcg}
For an extended action $\Psi\colon \A\to \Gamma(E)$ with momentum map 
$\mu\colon M\to \h^*$, let $K\subseteq E|_{\mu^{-1}(0)}$ be as in \eqref{eq:K}:
for each $x\in \mu^{-1}(0)$, $K|_x$ is the image of the map $a\mapsto \Psi(a)|_x$. Then condition \eqref{eq:extracond} admits an alternative formulation in terms of $K$ being isotropic. Indeed, it is immediate that \eqref{eq:extracond} implies that $K$ is isotropic. On the other hand, assuming that $K$ is isotropic, \eqref{eq:extracond} holds at all $x\in \mu^{-1}(0)$. But, by Remark \ref{rem:extracond}, part $(ii)$, and the fact that the anchor $\rho$ is surjective (since $E$ is exact), we have that 
$$
d(\mu^*(\Cour{a,a}))=\frac{1}{2} d\SP{\Psi(a),\Psi(a)}, \qquad \forall a\in \A,
$$
which implies that condition \eqref{eq:extracond} holds at every connected component of $M$ that contains a point in $\mu^{-1}(0)$. \hfill{$\diamond$}
\end{remark}

Although \eqref{eq:extracond} is not part of the original definition of momentum maps for extended actions in \cite[$\S$ 2.4]{HKquot}, the associated reduction procedure (see \cite[$\S$ 3.2]{HKquot}) includes the extra assumption that $K\subseteq E|_{\mu^{-1}(0)}$ be isotropic.  As explained in the previous remark, when considering Courant reduction at the value $0\in \h^*$, it makes no difference which of the two assumptions is made.

In conclusion, the Courant reductions for extended actions with momentum maps of \cite{bcg,HKquot}  are special cases of the setup in $\S$ \ref{subsec:hamCAred},  namely the cases in which both the DGLA and the Courant algebroid are exact. When restricted to this context, the reduction of Courant algebroids, Dirac and GC structures in Theorem~\ref{thm:geomhamred} 
coincide with the reduction schemes described in \cite[$\S$ 3 and $\S$ 4]{HKquot}, following \cite{bcg}. (In this special setup, note that  $\rho(K)$ agrees with the $\g$-orbits on $\mu^{-1}(0)$, and $\rho(K^\perp) = T\mu^{-1}(0)$, so reduced Courant algebroids are again exact, see Remark~\ref{rem:exactCA}.)

Back to the graded-geometric perspective, the reduction of exact Courant algebroids in \cite{HKquot} is nothing but usual symplectic reduction for cotangent bundles $T^*[2]T[1]M$, equipped with a suitable Courant function, with respect to a hamiltonian action of an exact DGLA of degree 2 (cf. Remark \ref{rem:short8.6}).

\begin{remark}\label{rem:nonisotropic}
More general reduction setups, not requiring that $K$ is isotropic, can be found in \cite{bcg}. For their graded geometric descriptions, one must consider symplectic reduction by submanifolds which are not necessarily coisotropic, such as arbitrary momentum levels in the case of hamiltonian actions. \hfill $\diamond$
\end{remark}

%%%%%%%%%%%%%%%%%%%%%%%%%%%%%%%%%%%%%%

\appendix
\section{Atiyah algebroids and quotients}\label{app:A}

We will consider two types of quotient constructions for (pseudo-euclidean) vector bundles and discuss their effects on the corresponding Atiyah algebroids.

\subsection{Pseudo-euclidean reduction}
Let $E\to N$ be a pseudo-euclidean vector bundle, and $K\to N$ be an isotropic subbundle. Then $E_{quot}:=K^\perp/K$ is also a pseudo-euclidean vector bundle over $N$. 
We will collect in this appendix some results relating the Atiyah algebroids $\mathbb{A}_E$ and $\mathbb{A}_{E_{quot}}$. (Comparing to  the setup and notation in the paper, the pseudo-euclidean vector bundle $E$ in this appendix is to be thought of as the restriction of a pseudo-euclidean vector bundle on a manifold $M$ to a submanifold $N\subseteq M$.)

We will also consider a lagrangian subbundle $L\subset E$. Assuming that $L\cap K$ has constant rank (which is equivalent to $L\cap K^\perp$ having constant rank), then $L_{quot} = (L\cap K^\perp + K)/K$ is a lagrangian subbundle of $E_{quot}$.

We will make use of the following facts from linear algebra.

\begin{lemma} \label{lem:linalgebra}
 Let $E\to N$ be a pseudo-euclidean vector bundle, and $K\to N$ be an isotropic subbundle.
\begin{itemize}
    \item[(i)]  There is an identification
    \begin{equation*}\label{eq:decomp}
    E \cong R \oplus (K \oplus K^*)
    \end{equation*}
    that preserves $K$, where the right-hand side is the direct sum of the pseudo-euclidean vector bundles $R$ and $(K\oplus K^*)$. 
    \item[(ii)] Given a lagrangian subbundle $L\subseteq E$ such that $L\cap K$ has constant rank, one can choose the identification in (i) such that $L= (L\cap R) \oplus (L\cap K) \oplus (L\cap K^*)$. 
\end{itemize}

\end{lemma}

\begin{proof}
One can always find a subbundle $T$ that is an isotropic complement for $K^\perp$ in $E$. Then $R:= (T\oplus K)^\perp \subset K^\perp$ is a complement for $K$ in $K^\perp$, $K^\perp=K \oplus R$, and the
pairing $\SP{\cdot,\cdot}$ on $E$ is non-degenerate on $R$. Hence
the pairing on $R^\perp= K\oplus T$ is also non-degenerate, and $K$ and $T$ are lagrangian subbundles therein, so there is a natural identification $R^\perp= K \oplus K^*$. It follows that $E= R\oplus R^\perp = R \oplus (K \oplus K^*)$, proving (i).

Note that an isotropic complement $T$ can be canonically constructed from any choice of complement $\tilde{T}$ of $K^{\perp}$ in $E$ following a procedure  analogous to  the one  described in \cite[Proposition 8.2]{AnaSymplectic} in symplectic linear algebra. Fix now a lagrangian subbundle $L$ in $E$, and take a subbundle $I$ such that $(L\cap K^{\perp})\oplus  I=L$. Then starting from a complement  $\tilde{T}$ of $K^{\perp}$ so that $I\subset \tilde{T} \subset I^{\perp}$,  we obtain from the aforementioned procedure an isotropic complement $T$ of $K^{\perp}$ in $E$  such that $L=(L\cap K^{\perp})\oplus(L\cap T)$. 
Moreover, for $R=(T\oplus K)^{\perp}$ (which, as we saw, satisfies $K^{\perp}=R\oplus K$),  we have that $L
\cap K^{\perp}=(L\cap R)\oplus (L\cap K)$, proving that (ii) holds. \end{proof}

In particular, in the decomposition of $E$  in (i) of the previous lemma, the factor $R$ is isomorphic to the quotient $E_{quot} = K^{\perp}/K$ as pseudo-euclidean bundles. Moreover, if (ii) holds, then $L\cap R$ is isomorphic to the lagrangian subbundle $L_{quot}$ in $E_{quot}$.

Consider the natural projection map $\Gamma_{K^\perp} \to \Gamma_{E_{quot}}$, denoted by $e\mapsto [e]$, and let $\Gamma^\st{K}_{\mathbb{A}_E}$ be the subsheaf of $\Gamma_{\mathbb{A}_E}$ given by metric-preserving derivations $(Y, D)$ of $E$ that satisfy $D(\Gamma_K)\subseteq \Gamma_K$ (and hence $D(\Gamma_{K^\perp})\subseteq \Gamma_{K^\perp}$ since $D$ preserves the metric).

\begin{prop}\label{prop:atiyahonto}
The map
\begin{equation}\label{eq:atiyahKmap}
\Gamma_{\mathbb{A}_E}^\st{K} \to \Gamma_{\mathbb{A}_{E_{quot}}}, \quad (Y, D) \mapsto (Y, [D]),
\end{equation}
where $[D]([e]) = [D e]$, for $e$ a local section of $\Gamma_{K^\perp}$, is onto.
%The map \eqref{eq:atiyahKmap} is onto.
\end{prop}

\begin{proof}
We use the decomposition of $E$ in Lemma~\ref{lem:linalgebra}, part (i), so that $R\cong E_{quot}$ and 
$$
\Gamma_{\mathbb{A}_{\Equot}}\cong \Gamma_{\mathbb{A}_{R}}.
$$
Fix $(Y,\Delta)\in \Gamma_{\mathbb{A}_{R}}(V)$, where $V\subseteq N$ is open. Let
$\nabla$ be any connection on the vector bundle $K\to N$, 
with dual connection $\nabla^*$ on $K^*\to N$.
Then $
\nabla_{Y} \colon \Gamma_K(V)\to \Gamma_K(V)$ and  $\nabla^*_Y :
\Gamma_{K^*}(V)\to \Gamma_{K^*}(V)$ are derivations of $K$ and $K^*$, respectively, and
$$
(Y, \nabla_Y\oplus \nabla^*_Y) \in \Gamma_{\mathbb{A}_{K\oplus K^*}}(V).
$$
Letting $D:=\Delta\oplus \nabla_Y \oplus \nabla^*_Y$, it follows that
$(Y,D)\in \Gamma^\st{K}_{\mathbb{A}_{E}}(V)$ is such that
$[D]=\Delta$.
\end{proof}

For an open subset $V\subseteq N$, note that $(Y,D) \in
\Gamma_{\mathbb{A}_E}^\st{K}(V)$ is in the kernel of \eqref{eq:atiyahKmap} if and only
if $Y=0$ and $D(\Gamma_{K^\perp}(V))\subseteq \Gamma_K(V)$. In
other words, %(cf. \eqref{eq:i*O}), 
the kernel of \eqref{eq:atiyahKmap} is 
$$
\{T \in \Gamma_{\wedge^2 E}(V)\;|\; 
%T(\Gamma_K(V))\subseteq\Gamma_K(V),\; 
T(\Gamma_{K^\perp}(V))\subseteq \Gamma_K(V) \} = \Gamma_{K\wedge E}(V).
$$
So we have an exact sequence
\begin{equation}\label{eq:exactatiyahquot}
0 \rightarrow  \Gamma_{K\wedge E} \rightarrow \Gamma_{\mathbb{A}_E}^\st{K} \rightarrow \Gamma_{\mathbb{A}_{E_{quot}}} \rightarrow 0.  
\end{equation}

Given a lagrangian subbundle $L\subset E$ such that $L\cap K$ has constant rank, let $\Gamma_{\mathbb{A}_E}^\st{K,L}$ be the subsheaf of $\Gamma_{\mathbb{A}_E}$ defined by metric-preserving derivations $(Y,D)$ of $E$ satisfying $D(\Gamma_K)\subseteq \Gamma_K$ and $D(\Gamma_L)\subseteq \Gamma_L$. Let $\Gamma_{\mathbb{A}_{E_{quot}}}^\st{L_{quot}}$ be the subsheaf of $\Gamma_{\mathbb{A}_{E_{quot}}}$ given by metric-preserving derivations $(Y, \Delta)$ of $E_{quot}$ such that $\Delta(\Gamma_{L_{quot}})\subseteq \Gamma_{L_{quot}}$. 

\begin{prop}\label{prop:Lonto}
The map \eqref{eq:atiyahKmap} restricts to a map
\begin{equation}\label{eq:atiyahKLmap}
    \Gamma_{\mathbb{A}_E}^{\st{K,L}} \to \Gamma_{\mathbb{A}_{E_{quot}}}^{\st{L_{quot}}}
\end{equation}
that is onto.
\end{prop}

\begin{proof}
We write $E$ as in Lemma~\ref{lem:linalgebra}, part (i), in such a way that $L$ can be written as in part (ii). 
The identification $R\cong E_{quot}$ is such that $L\cap R \cong L_{quot}$. Given an open subset $V \subseteq N$, take $(Y, \Delta) \in \Gamma_{\mathbb{A}_{E_{quot}}}^{L_{quot}}(V)$,  which can be seen as an element in $\Gamma_{\mathbb{A}_R}(V)$ preserving $\Gamma_{L\cap R}(V)$.
Let $\nabla$ be a connection on $K$ preserving the subbundle $L\cap K$. Then the dual connection $\nabla^*$ on $K^*$ automatically preserves $L\cap K^*$ (since $L\cap K^* \subseteq K^*$ coincides with the annihilator of $L\cap K\subseteq K$).
Letting $D=(\Delta\oplus \nabla_Y \oplus \nabla^*_Y)$, we have that $(Y,D) \in \Gamma^\st{K,L}_{\mathbb{A}_E}$ and $[D]=\Delta$.
\end{proof}

\subsection{Quotients of vector bundles}\label{subsec:atiyahquot}
We will consider the setup of $\S$ \ref{sec:coisoreduction}.
Let $A\to N$ be a vector bundle, let $F$ be a simple   distribution on $N$ with leaf space $\underline{N}$ and projection $p\colon N \to \underline{N}$. Suppose that $A$ is equipped with a flat $F$-connection $\nabla$ with trivial holonomy. As recalled in Lemma~\ref{lem:hol}, we have a quotient vector bundle $\underline{A} \to \underline{N}$, whose sheaf of sections $\Gamma_{\underline{A}}$ is naturally identified with $p_*\Gamma^\st{flat}_A$.  

We denote by $\mathbb{A}_A$ the Atiyah algebroid of $A$, i.e., the vector bundle whose sections are the derivations of $A$. In this subsection we explain how the Atiyah algebroids of $A$ and $\underline{A}$ are related, see Proposition \ref{prop:Ared}.

\begin{remark}
The whole discussion that follows carries over to the case where $A$ has an additional pseudo-euclidean structure, $\nabla$ is metric preserving (so that $\underline{A}$ is also pseudo-euclidean), and $\mathbb{A}_A$ and $\mathbb{A}_{\underline{A}}$ are defined
in terms of metric-preserving derivations. \hfill $\diamond$
\end{remark}

Recall that the $F$-connection $\nabla$ can be thought of as a vector-bundle map $\nabla\colon F \to \mathbb{A}_A$ such that $\sigma\circ \nabla = \mathrm{Id}_F$, where $\sigma\colon \mathbb{A}_A\to TN$ is the symbol map. Consider the vector bundle
$$
\mathbb{A}^\nabla_A\colon = \mathbb{A}_A/\nabla(F).
$$
Then $\nabla$ induces  a natural flat $F$-connection $\overline{\nabla}$ on $\mathbb{A}^\nabla_A\to N$ via 
\begin{equation}\label{eq:nablainduced}
\overline{\nabla}_Z(\overline{(Y,\Delta)})= \overline{([Z,Y],[\nabla_Z, \Delta])}, 
\end{equation}
where $Z$ is a section of $F$, and we denote the class of a section $(Y,\Delta)$ of $\mathbb{A}_A$  in $\mathbb{A}^\nabla_A$ by $\overline{(Y,\Delta)}$. The connection $\overline{\nabla}$ is well defined due to the flatness of $\nabla$. Note also that the symbol map induces a map $
\overline{\sigma}\colon {\mathbb{A}^\nabla_A} \to {TN}/{F}$ which relates the connection $\overline{\nabla}$ with the usual Bott connection $\nabla^{Bott}$ on $TN/F$ via $\overline{\sigma} \circ \overline{\nabla}_Z = \nabla^{Bott}_Z \circ \overline{\sigma}$.

\begin{lemma}\label{lem:flatYDelta}
A section $\overline{(Y,\Delta)}$ of $\Gamma_{{{\mathbb{A}^\nabla_A}}}$ is flat if and only if
\begin{itemize}
\item[(a)] $Y$ satisfies $[Y,\Gamma_F]\subseteq \Gamma_F$, and 
\item[(b)] $\Delta$ satisfies $\Delta(\Gamma^\st{flat}_{A})\subseteq \Gamma^\st{flat}_{A}$.
\end{itemize}
\end{lemma}

Note that the first condition in the lemma just says that $Y$ defines a $\nabla^{Bott}$-flat section of ${TN/F}$.

We will give an alternative description of the sheaf of flat sections of ${\mathbb{A}^\nabla_A}$.
Note that $\Gamma_{TN/F}^\st{flat}$ is a sheaf of $(C_N^\infty)_{bas}$-modules and that there is a natural action of $\Gamma_{TN/F}^\st{flat}$ on $(C_N^\infty)_{bas}$ by Lie derivatives.
We will consider another sheaf of $(C_N^\infty)_{bas}$-modules on $N$, denoted by $\mathcal{A}$. Its sections are given, on each open subset $V\subseteq N$, by pairs $(\xi, \delta)$, where $\xi \in \Gamma_{TN/F}^\st{flat}(V)$ and $\delta\colon \Gamma^\st{flat}_{A}(V)\to \Gamma^\st{flat}_{A}(V)$ satisfy
\begin{equation*}\label{eq:brk20}
\delta(f e) = \pounds_\xi(f) e + f\delta(e),
\end{equation*}
for $f\in (C_N^\infty)_{bas}(V)$ and $e\in \Gamma^\st{flat}_{A}(V)$.

\begin{lemma}\label{lem:flatcharac}
For each open $V\subseteq N$, the map
\begin{equation}\label{eq:A}
\Gamma^\st{flat}_{{\mathbb{A}^\nabla_A}}(V) \to \mathcal{A}(V), \quad \overline{(Y,\Delta) }\mapsto (\overline{Y}, \overline{\Delta}),
\end{equation}
where $\overline{Y}$ is the projection of $Y$ to $\Gamma_{TM/F}(V)$ and $\overline{\Delta}$ is the restriction of $\Delta$ to $\Gamma^\st{flat}_{A}(V)$, is an isomorphism.
\end{lemma}

\begin{proof}
The map \eqref{eq:A} is well defined by Lemma \ref{lem:flatYDelta}, and clearly injective.
 It remains to show that it is  surjective. For a given section $(\xi,\delta)\in  \mathcal{A}(V)$, one can pick $Y\in \Gamma_{TN}(V)$ such that $[Y, \Gamma_F(V)]\subseteq \Gamma_F(V)$ and $\xi=\overline{Y}$. Using the fact that $A$ admits local frames of $\nabla$-flat sections, one can directly check that the choice of $Y$ uniquely determines an extension of $\delta$ to an element $(Y,\Delta)$ in $\Gamma_{\mathbb{A}_{A}}(V)$. (This extension is constructed using local flat frames applying the Leibniz rule, and is well defined since the transition matrix of two flat frames consists of basic functions; uniquess is immediate.)
\end{proof}
 
Since $\underline{A}$ is the quotient of $A$ with respect to $F$ and $\nabla$, the pullback of sections gives us an isomorphism
\begin{equation}\label{eq:pbmap}
 p_1^\sharp \colon \Gamma_{\underline{A}}  \stackrel{\sim}{\to}  p_*
\Gamma_{A}^\st{flat}.
\end{equation}

\begin{prop}\label{prop:Ared}
There is a natural identification $$\Gamma_{\mathbb{A}_{\underline{A}}} \cong p_*\Gamma^\st{flat}_{{\mathbb{A}^\nabla_A}}$$ defined as follows: a section 
$(\underline{Y},\underline{\Delta})$ of $\Gamma_{\mathbb{A}_{\underline{A}}}$ corresponds to a section $\overline{(Y, \Delta)}$ of $p_*\Gamma^\st{flat}_{{\mathbb{A}^\nabla_A}}$ if and only if $\underline{Y}=p_*Y$ and $p_1^\sharp \circ \underline{\Delta} = \Delta\circ p_1^\sharp$.
\end{prop}

\begin{proof}
We have an identification
\begin{equation}\label{eq:projvf}
p_*\Gamma_{TN/F}^\st{flat}\stackrel{\sim}{\to} \Gamma_{T\Mred}
\end{equation} 
through the natural projection map, taking a section $\xi=\overline{Y}$ to $\underline{Y}= p_*(Y)$, that satisfies the relation
\begin{equation*}
p^*(\pounds_{\underline{Y}}f) = \pounds_Y (p^*f),    
\end{equation*}
for any section $f$ of $C^\infty_{\Mred}$. From the isomorphisms \eqref{eq:pbmap} and \eqref{eq:projvf} we obtain an isomorphism
$$
p_*\mathcal{A} \stackrel{\sim}{\to} \Gamma_{\mathbb{A}_{\underline{A}}},
$$
given on sections by $(\overline{Y},\delta)\mapsto (p_*(Y),\underline{\Delta})$, where $\underline{\Delta}$ and $\delta$ are related by $p_1^\sharp \circ \underline{\Delta} = \delta\circ p_1^\sharp$. Composing this last isomorphism with the one in Lemma~\ref{lem:flatcharac}, we obtain the identification in the statement.
\end{proof}

With the previous proposition, one can in fact verify that the quotient of ${\mathbb{A}^\nabla_A}$ with respect to $F$ and $\overline{\nabla}$ is $\mathbb{A}_{\underline{A}}$.

\bibliographystyle{halpha}
%\bibliography{bibio}

\def\polhk#1{\setbox0=\hbox{#1}{\ooalign{\hidewidth
  \lower1.5ex\hbox{`}\hidewidth\crcr\unhbox0}}} \def\cprime{$'$}

\end{document}